\documentclass{article}  
\usepackage[
	bookmarksnumbered,
	bookmarksopen,
	pagebackref,
	colorlinks,
	citecolor=blue,
	linkcolor=blue,
]{hyperref}
\usepackage{graphicx}%
\usepackage{multirow}%
\usepackage{amsmath,amssymb,amsfonts,amsthm}%
\usepackage{mathrsfs}%
\usepackage[title]{appendix}%
\usepackage{xcolor}%
\usepackage{textcomp}%
\usepackage{manyfoot}%
\usepackage{booktabs}%
\usepackage{algorithm}% 
\usepackage{algorithmicx}%
\usepackage{algpseudocode}%
\usepackage{listings}%
\usepackage{subcaption}
\usepackage{microtype}
\usepackage[inline]{enumitem}
\usepackage{relsize}

\newtheorem{theorem}{Theorem}[section]
\newtheorem{definition}{Definition}[section]
\newtheorem{proposition}{Proposition}[section]

\newtheorem{lemma}{Lemma}[section]
\newtheorem{corollary}{Corollary}[section]
\newtheorem{remark}{Remark}[section]
\newtheorem{example}{Example}[section]

\newcommand{\normSpectral}[1]{\ensuremath{\| #1\|_2}}
\newcommand{\normFSquare}[1]{\ensuremath{\| #1\|_F^2}}
\newcommand{\normF}[1]{\ensuremath{\| #1\|_F}}
\newcommand{\normA}[1]{\ensuremath{\| #1\|_a}}
\newcommand{\bi}{\mathbf{i}}
\newcommand{\bj}{\mathbf{j}}
\newcommand{\bk}{\mathbf{k}}

\newcommand{\bfi}{\mathbf{i}}
\newcommand{\bfj}{\mathbf{j}}
\newcommand{\bfk}{\mathbf{k}}

\newcommand{\qmat}[1]{\ensuremath{\mathbf{ #1}}}
\newcommand{\diag}{{\rm diag}} 
\newcommand{\range}[1]{\ensuremath{\mathcal{R}(#1)}}
\newcommand{\adjointJ}{ \mathcal{J} }
 \newcommand{\spanmat}[1]{ \ensuremath{{\rm span}(#1)}  }
 \newcommand{\rangemat}[1]{ \ensuremath{{\mathcal R}(#1)}  }

 \newcommand{\chiQ}[1]{ \ensuremath{ \mathlarger{\chi}_{ \mathbf{ #1} } }  }
 \newcommand{\chilarger}[1]{ \ensuremath{ \mathlarger{\chi}_{ #1 } }  }

 \newcommand{\cptQa}[1]{ \ensuremath{ \qmat{#1}_a   }  }
 \newcommand{\cptQc}[1]{ \ensuremath{ \qmat{#1}_c   }  }
 \newcommand{\cptQr}[1]{ \ensuremath{ \qmat{#1}_r   }  }
 \newcommand{\kappaq}[1]{ \ensuremath{ \kappa(\qmat{#1} )   }  }

 \newcommand{\bbQ}{ \mathbb{Q} }
 \newcommand{\bbQmn}{ \mathbb{Q}^{m\times n} }
 \newcommand{\bbQns}{ \mathbb{Q}^{n\times s} }
 \newcommand{\bbQms}{ \mathbb{Q}^{m\times s} }
 \newcommand{\bdOmega}{ \ensuremath{ \boldsymbol{\Omega}  } }
 \newcommand{\bdPsi}{ \ensuremath{ \boldsymbol{\Psi}  } }
 \newcommand{\bdPhi}{ \ensuremath{ \boldsymbol{\Phi}  } }

 \newcommand{\bbC}{ \mathbb{C} }

 \newcommand{\bigxiaokuohao}[1]{\ensuremath{ \left(  #1 \right) }}      
 \newcommand{\bigjueduizhi}[1]{\ensuremath{ \left|  #1 \right| }}   
 \newcommand{\bigdakuohao}[1]{\ensuremath{ \left\{  #1 \right\} }}         
 \newcommand{\bigzhongkuohao}[1]{\ensuremath{ \left[   #1 \right] }}

 \newcommand{\innerprodbig}[2]{\ensuremath{ \left\langle   #1 , #2\right\rangle }}      
 \newcommand{\innerprod}[2]{\ensuremath{  \langle   #1 , #2 \rangle }}

	\definecolor{darkgray}{rgb}{0.66, 0.66, 0.66}

\newenvironment{mytabular1}{\bgroup\footnotesize\tabular}{\endtabular\egroup}

\title{Randomized Large-Scale Quaternion Matrix Approximation: Practical Rangefinders and One-Pass  Algorithm}

\author{Chao Chang$^*$ \and Yuning Yang\thanks{College of Mathematics and Information Science, Guangxi University, Nanning, 530004, China} \thanks{Corresponding author: Yuning Yang, yyang@gxu.edu.cn.}
}
\begin{document} %\large
\maketitle

\begin{abstract}
	Recently, randomized algorithms for low-rank approximation of quaternion matrices have received increasing attention. However, for large-scale problems, existing quaternion orthonormalizations   are inefficient, leading to slow rangefinders. 	To address this, by appropriately leveraging  efficient scientific computing libraries in the complex arithmetic, this work  devises two practical quaternion rangefinders,    one of which is %allowed to be 
	non-orthonormal yet well-conditioned.  They are then integrated into the quaternion version of a    one-pass algorithm, which originally takes  orthonormal rangefinders only. We   establish the error bounds and demonstrate that the error is proportional to the condition number of the rangefinder. The probabilistic bounds are exhibited for both quaternion Gaussian and sub-Gaussian embeddings. Numerical experiments demonstrate that the one-pass algorithm with the proposed rangefinders significantly outperforms previous techniques in efficiency. Additionally,   we tested the   algorithm in a 3D Navier-Stokes equation ($5.22$GB) and a 4D Lorenz-type chaotic system ($5.74$GB)  data compression, as well as a $31365\times 27125$ image compression  to demonstrate its capability for handling large-scale   applications.

\end{abstract}

\section{Introduction}\label{sec:Introduction}
	
	\subsection{Background}
	Low-rank matrix approximation (LRMA) has been applied in various applications.
	% , including dimension reduction,   compression, classification, regression, clustering, and so on. 
	In the big data era, large amounts of   data are being captured and generated through various channels, such as high-definition   color video, scientific simulations, and artificial intelligence   training sets. This trend poses challenges in terms of computation time, storage, and memory costs to LRMA. In 2011, a randomized SVD algorithm (HMT)   \cite{FindingStructureHalko} was proposed by Halko, Martinsson, and Tropp, which uses a random sketch to obtain an oversampling approximation before implementing the truncated SVD. Compared to the deterministic SVD, the randomized one runs faster with adjustable precision loss. It is robust, but there is still a need to revisit the original data during the low-rank approximation. In 2017, Tropp et al. \cite{Practical_Sketching_Algorithms_Tropp} developed a one-pass randomized   algorithm using two sketches, which needs to visit the data only once.  It is more effective for managing data with limited storage, arithmetic, and communication capabilities.  Later on, the authors devised also a one-pass algorithm using three sketches and applied it to streaming data \cite{troppStreamingLowRankMatrix2019}. Prior to these work, randomized algorithms have been studied extensively in the literature; see, e.g., \cite{FastMontecarlo,woolfe2008Fast,clarkson2009NumericalLinear,mahoney2011RandomizedAlgorithms,woodruff2014SketchingTool,cohen2015dimensionalityreduction,boutsidis2016OptimalPrincipal}, and the recent surveys \cite{tropp2023RandomizedAlgorithms,kannan2017RandomizedAlgorithms,kireeva2024RandomizedMatrix,murray2023RandomizedNumerical,martinsson2020RandomizedNumerical}.

	 Despite the  noncummutativity in quaternion multiplications,   quaternion matrices have been widely used in various applications such as signal processing \cite{ellQuaternionFourierTransforms2014}, color image analysis \cite{miao2023quaternion,soo-changpeiQuaternionMatrixSingular2003}, and machine learning \cite{zhangAugmentedQuaternionExtreme2019,minemotoFeedForwardNeural2017} in recent years.  %For example, a color image can be represented by a pure quaternion matrix, and the relation between the RGB channels can be explored by the quaternion computations.  However, the noncommutativity of quaternion multiplications sometimes brings troubles in algorithm design and theoretical analysis \cite{zhangQuaternionsMatricesQuaternions1997}.
	Randomized quaternion low-rank matrix approximation has garnered increasing attention very recently. 
	Liu et al. \cite{RandomizedQSVD} developed a randomized quaternion SVD algorithm based on the HMT framework by using structure-preserving quaternion QR and   quaternion SVD,  and studied its error bound. This algorithm was later applied to nonnegative pure quaternion matrix approximation \cite{lyuRandomizedLowRankNonnegativePureQuaternionMatrices2024}.
	Ren et al. \cite{renRandomizedQuaternionQLP2022} proposed a randomized quaternion QLP decomposition algorithm. % where they considered both matrix approximation error analysis and singular approximation error analysis. 
	Li et al.    \cite{liRandomizedBlockKrylov2023}  also proposed a randomized block Krylov subspace algorithm with improved approximation accuracy. Very recently, a fixed-precesion randomized quaternion SVD was studied in \cite{liu2024FixedprecisionRandomized} and a randomized quaternion UTV decomposition has been proposed in \cite{xu2024RandomizedQuaternion}.

	The framework of the HMT algorithm \cite{FindingStructureHalko} (also \cite{Practical_Sketching_Algorithms_Tropp,troppStreamingLowRankMatrix2019}, and the quaternion randomized algorithms \cite{RandomizedQSVD,renRandomizedQuaternionQLP2022,liRandomizedBlockKrylov2023})  can   be divided into a randomized QB approximation stage and a truncation stage. Initially,   the QB stage involves generating a sketch of the input data matrix  through randomized oversampling with Gaussian or alternative embeddings. This then leads to the formation of the $Q$ matrix, representing an orthonormal basis for the range of the sketch, known as the \emph{rangefinder} step. The $B$ matrix is then determined, either exactly \cite{FindingStructureHalko} or approximately \cite{Practical_Sketching_Algorithms_Tropp,troppStreamingLowRankMatrix2019}.
	% a rangefinder is constructed using Gaussian embedding or other methods to project the data matrix into a suitable low-dimensional subspace.
  In the truncation stage, truncated SVD or other deterministic methods  are performed on the $B$ matrix to find a more accurate fixed-rank approximation. %\color{black} This framework works efficiently in real or complex cases, thanks to contributions from BLAS, MKL, and other linear algebra libraries. However, in quaternion cases, the matrix tools are not fast enough for us to deal with real-time processing.

	\subsection{Quaternion rangefinders}
	
	% The aforementioned quaternion randomized algorithms  follows the framework of HMT and its advancements (an original ref. here).  %In the first stage, a rangefinder is constructed   using Gaussian embedding or other methods to project the data matrix onto a suitable low-dimensional subspace.  This step is cheap in MATLAB by using QR decomposition. 
	% In this framework, the first stage is to construct a rangefinder to find an orthonormal basis of the sketch. 
	  In the real/complex case,  constructing an orthonormal rangefinder is cheap. Fast and stable algorithms for orthogonalization have matured, and highly optimized implementations are available, such as   MATLAB's built-in function \texttt{qr}, and various QR routines in LAPACK \cite{lapack} and the Intel Math Kernel Library (MKL) \cite{IntelMKL}.

	In the quaternion case, orthogonalization approaches are   developing. Classical QR decomposition methods can be extended to the quaternion arithmetic with little modifications. For example, the quaternionic Householder QR   was proposed for quaternion eigenvalue computations \cite{bunse-gerstnerQuaternionQRAlgorithm1989}; this was implemented by Sangwine and   Le Bihan in the quaternion toolbox for MATLAB (QTFM) \cite{sangwine2020qtfm} with the function name \texttt{qr}. To speed up quaternion matrix computations, in a series of papers \cite{jiaNewStructurepreservingMethod2013,li2016RealStructurepreserving,jiaNewRealStructurepreserving2018,chenNewStructurepreservingQuaternion2021}, the authors proposed structure-preserving algorithms. % by utilizing the special structure of the real representation of the quaternion matrices.  The  structure-preserving 
  This type of algorithms is promising, as its basic idea is to only operate on the   real representation of a quaternion matrix,     avoiding quaternion operations and smartly reducing computational complexity. 
	The structure-preserving quaternion Householder QR (QHQR)  \cite{jiaNewRealStructurepreserving2018,li2016RealStructurepreserving} is numerically stable and accurate; the structure-preserving  quaternionic modified Gram-Schimit (QMGS) \cite{wei2018quaternion} is more economic but may lose accuracy. These methods together with QTFM's \texttt{qr} function  were employed for orthonormal rangefinders by the quaternion randomized algorithms \cite{RandomizedQSVD,liRandomizedBlockKrylov2023,renRandomizedQuaternionQLP2022}.  

 \subsection{Limitation and motivation}

 \begin{figure}
	\centering
	\begin{subfigure}[b]{0.5\textwidth}
		\includegraphics[width=\linewidth]{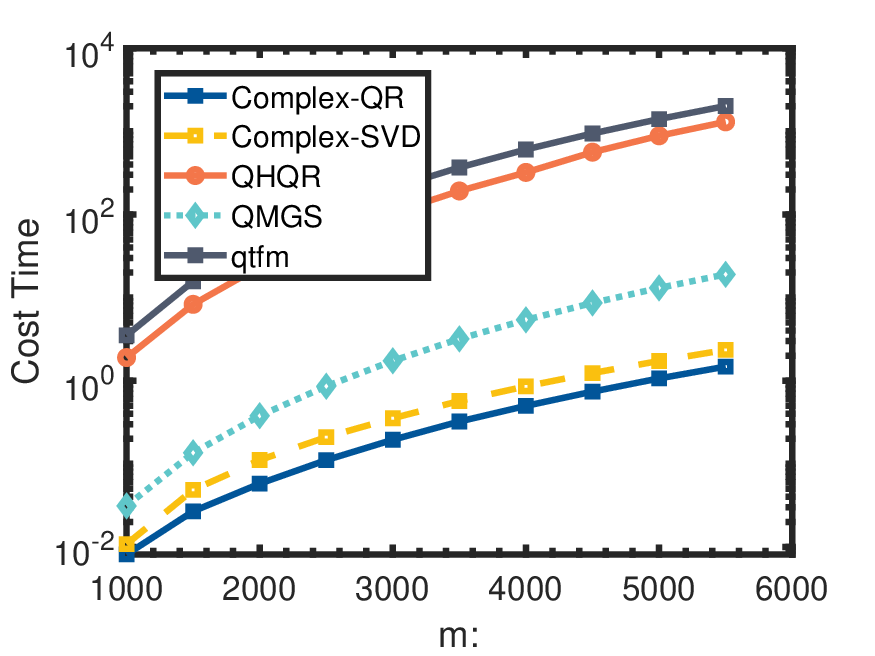}
	\end{subfigure}
	\caption{\small Running time comparisons of QTFM's \texttt{qr}, QHQR, QMGS, and MATLAB's \texttt{qr} and \texttt{svd} (the latter two are applied to the full complex representation; see Sect. \ref{sec:comlex_representation} for its definition). $m$ from $1000$ to $5500$, $n=0.1m$.}
	 \label{fig:ComplexCompare}
\end{figure}

	The above quaternionic orthonormalization methods are efficient in small and moderate problems;  
	however,  for large-scale   matrices, they are still expensive.   For instance,  even QMGS, which is the fastest algorithm mentioned previously\footnote{The code of QHQR was downloaded from Jia's homepage \url{http://maths.jsnu.edu.cn/_upload/article/files/40/5c/0abecd234d2c909be8b4fac9c4ad/1f0499ce-65d1-4364-9000-b8998137e516.zip} and that of QMGS was implemented from \cite{wei2018quaternion}. 
	QMGS and QHQR both have complexity $O(mn^2)$ for matrices of size $m\times n$ \cite{wei2018quaternion,jiaNewRealStructurepreserving2018,li2016RealStructurepreserving,golubMatrixComputations2013}. If the full orthonormal matrix $Q$ is computed, as was done in the implementation of QHQR, then its complexity becomes $O(m^2n)$; see also \cite[p. 249]{golubMatrixComputations2013}. This is expensive when $m\gg n$, as is usually the size of rangefinders.  Nevertheless, empirically  we find that even without explicitly forming $Q$, QHQR is still about twice slower than QMGS, which also confirms the discussions in \cite[p. 255]{golubMatrixComputations2013}.},   
	 is  at least an   order  of magnitude slower than   MATLAB's built in functions \texttt{qr} and \texttt{svd} for $5000\times 500$ quaternion matrices; see Fig. \ref{fig:ComplexCompare}\footnote{ MATLAB's \texttt{qr} and \texttt{svd} are applied to the full complex   representation of the quaternion matrix, which is of size $2m\times 2n$. Of course, this is a rough comparison, as complex QR and SVD may not directly generate an orthonormal rangefinder in the quaternion domain; nevertheless, the purpose of this   comparison is to show that the speed of quaternion rangefinders still has a large room to improve.}.  One of the possible reasons is that highly optimized implementations of these quaternion orthogonalization methods are far from developed. For example, MATLAB's built-in functions such as \texttt{qr} and \texttt{svd} do not support quaternions yet; nor LAPACK or MKL. Without an elegant implementation, even an advanced method cannot take advantage of the features such as parallelism and memory management in modern computing architecture, holding back its advancement.

	% \iffalse
	% Randomized embedding also plays a core role in constructing a proper rangefinder. Research on competing distributions, such as Rademacher, Subsampled Randomized Fourier Transform, and Ultrasparse Rademacher, relies on their special structures to ensure their theoretical error bounds. Besides the widely used random embeddings mentioned above, tall matrices can also be applied in randomized linear embeddings. Bai and Yin provide an asymptotic result for random matrices with independent entries\cite{baiLimitSmallestEigenvalue1993}. Vershynin et al. provided a non-asymptotic analysis in their work cited in \cite{vershynin_2012_Introduction_non-asymptotic}. In particular, the sub-Gaussian distribution, which contains numerous useful distributions, has become increasingly popular in recent years. With finite sub-Gaussian moments, many derivative inequalities can be derived, and its numerous properties will aid in our analysis.
	% \fi

% \color{red} THE FIRST PRRAGRAPH TOO LONG, MAY BE REMOVED, MANY REDUDENT EXPRESSIONS 
% \color{black}
%	In the task of handling large-scale and streaming data, one may expect that an algorithm can be computed at any instant \cite{Practical_Sketching_Algorithms_Tropp}. %Thus, the efficiency is important.  In this context,  for a quaternion randomized algorithm,   the rangefinder stage  should not occupy  too much time. How to  make rangefinder faster?

\color{black}  A common design philosophy to gain efficiency is to trade accuracy, akin to randomized algorithms. Following this vein,   it is possible to     consider a non-orthonormal   rangefinder. If so,  what criteria should the new rangefinder meet, and how to quantify the loss of orthonormality?
%and how the criteria influence the approximation accuracy? \color{black}
	
	%two  alternative ways can be considered: Firstly, what if one relaxes  the orthogonlization requirement to a certain extent in the rangefinder step? secondly, to implement the rangefinder, one may employ the existing matured scientific computing libraries as much as one can. More specifically, we raise the following questions:

	%To apply the Randomized sketching framework to large-scale quaternionic matrices, three difficulties need to be overcome. The first issue is that the quaternion orthogonalization operators cost most of the time. The second issue is that the structure within the data may not be utilized; for example, the sparse structure is often present in scientific simulation data.
	
	%To apply the Randomized sketching framework to large-scale quaternionic matrices, three difficulties need to be overcome. The first issue is that the quaternion orthogonalization operators cost most of the time. The second issue is that the structure within the data may not be utilized; for example, the sparse structure is often present in scientific simulation data.
	
	%If we carefully examine the randomized quaternion low-rank approximation framework, it is natural to raise some questions:

	Within the HMT and RQSVD frameworks, the orthonormal rangefinder plays two roles. One is to improve the numerical stability in matrix decompositions \cite[Section 4.1]{FindingStructureHalko}. 
	% which can be compensated by stabilization technology, as shown in \cite{nakatsukasaFastStableRandomized2020}. 
	The other is to  explode the singular values and their corresponding singular vectors in the truncation stage, which helps in better approximating the original matrix.  To  quantify the accuracy, it is natural to take the condition number of the rangefinder into account. Obviously, an orthonormal rangefinder gives the optimal condition number, while  deviations from orthonormality should maintain a condition number that does not indicate severe ill-conditioning.

%	 \color{blue} It controls the truncated part in the second stage and guarantees the approximate singular vectors.   Thus, if it is necessary to leverage the information of singular vectors in low-rank decomposition, an approximate orthogonality is essential at least, which is equivalent to having a small condition number.\color{red}
%	Orthogonal matrix as basis of projection space preserves the main singular values and their corresponding singular vectors, which ensure maintaining main singular vectors in truncated SVD. When random algorithms are to be applied to feature extraction, the information of the main singular vectors must also be contained, and an approximate orthogonality is essential at least, which is equivalent to having a small condition number. Users can choose the degree of the condition number based on their own needs.
	\color{black}

%  \color{blue}
	When the criterion above is met, it is expected to appropriately utilizing     mature  	scientific computing libraries or advanced algorithms in the real/complex arithmetic  to accelerate our computations. % 
	\color{black}
	% When the criterion above is met, it is expected that the   rangefinder is built upon   mature  	scientific computing libraries or advanced algorithms in the real/complex arithmetic  to fully utilize the modern computing architecture. % In addition, these libraries supplying sparse matrix compuation may allow  one to   handle data and test matrices with sparse structure. A common application is sparse sign matrices \cite[Section 3.9]{Practical_Sketching_Algorithms_Tropp} in randomized sketching algorithms  \cite{troppStreamingLowRankMatrix2019,troppRandomizedAlgorithmsLowrank2023}, which helps control storage and arithmetic costs. 

	%	Whenever we aim to make the fast rangefinder more practical, designing the rangefinder based on highly optimized built-in functions in MATLAB or other orthogonalization algorithms available in scientific computing libraries is first choice. Furthermore, it can also utilize the sparse structure supplication of MATLAB, which allows us to handle data and test matrices with sparse structures. 
	
%\color{red} REDUDENT PARAGRAPH.	We hope that the impact of our compromise on the accuracy of the algorithm is controllable. Fortunately, the fixed-rank approximation error only depends on the condition number of the rangefinder, and singular vector perturbation analysis is presented to provide a theoretical guarantee.
	\color{black}
	
%			\begin{itemize}
		%\item Which role does the orthogonalization operator act in the framework? 
		%		\item Is it possible to relax the orthogonalization requirement in the rangefinder stage?
		%		\item If we try to use a cheaper approximate orthogonal operation instead, what conditions does it need to meet? 
%T		\item How can we make the most of existing scientific computing libraries, including their support for sparse structures?
	%	\item How to evaluate the loss of accuracy due to our concession?	
%	\end{itemize}

	\subsection{This work}
The general idea throughout this work is to transform heavy quaternion computations   to  QR, SVD, and solving linear equations in the complex arithmetic.
% , such that mature scientific computing libraries or advanced algorithms   can be employed to accelerate the computations. 
To this end, we employ compact or full complex representations of quaternion matrices as intermediaries, accepting potential reductions in orthonormality  while maintaining favorable condition numbers. The influence of condition numbers on approximation accuracy will be examined theoretically.
Specifically:
 
% The high-level idea throughout this work is to mostly employ established   scientific computing libraries or advanced algorithms in the real/complex arithmetic  for heavy computations such as QR, SVD,   and solving linear equations. The latter two, while costly in the usual understanding, are still   more efficient than some existing quaternion computations. 
%We employ compact or full complex representations of quaternion matrices as intermediaries, accepting potential reductions in orthonormality and accuracy to maintain favorable condition numbers. The influence of condition numbers on approximation accuracy will be examined in our theoretical analysis.
%Specifically:

In Section \ref{sec:quaternionFastRangeFinder}, we    provide two  practical range-preserving  rangefinders, termed pseudo-QR and pseudo-SVD. Theoretically,  Pseudo-QR  can reduce the condition number of the sketch from $10^8$ within $10$. Pseudo-SVD on the other hand generates an orthonormal matrix even with very ill-conditioned sketch. Almost all the computations are built upon 
established scientific computing libraries in the complex arithmetic to ensure  their efficiency.  Comparisons with previous techniques in terms of time and accuracy are illustrated in Fig. \ref{fig:CompareTime} and \ref{fig:ComparePrecision}.

In Section \ref{sec:one-pass-alg-error-anal}, the proposed rangefinders are   incorporated  into   the quaternion version of the one-pass framework of \cite{Practical_Sketching_Algorithms_Tropp}. Originally, the algorithm only employs orthonormal rangefinders. 
 Our theoretical findings in Sections \ref{subsec:ErrorAnalysis}, \ref{sec:Gaussian}, and \ref{sec:sub-gaussian} are summarized as follows:
\begin{itemize}
    \item In the QB approximation stage,  the approximation error,  measured by the tail energy, is independent of the condition number of the rangefinder;
    \item  In the truncation stage,  the truncation error is proportional to the condition number of the rangefinder. This and the previous   point   ensure the reasonability of using a non-orthonormal yet well-conditioned rangefinder in theory.
    \item % We derive   a deviation bound  for   extreme singular values of a quaternion sub-Gaussian matrix, which may be of independent interest. This theoretically justifies the use of %not only quaternion Gaussian  but also
	% quaternion sub-Gaussian test matrices.
	The probabilistic bounds are exhibited for either quaternion Gaussian or sub-Gaussian test matrices. 
\end{itemize}

% Some comments are in order. For the first two points, it seems that it is rare to study the influence of the condition number of the rangefinder to the approximation error in the literature, as computing an orthonormal rangefinder is cheap in the real/complex case; there is no need to use a non-orthonormal one. For the third point, 

Some comments are in order. 
\begin{itemize}
	\item The result  of the first two points   applies to any  range-preserving while non-orthonormal rangefinders.
	\item The probabilistic bounds rely on  the non-asymptotic deviation bounds of extreme singular values of a quaternion random matrix. The  quaternion  Gaussian case was first studied in \cite{RandomizedQSVD}, while we    generalize    a   result   in the real case  by Vershynin \cite{vershynin_2012_Introduction_non-asymptotic} to quaternions. %Specifically,   with high probability, the singular values of a tall $m\times n$ quaternion sub-Gaussian matrix are shown to  lie in the interval $[2\sqrt m-2C\sqrt n - t,2\sqrt m + 2C\sqrt n +t]$ for some   $C,t>0$.  
\end{itemize}

Finally,  Section \ref{sec:numerical} evaluated the performance of our algorithm and previous ones through various experiments using synthetic data. 
% On the other hand, existing quaternion randomized algorithms  
To demonstrate the applicability of our approach in large-scale problems, we tested the   algorithm in a 3D Navier-Stokes equation ($5.22$GB) and a 4D Lorenz-type chaotic system ($5.74$GB)  data compression, as well as a $31365\times 27125$ color image compression. 
% \color{blue} For the two equations, the running time respectively takes   less than $20$ and $300$ seconds, while that for the image varies from $370$ to $1675$ seconds, depending on the target rank. 
\color{black}Note that existing quaternion randomized algorithms \cite{RandomizedQSVD,renRandomizedQuaternionQLP2022,liRandomizedBlockKrylov2023,liu2024FixedprecisionRandomized,xu2024RandomizedQuaternion} were only reported to handled matrices of size up to a few thousand. \color{black}

\color{black} Our implementation in MATLAB is available at \url{github.com/Mitchell-Cxyk/RQLRMA}.

	\color{black}

	\section{Preliminaries on Quaternions}\label{sec:Preliminaries}
	Throughout this work, quaternion  vectors  and matrices are written in bold face characters; quaternion scalars and real/complex scalars, vectors, and matrices are written as italic characters.

	Quaternions  were invented by Sir William Rowan Hamilton in 1843. A quaternion scalar is   of the form 
	$ {a} = a_w + a_x \bi + a_y \bj + a_z \bk $
	where \(a_w, a_x, a_y, a_z\) are real numbers.
	The sum of quaternions is defined component-wise 
	and the their multiplication   is determined by the following rules along with the associative and distributive laws $\mathbf{i}^2=\mathbf{j}^2=\mathbf{k}^2=\bfi\bfj\bfk=-1$. For $a,b\in\bbQ$, $ab\neq ba$ in general. 
	 A  quaternion matrix $\qmat{Q}\in \bbQmn$    is defined   as $\qmat{Q}=Q_w+Q_x\bi+Q_y\bj +Q_z\bk $, where $Q_i (i=w,x,y,z) \in\mathbb R^{m\times n}$.
	The conjugate and  conjugate transpose of $\qmat{Q}$ are respectively denoted as $\overline{\qmat{Q}}:= Q_w-Q_x\bi - Q_y\bj - Q_z\bk$ and 
	 $\qmat{Q}^*:=Q_w^T-Q_x^T\bi-Q_y^T\bj-Q_z^T \bk $. For two quaternion matrices $\qmat{P}$ and $\qmat{Q}$ of proper size, $ \qmat{P}^*\qmat{Q} = (\qmat{Q}^*\qmat{P})^*$ while $\overline{\qmat{P}^*\qmat{Q}}\neq\overline{\qmat{P}^*}\cdot \overline{\qmat{Q}}$. More   properties are referred to \cite{zhangQuaternionsMatricesQuaternions1997}. $\|{\qmat{Q}}\|_F:=(\sum_{k\in \{w,x,y,z \}}\|Q_k\|_F^2)^{1/2}$.

	\subsection{Quaternion vector space}
	% Quaternion inner product play an important role in our theoretical analysis and geometry intuition. This part can trace to \cite{zhangQuaternionsMatricesQuaternions1997}. Here we give some fundamental definitions and properties of quaternion inner product space. 
	
	 Considering vectors with quaternion coordinates, a module over the ring $\mathbb{Q}$ is usually called the quaternion right vector space under the summation and the right scalar multiplication.  Given quaternion vectors $\qmat{v}_1,\ldots,\qmat{v}_r$, they are right linearly independent if  for quaternions $k_1,\ldots,k_r$,
	\begin{align*}
		\qmat{v}_1 k_1+ \qmat{v}_2 k_2+\cdots+\qmat{v}_r k_r =0 \quad\text{implies}\quad k_i=0,~ ~i=1,\ldots,r.
	\end{align*}
	Most linear algebra concepts and results can be transplanted to the right vector space in parallel \cite{zhangQuaternionsMatricesQuaternions1997} and throughout this work, we always omit the prefix ``right''. 

	Given a quaternion matrix $\qmat{V}=[\qmat{v}_1,\ldots,\qmat{v}_r]$, its range space is defined as
	\[
	\rangemat{\qmat{V}} = \spanmat{\qmat{v}_1,\ldots,\qmat{v}_r} = \bigdakuohao{\sum^r_{i=1}\nolimits\qmat{v}_ik_i~\mid~ k_i\in\bbQ, ~i=1,\ldots,r  }.	
	\]
	The inner product between $\qmat{u}$ and $\qmat{v}$
is given by	 $\langle\mathbf{u} ,\mathbf{v}\rangle=\mathbf{u}^*\mathbf{v}$ \cite{jia2019RobustQuaternion,chen2022colorImage}. 
	They are orthogonal if $\langle\mathbf{u},\mathbf{v}\rangle=0$ and written as $\qmat{u}\perp\qmat{v}$, and orthonormal if in addition $\langle\qmat{u},\qmat{u}\rangle=\langle\qmat{v},\qmat{v}\rangle=1$. Given  a quaternion linear space $\mathcal{V}$, for a subspace $\mathcal{L}\subset\mathcal{V}$,   its orthogonal complement is defined as
	 $
		\mathcal{L}^\bot:=\{\qmat{v} \in\mathcal{V}: \langle \qmat{v},\qmat{w}\rangle=0,\forall \qmat{w}\in\mathcal{L}\}
	 $. 
	\subsection{Complex representation}\label{sec:comlex_representation}
	% Denote by $\mathbb{Q}$ the algebra of Quaternions
	% $$q=a+b\mathbf{i}+c\bj+d\bk, \qquad where\qquad a,b,c,d\in\mathbb{R}$$
	% with the norm
	% $$|q|=\sqrt{q\bar{q}}=\sqrt{a^2+b^2+c^2+d^2}$$
	% The $\bar{q}=a-b\bi+c\bj+d\bk$ is the conjugate of $q$. And multiplication of quaternions is associative but not commutative in general. It is defined as follows:
	% $$\bi^2=\bj^2=\bk^2=-1\quad and \quad \bi\bj=-\bj\bi=\bk,\quad \bj\bk=-\bk\bj=bi,\quad \bk\bi=-\bi\bk=\bj$$
	% With a neutral element $1\in\mathbb{Q}$ and distributive property, hence $\mathbb{Q}$ forms a ring, which is usually called a skew field or noncommutative field.

	% With the ordinary addition and multiplication, there are some important notation of quaternion matrices.
	
	% the multiplication $\mathbf{QS}$ can be also represented as:
	% $$\begin{array}{c}
		%     (Q_0S_0-Q_1S_1-Q_2S_2-Q_3S_3)+(Q_0S_1+Q_1S_0+Q_2S_3-Q_3S_2)\mathbf{i}+\\(Q_0S_2-Q_1S_3+Q_2S_0+Q_3S_1)\mathcal{J}+(Q_0S_3+Q_1S_2-Q_2S_1+Q_3S_0)\mathbf{k}.
		% \end{array}$$

	  $\qmat{Q}\in\bbQmn$ can   be represented as $\qmat{Q}=Q_0+Q_1\bj$, where $Q_0,Q_1\in\mathbb{C}^{m\times n}$ 
	with $Q_0 = Q_w + Q_x\bi$ and $Q_1 = Q_y+Q_z\bi$. 
	The (full) complex representation of $\qmat{Q}$ is defined as \cite{zhangQuaternionsMatricesQuaternions1997}:
	\begin{gather*}
		\chiQ{Q}:=\begin{bmatrix}
			Q_0 & Q_1\\
			-\overline{Q_1} &\overline{Q_0} 
		\end{bmatrix} \in \bbC^{2m\times 2n}.
	\end{gather*}
	$\chiQ{Q}$ has several nice properties that it is useful in the study of quaternions:
	\begin{proposition}[\cite{zhangQuaternionsMatricesQuaternions1997}]\label{prop:Chi_properties}
		 Let $\qmat{P},\qmat{Q}$ be quaternion matrices of proper size. Then
	\begin{gather*}
		\mathlarger{\chi}_{k_1\mathbf{P}+k_2\mathbf{Q}}=k_1\chiQ{P}+k_2\chiQ{Q}(k_1,k_2\in\mathbb{R}), \quad\mathlarger{\chi}_{\qmat{P}\qmat{Q}}=\chiQ{P}\chiQ{Q},\\
		\mathlarger{\chi}_{\qmat{Q}^*}=\chiQ{Q}^*,\quad \chilarger{ \qmat{Q}^{-1} } = \chiQ{Q}^{-1} ~{\rm if}~ \qmat{Q}^{-1}~{\rm exists}, \\ 
		\chiQ{Q} ~{\rm is~ (partially)~unitary/Hermitian ~if ~and~ only~ if}~\qmat{Q}~{\rm is~(partially)~unitary/Hemitian}.
	\end{gather*}
\end{proposition}

	 $\chiQ{Q}$  can be partitioned as two blocks:
	\begin{align*} \label{def:Q_c_Q_a}
\chiQ{Q} = [ \cptQc{Q},\cptQa{Q}  ],~{\rm with}~\cptQc{Q}:=\begin{bmatrix}
	Q_0  \\
	-\overline{Q_1}   
\end{bmatrix},~\cptQa{Q}:= \begin{bmatrix}
	 Q_1\\
	 \overline{Q_0} 
\end{bmatrix}.
	\end{align*}
	We  call $\cptQc{Q}$ the \emph{compact} complex representation of $\qmat{Q}$.   $\cptQa{Q}$ can be generated from $\cptQc{Q}$  as $\cptQa{Q} = \mathcal J\overline{\cptQc{Q}}$, where 
	  $\mathcal{J}:=\begin{bmatrix}
	0&-I_m\\
	I_m&0
\end{bmatrix}$ is the symplectic matrix.  The relation between $\cptQc{Q}$ and $\cptQa{Q}$ is important in the design and analysis of our rangefinders. One can directly check that $\mathcal J$ admits the following properties:
			\begin{lemma} \label{lem:J-adjoint}
 J-adjoint satisfies:
	\begin{gather*}
		    \mathcal{J}^*=\mathcal{J}^{-1}=-\mathcal{J}; \quad 
		   \chiQ{Q}=[\cptQc{Q},\mathcal{J}\overline{\cptQc{Q}}];  \quad
		   \mathcal{J}^*\chiQ{Q}\mathcal{J}=\mathcal{J}\chiQ{Q}\mathcal{J}^*=\overline{\chiQ{Q}} ;\\
		   \mathcal{J}^*[v,\adjointJ \overline{v}  ]\mathcal{J}=\mathcal{J}[v,\adjointJ\overline{v}]\mathcal{J}^*=\overline{[v,\adjointJ \overline{v}]},\quad{\rm and}  
		   \quad \innerprodbig{v}{\mathcal J\overline v}=0,\forall v\in\mathbb C^{2m}.
	\end{gather*}
\end{lemma}
	
The quaternion Moore-Penrose (MP)  inverse can be defined similarly as its real/complex counterpart \cite[section 1.6]{wei2018quaternion}. For   $\qmat{A}\in\bbQmn$,    there exists a unique solution $\qmat{X}$, denoted as $\qmat{A}^\dagger$, that satisfies the following four matrix equations: 	\begin{gather*}%[label =\color{black} (\arabic*)]
		  \qmat{A}\qmat{X}\qmat{A}=\qmat{A}, ~~
		   \qmat{X}\qmat{A}\qmat{X}=\qmat{X} ,~~
		   (\qmat{A}\qmat{X})^*=\qmat{A}\qmat{A}\qmat{X}, ~~
	 (\qmat{X}\qmat{A})^*=\qmat{X}\qmat{A}.
	\end{gather*}

	\begin{lemma}\label{lem:AdaggerComplex}
		Let $\qmat{A}\in\mathbb{Q}^{m\times n}$; then   $\chi_{\qmat{A}^\dagger}=\left(\chi_\qmat{A}\right)^\dagger$.
	\end{lemma}
	\begin{proof}
		The lemma can be proved by checking the MP inverse directly.
		% \begin{align*}
		% \chiAdag \chiQ{A} \chiAdag &= \mathlarger{\chi}_{\qmat{A}^\dagger \qmat{A}\qmat{A}^\dagger}	=\chiAdag,\quad % 
		% \chiQ{A}\chiAdag \chiQ{A}  = \chilarger{\qmat{A}\qmat{A}^\dagger\qmat{A}} = \chiQ{A}, \\%	
		% \bigxiaokuohao{ \chiQ{A}\chilarger{ \qmat{A}^\dagger}  }^*  &= \mathlarger{\chi}^*_{\qmat{A}\qmat{A}^\dagger} = \chilarger{ (\qmat{A}\qmat{A}^\dagger)^* } = \chilarger{\qmat{A}\qmat{A}^\dagger} = \chiQ{A}\chilarger{\qmat{A}^\dagger},\\ %~	\left
		% \bigxiaokuohao{ \chilarger{ \qmat{A}^\dagger}\chiQ{A}  }^*  &= \mathlarger{\chi}^*_{\qmat{A}^\dagger\qmat{A}} = \chilarger{ (\qmat{A}^\dagger\qmat{A})^* } = \chilarger{\qmat{A}^\dagger\qmat{A}} = \chilarger{\qmat{A}^\dagger}\chiQ{A} .
		% \end{align*}
	\end{proof}

	\begin{lemma}\label{lem:j_pm_jstar}
		Let $U\in\bbC^{2m\times n}$; denote $M:=[U,\adjointJ\overline{U}]$ and $P_{M }:=  MM^\dagger$ the orthogonal projection onto $\mathcal R(M)$. Then $\adjointJ P_{M } \adjointJ^* = \adjointJ^*P_{M } \adjointJ = \overline{P_M}$.
	\end{lemma}
	\begin{proof}
		Denote $\qmat{U}\in\bbQ^{m\times n}$ such that $\cptQc{U}=U$. Then $\chiQ{U}=M$. Thus 
		\begin{align*}
			\adjointJ P_{M } \adjointJ^* = \adjointJ \chiQ{U}\chiQ{U}^\dagger \adjointJ^* = \adjointJ \chilarger{\qmat{U}\qmat{U}^\dagger }\adjointJ^* = \overline{\chilarger{ \qmat{U}\qmat{U}^\dagger }} = \overline{\chiQ{U} \chiQ{U}^\dagger}= \overline{MM^\dagger}=\overline{P_M},
		\end{align*}
		where the second and fourth equalities are due to Lemma \ref{lem:AdaggerComplex} and Proposition \ref{prop:Chi_properties}, and the third one comes from Lemma \ref{lem:J-adjoint}. Verifying $\adjointJ^*P_{M } \adjointJ = \overline{P_M}$ is similar. 
	\end{proof}

	\begin{lemma}
		\label{lem:range_equivalent_quaternion_chi_Q}
		Let $\qmat{A},\qmat{B}\in\bbQmn$. Then $\rangemat{\qmat{A}}=\rangemat{\qmat{B}}$ if and only if $\rangemat{\chiQ{A}}=\rangemat{\chiQ{B}}$. 
	\end{lemma}
One can verify the above lemma as $  
			\qmat{A}\qmat{A}^\dagger = \qmat{B}\qmat{B}^\dagger   \Leftrightarrow \chilarger{  \qmat{A}\qmat{A}^\dagger    } = \chilarger{\qmat{B}\qmat{B}^\dagger} \Leftrightarrow \chiQ{A}\chiQ{A}^\dagger  = \chiQ{B}\chiQ{B}^\dagger $. 
 
	% $\qmat{X}$ is called the Moore-Penrose inverse of $\qmat{A}$, denoted by 
	
	Analygously to the real/complex case, quaternion matrices admit SVD:
	\begin{theorem}{(Compact QSVD \cite[Theorem 7.2]{zhangQuaternionsMatricesQuaternions1997})}\label{lem:QSVD}
		Let $\qmat{A}\in\mathbb{Q}^{m\times n}$ $(m\geq n)$ be of rank $r$. Then there exists unitary quaternion matrices $U\in \mathbb{Q}^{m\times m}$, $V\in\mathbb{Q}^{n\times n}$ and diagonal real matrix $\Sigma = \diag(\sigma_1,\ldots,\sigma_r,0,\ldots,0)\in\mathbb{R}^{n\times n}$ with $\sigma_1\geq\cdots\geq\sigma_r>0$, such that $\qmat{A}=\qmat{U}\Sigma\qmat{V}^*$. 
		% $$UAV=\begin{pmatrix}
		% 	D_r&0\\
		% 	0&0
		% \end{pmatrix}$$
	\end{theorem}

The following property, which can be deduced from \cite{zhangQuaternionsMatricesQuaternions1997},   reveals the relation between the SVD of $\mathbf A\in\mathbb{Q}^{m\times n}$ and  its complex representation $\chi_{\mathbf A}$:
\begin{proposition}\label{lem:qsvd_svd_relation}
	Under the notations in Theorem \ref{lem:QSVD}, if $\qmat{A}=\qmat{U}\Sigma \qmat{V}^*$, then $\chiQ{A} = \chiQ{U}\cdot S\cdot \chiQ{V}^*$ with $S = \diag(\Sigma,\Sigma)\in\mathbb R^{2n\times 2n}$ is a compact SVD of $\chiQ{A}$, and vice versa.
\end{proposition}

	\begin{lemma}\cite[Section 1.6]{wei2018quaternion}
		Under the notations in Theorem \ref{lem:QSVD}, if $\qmat{A}=\qmat{U}\Sigma \qmat{V}^*$, 
		then $\qmat{A}^\dagger=\qmat{V}\Sigma^\dagger\qmat{U}^*$ with $\Sigma^\dagger = \diag(\sigma_r^{-1},\ldots,\sigma_1^{-1},0,\ldots,0)\in\mathbb R^{n\times n}$. 
	\end{lemma}
	
	% Following lemmas will be used in our theoretical analysis. They are similar to their real version.

	% \subsection{Random variables and random vectors}

	\section{Practical Quaternion   Rangefinders}\label{sec:quaternionFastRangeFinder}

\color{black}

Given a   data matrix $A\in\mathbb F^{m\times n}$ $(\mathbb F=\mathbb R,\bbC,\bbQ)$, a randomized rangefinder first draws a random test matrix $\Omega\in\mathbb F^{n\times s}$ with $s\ll n$ \cite{Practical_Sketching_Algorithms_Tropp}   ($s$ is close to the target rank and sometimes can be regarded as a constant),   takes a sketch $Y = A\Omega$, and then orthonormalizes it, i.e.,  
\[
	Y=A\Omega \in\mathbb R^{m\times s},~ Q = \texttt{orth}(Y),
	\]
where $Q$ is orthnormal and preserves the range of $Y$. In the real/complex case, computing $Q$ is cheap by using QR decomposition, while things change in the quaternion setting, especially for large-scale problems, as discussed in the introduction. To better fit into the modern need, we present two practical rangefinders in this section by trading accuracy or space for time cost. To achieve this, we appropriately employ mature libraries such as QR, SVD, and linear equation solvers in complex arithmetic for heavy computations.  Finally, we will compare the running time and accuracy  of the proposed rangefinders with previous techniques. 

\color{black}

	\subsection{Pseudo-QR}
	
%Given a quaternion matrix $X = X_0 + X_1 \j \in \mathbb{Q}^{m\times n}$, consider its compact complex representation
%	\begin{align*}
%X_c = 	\begin{bmatrix}
%		X_0\\
%		X_1
%	\end{bmatrix} \in\mathbb{C}^{2m\times n}.
%\end{align*} 	
%From a data viewpoint, 	$X_c$ shares the same information as $X$. On the other hand, although $X_c$ does not capture the full structure of $X$, as the full representation $\chi_X$ can be conveniently generated from $X_c$, we expect that $X_c$  preserves certain structure of $X$. Starting from this observation, we prefer to operate on $X_c$ to obtain our first rangefinder. 

%	We present the first rangefinder in this subsection. To gain efficiency, our first idea is to employ the   QR in the complex arithmetic to the compact complex representation of the sketch $Y$

Given a quaternion matrix $\qmat{X} = X_0 + X_1 \bj \in \mathbb{Q}^{m\times n}$,   its full information has been contained in its compact represtation 
%	\begin{align*}
	$
	\qmat{X}_c = 	\begin{bmatrix}
				X_0\\
			-	\overline{X_1}
			\end{bmatrix} \in\mathbb{C}^{2m\times n},
		$
%	\end{align*} 	, 
while   full   representation  $\chi_{\qmat{X}}$ futher preserves its structure as an  operator.  Although $\qmat{X}$ does not capture the entire structure in the same way as $\chiQ{X}$, the fact that $\chiQ{X}$ can be easily derived from $\cptQc{X}$ implies that underlying structure may still be preserved within $\qmat{X}_c$. Starting from this observation, we prefer to operate on $\qmat{X}_c$ to obtain our first rangefinder.

%		Our goal is constructing a faster rangefinder satisfying two criteria above only using the built-in function in MATLAB. Here we propose a rangefinder based on the \texttt{qr} on column representation $Y_c$ of quaternion sketch $Y=A\Omega\in\mathbb{Q}^{m\times s}$ and two to three condition numbers reduction steps.
		
		% \subsubsection{Complex QR}
		Let $ \qmat{Y}=\mathbf A\boldsymbol{\Omega} =  Y_0+ Y_1 \bj\in\mathbb{Q}^{m\times s} $  with $m>s$ be the sketch of the   data matrix $\mathbf A\in\mathbb{Q}^{m\times n}$, with $\boldsymbol{\Omega}\in\mathbb{Q}^{n\times s}$ the random test matrix. Let 	$
		\cptQc{Y} = 	\begin{bmatrix}
			Y_0\\
			-\overline{Y_1}
		\end{bmatrix} \in\mathbb{C}^{2m\times s}
		$ be its compact complex representation.  Then, a   thin QR in the complex arithmetic can be directly applied to $\cptQc{Y}$:
		\begin{align}\label{eq:Y_equals_HcR_old}
			  \cptQc{Y} %= \begin{bmatrix}
			% 	Y_0\\
			% 	-\overline{Y_1}
			% \end{bmatrix}
			=  Q R,
		\end{align}
where $Q\in\mathbb{C}^{2m\times s}$ is orthonormal in the complex space, and $R\in\mathbb{C}^{s\times s}$ is upper triangular. 	Then, 	we partition $Q$ as 
$
Q  = \begin{bmatrix}
	Q_0\\
	Q_1
\end{bmatrix}
$
	with $Q_0,Q_1\in\mathbb{C}^{m\times s}$. Furthermore, denote 
\begin{align} \label{eq:relation_Q_H}
	H_0:=Q_0,~ H_1 := -\overline{Q_1}~~{\rm and}~~\mathbf H:=H_0+H_1\bj \in\mathbb{Q}^{m\times s}.
\end{align}
It then follows from \eqref{eq:relation_Q_H} that  the compact   representation of $\mathbf H$ is exactly $Q$:
\[
	~ \cptQc{H}  = \begin{bmatrix}
		H_0\\
		-\overline{H_1}
	\end{bmatrix} =  \begin{bmatrix}
		Q_0\\
		Q_1
	\end{bmatrix} = Q.	
\]
This together with \eqref{eq:Y_equals_HcR_old} shows that $\cptQc{H}$ and $R$ gives the QR decomposition of $\cptQc{Y}$:
	\begin{align} \label{eq:Y_equals_HcR}
		\cptQc{Y} = \cptQc{H} R,\quad\cptQc{H}\in\mathbb{C}^{2m\times s},~R\in\bbC^{s\times s}.
	\end{align}
			The following result shows that $\mathbf H$ has the same range as $\mathbf Y$:

			\begin{proposition}\label{thm:p_QR_same_range_Y_H}
				Assume that $\mathbf Y\in\bbQms$ has full column rank. Then, 
				\begin{align}\label{eq:EqualSpacepseudoQR}
					\rangemat{\qmat{H}} =\rangemat{\qmat{Y}}.
				\end{align}
			\end{proposition}
			\begin{proof} It suffices to show that there exists an invertible   matrix $\mathbf R\in\mathbb{Q}^{s\times s}$ such that $\mathbf Y = \mathbf H\mathbf R$. 
				 Let $\cptQc{Y}=\cptQc{H}R$ be as in \eqref{eq:Y_equals_HcR}.
				%   the thin QR decomposition of $\cptQc{Y}$ where $\cptQc{H}\in\mathbb{C}^{2m\times n}$ is   orthonormal and $R\in\mathbb{C}^{n\times n}$ is upper triangular. 
				  As $\mathbf{Y}$ is of full column rank, Proposition \ref{prop:Chi_properties}     indicates that $\chiQ{Y}$ is also of full column rank, and so is $\mathbf{Y}_c$. Thus $R$ is invertible. 
				  By Lemma \ref{lem:J-adjoint},   $\chiQ{Y}$ can be represented as:
				\begin{align*}
				\chiQ{Y}	 &=[\mathbf Y_c,\mathcal{J}\overline{\mathbf  Y_c}] 
					 =[\mathbf H_cR,\mathcal{J}\overline{\mathbf H_cR}]\\
					&=[\mathbf H_c,\mathcal{J}\overline{\mathbf H_c}]\diag(R,\overline{R}) =\chiQ{H} \diag(R,\overline{R}) . 
				\end{align*}
%			On the other hand, denote the quaternion matrix $\mathbf R:= R + 0\bj \in\mathbb{Q}^{s\times s}$ with $0$ the zero matrix. 
Transforming back to quaternions,			  the above identity is equivalent to 
			$
			\mathbf Y = \mathbf H R
			$. 
			Thus \eqref{eq:EqualSpacepseudoQR} follows from the invertibility of $R$. 
	%		 Then we observe that $\chi_{\mathbf R} = \diag(R,\overline{R})$. 		It follows from the invertibility of $R$ that
 % $\diag(R,\overline{R}$ is also invertible, and so is $\chi_{\mathbf R}$, which again by \cite[Theorem 4.3]{zhangQuaternionsMatricesQuaternions1997} indicates the invertibility of $\mathbf R$ in the quaternion space.  \eqref{eq:EqualSpacepseudoQR} follows. 	 %	We can obtain a quaternion identity $Y=HR$ where $\chi_R=diag(R,\overline{R})$. Obviously $R$ is invertible. Thus, $Y$, $H$ are offset, and they have the same right range.  
			\end{proof}
			\begin{remark}
				Even if $\mathbf Y$ is rank-deficient, we still have $\mathcal{R}(\mathbf H) \supset \mathcal{R}(\mathbf Y) $, i.e., the range of $\mathbf H$ captures the range of the sketch $\mathbf Y$.
			\end{remark}
		
Throughout this work, we respectively denote 
			\begin{align*}\label{eq:def:cond_num}
				\sigma_i(\qmat{H}),~\sigma_{\max}(\qmat{H})=\sigma_1(\qmat{H}),~\sigma_{\min}(\qmat{H}) = \sigma_s(\qmat{H}),	~\kappaq{H}:= \sigma_{\max}(\qmat{H})/\sigma_{\min}(\qmat{H}),
				\end{align*}
	as the $i$-th, the largest and   the smallest singular value, as well as the condition number of $\qmat{H}$. 			
The proposition belows shows that, although $\mathbf H$ may be non-orthonormal in the quaternion domain,   its      singular values	are structured. % This partly explains why $Y_c$ contains certain structure of $\chi_{\mathbf Y}$.	Besides, $H$ has the largest singular value $\sigma_1=1$ with multiplicity at least $\lfloor s\rfloor/2$.
			\begin{proposition}\label{thm:spectralNormProofOfpseudoQR}
				All the singular values of $\qmat{H}$ takes the form  of 
				\[
					 \sqrt{1\pm \mu_1},\sqrt{1\pm \mu_2},\ldots,	
				\]
				with $ 0\leq \mu_i \leq 1$, and  $\sigma_{\max}(\qmat{H})\leq \sqrt 2$.
				% and satisfy 
				% \begin{align}
					% \sigma_1\geq\sigma_2\geq\ldots\geq\sigma_{\left\lfloor s/2\right\rfloor }\geq 1\geq\sigma_{\left\lfloor s/2\right\rfloor+1 }\geq\ldots\geq\sigma_s\geq0.
				% \end{align}
			\end{proposition}
			\begin{proof}
			By Lemma \ref{lem:qsvd_svd_relation}, 	It suffices to consider the singular values of its complex  representation  
				$$
				\chi_{\mathbf H}=\begin{bmatrix}
					H_0&H_1\\
					-\overline{H_1}&\overline{H_0}
				\end{bmatrix}= \begin{bmatrix}
					\mathbf H_c,& \cptQa{H}
				\end{bmatrix} = \begin{bmatrix}
					\mathbf H_c,& \adjointJ \overline{\mathbf  H_c}
				\end{bmatrix}\in\bbC^{2m\times 2s}.$$
				%$H_c$ is the first column block and $H_a$ is J-adjoint of $H_c$, i.e., $
				%where $\mathbf H=H_0+H_1\bj$, $H_0,H_1\in\mathbb{C}^{m\times n}$.
				 It follows from    \eqref{eq:Y_equals_HcR} that $\mathbf  H_c^* \mathbf  H_c =  I_s$;    $\cptQa{H}^*\cptQa{H}= \overline{\mathbf  H_c}^*\adjointJ^* \adjointJ \overline{\mathbf  H_c} = \overline{\mathbf  H_c}^* \overline{\mathbf  H_c}= I_s$ as well.  
				Then we have: 
				\begin{align*}
					\chiQ{H}^*\chiQ{H}=\begin{bmatrix}
						I_s&H_0^*H_1-\overline{H_1}^*\overline{H_0}\\
						H_1^*H_0-\overline{H_0}^*\overline{H_1}&I_s
					\end{bmatrix}
				\end{align*}
				and it is positive semi-definite.
				 Denote $T := \mathbf H_c^*\cptQa{H} =   H_0^*H_1-\overline{H_1}^*\overline{H_0}$; then %$\chiQ{H}^*\chiQ{H}$ can be written as  %we have $T=-T^T$ and we rewrite $\chi_{\mathbf H}^*\chi_{\mathbf H}$ as
				\begin{align}\label{eq:chiH*chiH}
					\chiQ{H}^*\chiQ{H}=\begin{bmatrix}
						I_s&T\\
						T^*&I_s
					\end{bmatrix}.
				\end{align} 

				% % Considering the $\lambda$-matrix 
				% % \begin{align}
				% 	% \chi_{\mathbf{H}}^*\chi_{\mathbf H}-\lambda I=\begin{bmatrix}
				% 		% (1-\lambda)I&T\\
				% 		% T^*&(1-\lambda)I
				% 	\end{bmatrix},
				% \end{align}
				% the solutions of the equation $\rm det(\chi_{\mathbf{H}}^*\chi_{\mathbf H}-\lambda I)=0$ are eigenvalues of $\chi_{\mathbf{H}}^*\chi_{\mathbf H}$. Determinant of the $\lambda$-matrix can be computed with assumption that $T$ is invertible by

				Consider the characteristic polynomial of $\chiQ{H}^*\chiQ{H}$:
				\begin{align*}
					\rm det\bigxiaokuohao{ \begin{bmatrix}
						(1-\lambda)I_s&T\\
						T^*&(1-\lambda)I_s
					\end{bmatrix} } & =\rm det\bigxiaokuohao{(1-\lambda)I_s}\rm det\left( (1-\lambda)I_s -\left(1-\lambda\right)^{-1}T^*T \right)\\
					&= \rm det\bigxiaokuohao{ (1-\lambda)^2 I_s - T^*T }.
				\end{align*}
				Let  $\mu^2$ be an eigenvalue of $T^*T$. Then the above relation shows that $\lambda = 1\pm \mu$ are a pair of eigenvalues of $\chiQ{H}^*\chiQ{H}$, i.e., $\sqrt{1\pm \mu}$ are a pair of singular values of $\chiQ{H}$, which together with Lemma \ref{lem:qsvd_svd_relation} shows that they are also singular values of $\qmat{H}$. Finally, it is easily seen from \eqref{eq:chiH*chiH} that $I_s\succeq T^*T$, namely, $\mu \in [0,1]$, which implies that   $\sigma_{\max}(\qmat{H})\leq \sqrt 2$.
	\iffalse
				 One  obtains   
				\begin{align}
					\begin{bmatrix}
						I&0\\
						T^*&I
					\end{bmatrix}^{-1}
					\begin{bmatrix}
						I&T\\
						T^*&I
					\end{bmatrix}
					\begin{bmatrix}
						I&0\\
						T^*&I
					\end{bmatrix}=
					\begin{bmatrix}
						I&0\\
						0&I-T^*T
					\end{bmatrix},
				\end{align}
				which shows that $\chi_H^*\chi_H$ is similar to $diag(I,I-T^*T)$. Due to positive semi-definite property of $T^*T$ and $\chi_H^*\chi_H$, the eigenvalues $\lambda_i$ of  $H^*H$ satisfying:
				\begin{align}
					\lambda_1=\lambda_2=\ldots\lambda_{\left\lfloor s/2\right\rfloor}\geq\lambda_{\left\lfloor s/2\right\rfloor+1}\geq\ldots\geq\lambda_s\geq 0
				\end{align}
				which yields result directly.
				\fi
			\end{proof} 

			\begin{remark}\label{rmk:Hc_Ha_T}	The two properties above are not enough to control the condition number $\kappaq{H}$. In fact,   the analysis shows that the smallest singular value of $\qmat{H}$ depends on $T = \cptQc{H}^*\cptQa{H}$. If the angle between $\spanmat{\cptQc{H}}$ and $\spanmat{\cptQa{H}}$ is very small, then $T$ tends to be an identity matrix and so the smallest eigenvalue of $\chiQ{H}^*\chiQ{H}$ tends to zero; on the contrary, if $\spanmat{\cptQc{H}}$ and $\spanmat{\cptQa{H}}$ are perpendicular to each other, then $T$ is exactly $0$ and the smallest eigevalue of $\chiQ{H}^*\chiQ{H}$ is $1$. However, by the construction \eqref{eq:Y_equals_HcR},   $\spanmat{\cptQc{H}} = \spanmat{\cptQc{Y}}$ and $\spanmat{\cptQa{H}} = \spanmat{\cptQa{Y}}$. As $[\cptQc{Y},\cptQa{Y}] = \chiQ{Y}$ and $\qmat{Y}$ is data-dependent, we cannot make any assumption on the angle between $\spanmat{\cptQc{Y}}$ and $\spanmat{ \cptQa{Y}}$. Thus    $\sigma_{\min}(\qmat{H})$ cannot be estimated, nor $\kappa(\mathbf H)$. 
			\end{remark}

		Nevertheless, empirically  we usually observe that    $\kappaq{H}$ is at least two times smaller than $\kappaq{Y}$. To further reduce $\kappaq{H}$, it is possible to perform the correction step  several times:
		\begin{align}\label{eq:newton_schulz}
			\qmat{H}  \leftarrow 0.5\qmat{H}(3I_s - \qmat{H}^*\qmat{H}).
		\end{align}
It was shown in \cite[Sect. 8]{higham2008functions} that if all singular values $\sigma_i(\qmat{H}) \in [0,\sqrt 3]$ (which is confirmed by Proposition \ref{thm:spectralNormProofOfpseudoQR}), then  \eqref{eq:newton_schulz} ensures that all $\sigma_i(\qmat{H})\rightarrow 1$   quadratically, and hence $\qmat{H}$ is gradually well-conditioned. However, we empirically find that  \eqref{eq:newton_schulz} converges very slowly. Instead, we resort to the following
  range-preserving correction step:
			\begin{align}\label{eq:CondNumReductionIter}
				\qmat{H}_{new} \leftarrow (1-\epsilon)\qmat{H}+\epsilon (\qmat{H}^\dagger)^*.
			\end{align}
			The next proposition shows that, if $\epsilon$ is chosen close to the smallest singular value of $\qmat{H}$, and if $\kappaq{H}$ is still large, then it will be reduced rapidly.
			\begin{proposition}\label{thm:CondReductionSpeed}
			Suppose in \eqref{eq:CondNumReductionIter}, one chooses $\epsilon$ such that  $\epsilon \in \bigzhongkuohao{\sigma_s(\qmat{H}),\delta\sigma_s(\qmat{H})}$, where $\delta\in[1,\sqrt{7}/2]$. If $\kappa(\qmat{H})>\max \{2\sqrt{2}\delta,2\delta^2+1/2\}$, then $\kappa(\qmat{H}_{new})< \sqrt{\kappa(\qmat{H})}$.  
		\end{proposition}
		The role of $\delta$ above means that one can compute $\sigma_s(\qmat{H})$ inexactly; in practice, one  usually performs  two or three power iterates of approximating $\sigma_1(\qmat{H}^{\dagger})$ (namely, $\sigma_s(\qmat{H})^{-1}$) to obtain $\epsilon$.   The upper bound of $ \max \{2\sqrt{2}\delta,2\delta^2+1/2\}$ is $4$, i.e, when $\kappaq{H}>4$, it will be reduced to its square root by \eqref{eq:CondNumReductionIter}.

In practice, one can execute the correction step \eqref{eq:CondNumReductionIter} at most   three times to obtain a desirable $\qmat{H}$. The following corollary shows that, if $\qmat{H}$ is generated by \eqref{eq:Y_equals_HcR_old} and \eqref{eq:relation_Q_H} with $\kappaq{H}<10^8$, then $\kappaq{H}$ will not exceed $10$ after at most three correction steps. Empirically, an $\qmat{H}$ with $\kappaq{H}<10$ is enough for practical use. \begin{corollary} \label{col:correction_H_iteratively}
	Let $\qmat{H}_{k+1} \leftarrow (1-\epsilon_k)\qmat{H}_k+\epsilon_k (\qmat{H}_k^\dagger)^*$ with $\qmat{H}_0$   given by   \eqref{eq:Y_equals_HcR_old} and \eqref{eq:relation_Q_H}. If   $\epsilon_k \in \bigzhongkuohao{\sigma_s(\qmat{H}_k),\delta\sigma_s(\qmat{H}_k)}$, where $\delta\in[1,\sqrt{7}/2]$, and if $\kappa(\qmat{H}_k)>\max \{2\sqrt{2}\delta,2\delta^2+1/2\}$, then $\kappa(\qmat{H}_{k+1})< \sqrt{\kappa(\qmat{H}_k)}$.  
\end{corollary}
% \begin{proof}
% Using the analysis in Theorem \ref{thm:CondReductionSpeed}, one can  use the induction method to show that if all the singular values $\sigma_i(\qmat{H}_k)$ lies in  $\bigzhongkuohao{ 2\sqrt{\sigma_s(\qmat{H}_{k-1})(1-\sigma_s(\qmat{H}_{k-1}) ) },  \max \{\sqrt{2},\delta + 1/(4\delta)\} }$, then all $\sigma_i(\qmat{H}_{k+1})$ also lie in $[2\sqrt{\sigma_s(\qmat{H}_{k})(1-\sigma_s(\qmat{H}_{k}) ) },  \max \{\sqrt{2},\delta + 1/(4\delta)\}  ]$.  Similar to the proof of   Theorem \ref{thm:CondReductionSpeed},  $\kappa(\qmat{H}_{k+1})< \sqrt{\kappa(\qmat{H}_k)}$.
% \end{proof}

The proofs of Proposition \ref{thm:CondReductionSpeed} and Corollary \ref{col:correction_H_iteratively} are given in the appendix. 
  The iterative scheme above is essentially the quaternion version of the     Newton method for computing   polar decomposition \cite[Sect. 8]{higham2008functions} but with more relaxed parameters. % We   emphasize again that two or three iterations are enough; there is no need to run it until convergent.

% The   iterative s

We summarize the computation and analysis of pseudo-QR as follows. The pseudo code is given in Algorithm \ref{alg:pseudoQR}.
\begin{proposition}
	\label{prop:pseudo-qr_summarize}
Let $\qmat{H}$ be   given by \eqref{eq:Y_equals_HcR_old} and \eqref{eq:relation_Q_H}  and then the correction step \eqref{eq:CondNumReductionIter} is excecuted three times.  Then  $\rangemat{\qmat{H}_{new}}= \rangemat{\qmat{Y}}$; under the setting of Corollary \ref{col:correction_H_iteratively}, if $\kappaq{H}<10^8$, then $\kappa(\qmat{H}_{new} )<10$. 
\end{proposition}

\begin{algorithm}
	\caption{(pseudo-QR) quaternion pseudo-QR implementation}\label{alg:pseudoQR}
	\begin{algorithmic}[1]
		\Require   Sketch matrix $\qmat{Y}\in\bbQms$.%, integer $s$.
		\Ensure A   quaternion rangefinder $\qmat{H}\in\bbQms$
		% \State Draw a   test matrix $\boldsymbol{\Omega}\in\bbQns$ and compute the randomized sampling $\qmat{Y}=\qmat{A}\boldsymbol{\Omega}$. 
		\State Construct the $2m\times s$ complex matrix $\cptQc{Y}$ from $\qmat{Y}$.
		\State Compute complex QR $[U,\sim ]=\texttt{qr}(\cptQc{Y},0)$
		\State Construct $\qmat{H}=H_0+H_1\bj$ from   $U$ with  $H_0=U(1:m,:)  $ and $ H_1=- \overline{U(m+1:2m,:)}$ .       
		\State Execute the correction step \eqref{eq:CondNumReductionIter} a few times (often   $3$ times). Use Algorithm  \ref{alg:qequation_cplx_solver} to solve $\mathlarger{\chi}_{\qmat{H}^*\qmat{H} }Z = (\qmat{H}^*)_c$ to obtain $ \qmat{H}^\dagger  $.
	\end{algorithmic}
\end{algorithm}

The remaining question is how to compute  $(\qmat{H}^\dagger)^*$. \cite{higham2008functions} computed \eqref{eq:CondNumReductionIter} using QR, which is expensive in the quaternion setting.  %In the real/complex case, it can be efficiently computed by QR, which may be expensive in the quaternion case. 
We convert it to solving linear equations in the complex arithmetic such that highly efficient solvers can be used. To this end,   %assume that $\qmat{Y}$  has full column rank (and by Proposition \ref{thm:p_QR_same_range_Y_H}, so does $\qmat{H}$); further 
assume that $\qmat{H}$ is not too ill-conditioned (say, $\kappa(\qmat{H}^*\qmat{H})<10^{16} $). Since $ \qmat{H}^\dagger  = (\qmat{H}^*\qmat{H})^{-1}\qmat{H}^*$, we can solve the equation $\qmat{H}^*\qmat{H}\qmat{X} = \qmat{H}^*$ to obtain $(\qmat{H}^\dagger)^*$. The following idea comes from \cite{zhangQuaternionsMatricesQuaternions1997}. For  a general quaternion linear equation  $\qmat{A}\qmat{X}=\qmat{B}$ with $\qmat{A}\in\bbQmn, \qmat{X}\in\bbQns,\qmat{B}\in\bbQ^{m\times s}$,  according to Proposition \ref{prop:Chi_properties},
\begin{align*}
	\qmat{A}\qmat{X}=\qmat{B}  \Leftrightarrow \chiQ{A}\chiQ{X} = \chiQ{B} \Leftrightarrow \chiQ{A}\bigzhongkuohao{ \cptQc{X},\cptQa{X} } = \bigzhongkuohao{\cptQc{B},\cptQa{B}  }.
\end{align*}
In fact, using the relation $\cptQa{X} = \mathcal J\overline{\cptQc{X}}$, solving the complex equation $\chiQ{A}Z = \cptQc{B}$ is enough to give a solution to $\qmat{A}\qmat{X}=\qmat{B}$. To see this, let $\hat Z\in\bbC^{2n\times s}$ be a solution to $\chiQ{A}Z = \cptQc{B}$ and partition it as $\hat Z = [\hat Z_0^*,~ \hat Z_1^*]^*$ with $\hat Z_0,\hat Z_1\in\bbC^{n\times s}$. Similar to \eqref{eq:relation_Q_H}, let $\qmat{X}:= \hat Z_0 - \overline{\hat Z_1}\bj$; then $\cptQc{X} = \hat Z$, namely, $\chiQ{A}\cptQc{X}=\cptQc{B}$, and it follows from Lemma \ref{lem:J-adjoint} that
\begin{align*}
	\chiQ{A}\cptQa{X} = \chiQ{A}\mathcal J\overline{\cptQc{X}} = \mathcal J\overline{\chiQ{A}}\mathcal J^*\mathcal J \overline{\cptQc{X}} = \mathcal J \overline{ \chiQ{A}\cptQc{X} } = \mathcal J\overline{\cptQc{B}} = \cptQa{B}, 
\end{align*}
which together with $\chiQ{A}\cptQc{X}=\cptQc{B}$ means that $\qmat{A}\qmat{X}=\qmat{B}$.  

Therefore,   it suffices to solve $\mathlarger{\chi}_{\qmat{H}^*\qmat{H} }Z = (\qmat{H}^*)_c$ to obtain $ \qmat{H}^\dagger  $. Note that $\mathlarger{\chi}_{\qmat{H}^*\qmat{H} }\in\bbC^{2s\times 2s}$ and $(\qmat{H}^*)_c\in\bbC^{2s\times m}$. As the sampling size $s$ is usually small, solving this equation in the complex arithmetic  can be efficient by using mature solvers.  

\begin{algorithm} 
	\caption{Quaternion linear equations solver}\label{alg:QLEQ}
	\begin{algorithmic}[1]
		\Require $\qmat{A}\in\mathbb{Q}^{n_1\times n_2}$ and $\qmat{B}\in\mathbb{Q}^{n_1\times l}$ 
		\Ensure  Quaternion matrix $\qmat{X}\in\mathbb{C}^{n_2\times l}$ satisfy $\qmat{A}\qmat{X}=\qmat{B}$
		\State Construct $\chiQ{A}\in\mathbb{C}^{2n_1\times 2n_2}$ and $\cptQc{B}\in\mathbb{C}^{2n_1\times l}$.
		\State Compute $\cptQc{X}=\chiQ{A}\backslash \cptQc{B}$.
		\State Construct $\qmat{X}=\cptQc{X}(1:n_2,:)-\overline{\cptQc{X}(n_2+1:2n_2,:)}\bj$
	\end{algorithmic}\label{alg:qequation_cplx_solver}
\end{algorithm}

	\subsection{Pseudo-SVD}\label{sec:pseudo-svd}

	If the orthonormality of $\qmat{H}$ is  absolutely required,  
	%In the complex QR step of pesudo-QR, we only enforce orthonormality on $\cptQc{H}$. %To obtain a better well-conditioned $\qmat{H}$, 
	   the relation between   $\cptQc{H}$ and $\cptQa{H}=\adjointJ\overline{\cptQc{H}}$ should be taken into account. %However, if we can force the $H_c$ to satisfy some constructions, a better rangefinder may be obtained based on the underlying constraints.
	   A motivation is from the following lemma. %J-adjoint mentioned in lemma \ref{lem:J-adjoint}  and property of eigenvectors as follows.
	  \begin{lemma}[c.f. \cite{zhangQuaternionsMatricesQuaternions1997}]\label{lem:eigenvectorsSpace}
	Let $\qmat{A}\in\bbQ^{m\times m}$ be Hermitian.	If $v$ is an eigenvector of $\chiQ{A}$  corresponding to the eigenvalue $\lambda$, then     $\adjointJ\overline{v}$ is also an eigenvector asssociated to   $\lambda$. Moreover, $v\perp\mathcal J\overline{v}$. 

	  \end{lemma}

	  The above lemma implies that in the ideal situation, if every (normalized) columns of $\cptQc{H}$   are   from different eigenspaces, then $\cptQc{H}^*\cptQc{H}=I$, $\cptQc{H}^*\cptQa{H}=\cptQc{H}^*\adjointJ\overline{\cptQc{H}}=0$, which means that $\chiQ{H}^*\chiQ{H}=I$, and  so  $\qmat{H}$ is orthonormal. To this end, we   resort to (and modify) the method introduced in \cite{lebihan2004SingularValue} to find a suitable pair $(\cptQc{H},\cptQa{H})$. The idea is to find a  QSVD of $\qmat{Y}$ via computing the   complex SVD of $\chiQ{Y}$. 	Let
	%   To this end, the complex SVD of the full complex representation of the sketch $\chiQ{Y}$ can be considered. Let 
	%   \begin{proof}
	% 	See \cite{bunse-gerstnerQuaternionQRAlgorithm1989}. 
	% 		% Since $\chi_Av=v\lambda$, 
	% 		% By lemma \ref{lem:J-adjoint}, we have $$\chi_A \adjointJ\overline{v}=\adjointJ\overline{\chi_A}\adjointJ^*\adjointJ\overline{v}=\adjointJ\overline{\chi_A}\overline{v}=\adjointJ\overline{v\lambda}=\adjointJ\overline{v}\overline{\lambda}.$$ 
	% 		% When $\qmat{A}$ is Hermitian, so is $\chiQ{A}$, whose eigenvalues are all real. The results follow.
	% 		% which indicates that $J\overline{v}$ is an eigenvector belongs to $\overline{\lambda}$. 
	%   \end{proof}  
% A quaternion SVD decomposition based on the direct SVD of its complex representation is presented in (cite). 
	%Let us add these construction requirements to our rangefinder framework, the intuition is as follows:
	% \begin{enumerate}
	% 	\item $\cptQc{H}$ is orthogonal basis of $\cptQc{Y}$, which contains information captured in randomized embeddings.
	% 	\item The columns of $\cptQc{H}$ should be eigenvectors of an Hermitian complex matrix whose structure is as complex representation.
	% 	\item Columns of $H_c$ belongs to distinct eigenvalues. 
	% \end{enumerate}
	% Through computing the  Compact Singualr Value Decomposition of complex representation $\chiQ{Y}$
	\begin{align}\label{eq:USVOfchiY}
		\chiQ{Y}:=US V^*,\quad U\in\mathbb{C}^{2m\times 2s},~ V\in\mathbb{C}^{2s\times 2s},~S =\begin{bmatrix}
			\Sigma &  0 \\ 0 & \Sigma
		\end{bmatrix} \in\mathbb R^{2s\times 2s} 
	\end{align}
	be a compact SVD of $\chiQ{Y}$, where $\Sigma$ is diagonal. Note that due to duplicated singular values, $U S V^*$ may not take the structure as that in Proposition \ref{lem:qsvd_svd_relation}. Nevertheless, this may be constructed.  Partition 
	\begin{align} \label{eq:sudoSVD_partition_U}
		U = \begin{bmatrix}
			U_{ul} & U_{ur}\\ 
			U_{dl} & U_{dr}
		\end{bmatrix},~
		V = \begin{bmatrix}
			V_{ul} & V_{ur}\\
			V_{dl} & V_{dr}
		\end{bmatrix},~{\rm with}~ U_{ul}\in\mathbb C^{m\times s},~V_{ul}\in\mathbb C^{s\times s}.
	\end{align}
	Denote 
	\begin{align}\label{eq:ComplexUSVToQuaternion}
		\qmat{H}:=U_{ul}  -  \overline{U_{dl}} \bj \in\bbQms\quad{\rm and} \quad
		\qmat{V}:=V_{ul}  - \overline{V_{dl}}\bj\in\mathbb  Q^{s\times s}. 
	\end{align}
Conditionally, $ \qmat{H},\Sigma,\qmat{V}^*$ given by the above approach is a compact QSVD of $\qmat{Y}$:

% \begin{lemma}\label{lem:span_U_equal_span_U_JconjU}
% 	Let $U $ consit  of the singular vectors of $\chiQ{Y}$ that spans  the left invariant     subspace corresponding to a singular value $\sigma$. Then $\spanmat{U }=\spanmat{U ,\adjointJ \overline{U }}$.
% \end{lemma}
% \begin{proof}
%   Lemma \ref{lem:eigenvectorsSpace} shows that every column in $\adjointJ\overline{U }$ is also a singular vector of $\sigma$, which together with the assumption of $U $ gives the desired result.  
% \end{proof}
\begin{theorem}\label{thm:pseudoSVD}
Let $\qmat{H},\Sigma,\qmat{V}$ be given by \eqref{eq:USVOfchiY}  and \eqref{eq:ComplexUSVToQuaternion}.	If all the singular values of $\qmat{{Y}}$ are distinct, i.e., $\Sigma$ consists of distict singular values, then $
\qmat{H}\Sigma\qmat{V}^*$ is a compact QSVD of $\qmat{Y}$, i.e., $\qmat{Y}=\qmat{H}\Sigma\qmat{V}^*$.   
\end{theorem}
We first present the following lemma.
\begin{lemma}\label{lem:sudoSVD_Usigma_span_invariant_distict_sig}
	Let $\sigma_1> \cdots >\sigma_r$ be $r$ singular values of $\chiQ{Y}$, each of which has multiplicity exactly two, and $U_\sigma = [u_1,\ldots,u_r]$ be $r$ left singular vectors corresponding to $\sigma_1,\ldots,\sigma_r$. Then $[U_\sigma,\adjointJ\overline{U_\sigma}]$ is orthonormal and spans the left invariant subspace corresponding to $\sigma_1,\ldots,\sigma_r$.  
	\end{lemma}
	\begin{proof}
		By the multiplicity assumption on $\sigma_1,\ldots,\sigma_r$ and noting    $u_j\perp\adjointJ\overline{u_j}$,  Lemma \ref{lem:eigenvectorsSpace} shows that $\spanmat{u_j,\adjointJ \overline{u_j}}$ is the left invariant subspace of $\chiQ{Y}$ corresponding to $\sigma_j$. On the other hand, each $u_j$ belongs to a distinct $\sigma_j$, and so $u_i\perp u_j$, $ {u_i}\perp\adjointJ\overline{u_j}$, $\adjointJ\overline{u_i}\perp\adjointJ\overline{u_j}$, $i\neq j$. Thus the results follow.
	\end{proof}
\begin{proof}[Proof of Theorem \ref{thm:pseudoSVD}]
	 The assumption shows that %every singular value of $\chiQ{Y}$ has multiplicity exactly two, i.e.,
	  $\Sigma$ consists of distinct singular values. By Lemma \ref{lem:qsvd_svd_relation},    it suffices to prove that
	\begin{align}\label{eq:proof:distinct}
	\chiQ{Y} = \chiQ{H}	S \chiQ{V}^* = [\qmat{H}_c,\adjointJ\overline{\cptQc{H}}]\begin{bmatrix}
		\Sigma &  0 \\ 0 & \Sigma
	\end{bmatrix} \begin{bmatrix}
		\cptQc{V}^* \\ (\adjointJ \overline{\cptQc{V}})^*
	\end{bmatrix}
\end{align}
	with $[\qmat{H}_c,\adjointJ\overline{\cptQc{H}}]$ and $ [\cptQc{V}, \adjointJ \overline{\cptQc{V}}]$ orthonormal.  It follows from the construction of $\qmat{H}$ that $\cptQc{H}=\begin{bmatrix}
		U_{ul}\\ U_{dl}
	\end{bmatrix}$, which, by Lemma \ref{lem:sudoSVD_Usigma_span_invariant_distict_sig}, demonstrates the orthonormality of $[\qmat{H}_c,\adjointJ\overline{\cptQc{H}}]$. Similarly, $ [\cptQc{V}, \adjointJ \overline{\cptQc{V}}]$ is   orthonormal.  
	% Let $h$ be a column of  $\cptQc{H}$.  Then   it follows from the construction of $\qmat{H}$ that $\cptQc{H}=\begin{bmatrix}
	% 	U_{ul}\\ U_{dl}
	% \end{bmatrix}$, and so $h$ is a left singular vector of $\chiQ{Y}$ corresponding to a singular value, say $\sigma$.    By Lemma \ref{lem:eigenvectorsSpace},  $\adjointJ\overline{h}$ is also a singular vector of $\sigma$ and    $h\perp \adjointJ\overline{h}$, which shows that they   span  the invariant subspace corresponding to $\sigma$. As different columns of $\cptQc{H}$ corresponds to different singular values in $\Sigma$, $[\qmat{H}_c,\adjointJ\overline{\cptQc{H}}]$ is orthonormal. Similarly, $ [\cptQc{V}, \adjointJ \overline{\cptQc{V}}]$ is also orthonormal. 
The construction of $\qmat{H}$ and $\qmat{V}$ obeys $\chiQ{Y}\cptQc{V} = \cptQc{H}\Sigma$. Then,
	\[
		\chiQ{Y}\adjointJ\overline{\cptQc{V}} = \adjointJ\adjointJ^* \chiQ{Y}\adjointJ \overline{\cptQc{V}} = \adjointJ\overline{\chiQ{Y}}\overline{\cptQc{V}} = \adjointJ \overline{\cptQc{H}\Sigma } = \adjointJ\overline{\cptQc{H}}\Sigma,
		\]
		and so $\chiQ{Y}[\cptQc{V},\adjointJ\overline{\cptQc{V}}] = [\cptQc{H},\adjointJ\overline{\cptQc{H}}]\diag(\Sigma,\Sigma)$, which together with the orthonormality of $[\cptQc{V},\adjointJ\overline{\cptQc{V}}]\in\bbC^{2s\times 2s}$ yields \eqref{eq:proof:distinct}. 
\end{proof}
% \begin{remark}
% 	The above proof implies why the assumption in Theorem \ref{thm:pseudoSVD} is neccessary. Assume that $\sigma$ is a singular value of $\qmat{Y}$ with multiplicity two (and $\sigma$ is that of $\chiQ{Y}$ with multiplicity four). Let $[h_1,h_2]$ be consisted in $\cptQc{H}$ associated with $\sigma$. Then $h_1$ may not be perpendicular to $\adjointJ\overline{h_2}$ as they belong to the same singular subspace. What is worse is that $\spanmat{h_1,h_2,\adjointJ\overline{h_1},\adjointJ\overline{h_2}}$ may not span the invariant subspace of $\sigma$. 
% \end{remark}
\begin{remark}
We provide Theorem \ref{thm:pseudoSVD} and the proof   as we cannot find one in the literature.	 The   proof implies why the   assumption    is neccessary. Consider the  example $\qmat{Y}=I_2\in\bbQ^{2\times 2}$, and so $\chiQ{Y}=I_4\in\bbC^{4\times 4}$. Any orthonormal $U=[u_1,\ldots,u_4]\in\bbC^{4\times 4}$ is a singular vector matrix of $\chiQ{Y}$. If constructing $\cptQc{H}=[u_1,u_2]$, then although $u_1\perp u_2, u_1\perp\adjointJ \overline{u_1}$, as $\adjointJ\overline{u_1}$ may not perpendicular to $u_2$, $\qmat{H}$ may not be orthonormal and so $\qmat{H}$ is not a singular vector matrix of $\qmat{Y}=I_{2}$.  
%  Any normalized  $u\in\bbC^4$ is a singular vector. If constructing $\qmat{H}\in\bbQ^{2\times 2}$ with    $\cptQc{H}=[u,\adjointJ\overline{u}]$, it is clear that $\adjointJ\overline{\cptQc{H}}=\cptQc{H}$ and so $[\cptQc{H},\adjointJ\overline{\cptQc{H}}]$ is rank-deficient, namely, $\qmat{H}$ cannot be orthonormal. 
\end{remark}

However, some issues should   be addressed. Firstly, as $\chiQ{Y}$ is two times larger than $\qmat{Y}$ (in terms of the real elements), directly computing the SVD of $\chiQ{Y}$ seems to be redundant. Nevertheless, recall that $\qmat{Y}\in \bbQms$ is a sketch, whose column size $s$ is usually much smaller than $m$ \cite{Practical_Sketching_Algorithms_Tropp}; thus the SVD of $\chiQ{Y}$ scales well. As shown in subsection  \ref{sec:time_comparison_rangefinders}, rangefinder based on  SVD  of $\chiQ{Y}$ is still much faster than the competitors.

	A much criticized flaw is that when $\qmat{Y}$ is too ill-conditioned (say, $\kappaq{Y}>10^{13}$), due to rounding errors, numerically, the small singular values of $\chiQ{Y}$ may not appear twice, and so doing \eqref{eq:ComplexUSVToQuaternion} may not generate the correct singular vectors for $\qmat{H}$ corresponding to the small singular values  \cite[p. 84]{bunse-gerstnerQuaternionQRAlgorithm1989}. %  computing SVD of $\chiQ{Y}$ may not give the correct SVD of $\qmat{Y}$  
	 The same situation also occurs when $\qmat{Y}$ has duplicated singular values. In these two cases, $\qmat{H}$ given by \eqref{eq:ComplexUSVToQuaternion} is no longer orthonormal and may not even span the correct range.% $\kappaq{H}$ cannot be controlled. 
% If this flaw happens, we cannot obtain the correct range of $\qmat{Y}$. 

Fortunately, we find that the following correction step can tackle the above issue. % because the task is to find a well-conditioned range of $\qmat{Y}$ instead of a  QSVD.
 When the above issue occurs,    numerically computing an SVD of $\chiQ{Y}$   exhibits  the following form:
\begin{align}\label{eq:USVOfchiY_incorrect}
	\chiQ{Y}&=U\tilde S V^*,\quad U\in\mathbb{C}^{2m\times 2s},~ V\in\mathbb{C}^{2s\times 2s},~\tilde S =\begin{bmatrix}
		S_g &  0  \\ 0 & \Sigma_b   
	\end{bmatrix} \in\mathbb R^{2s\times 2s}, \\
  U^*U &= I_{2s}, V^*V = I_{2s},  \Sigma_{b}\in\mathbb R^{2t\times 2t} (t<s), S_g= \begin{bmatrix}
	\Sigma &  0  \\ 0 & \Sigma    
\end{bmatrix} \in\mathbb R^{2(s-t)\times 2(s-t)} ; \nonumber
\end{align}
i.e., now the singular values $\tilde{S}$ can be partitioned as the ``good'' part $S_g$ and the ``bad'' part $\Sigma_d$. $S_g$   consists of   singular values of $\chiQ{Y}$  still appearing exactly twice, i.e., $\Sigma$ consists of distinct singular values.     $\Sigma_b$  represents those  small singular values, which, due to   rounding errors,   are distinct and the order may be  disturbed, as well as    the sigular values with multiplicity larger than $2$. \color{black}Besides,    there still holds   $\rangemat{\chiQ{Y}} = \rangemat{U}$; this is assured for example by using MATLAB's \texttt{svd} even with very ill-conditioned $\chiQ{Y}$. \color{black}

Given $\chiQ{Y}=U\tilde{S}V^*$ as in \eqref{eq:USVOfchiY_incorrect}, if still generating $\qmat{H}$ by \eqref{eq:ComplexUSVToQuaternion}, then $\qmat{H}^*\qmat{H}$ will exhibit the form 
\[
	\qmat{H}^*\qmat{H} = \begin{bmatrix}
		I_{ (s-t)}& 0 \\ 0 & \times
	\end{bmatrix} \in\bbQ^{ s\times  s},~\text{``}\times\text{''} \in\bbQ^{t\times t},
\]
where ``$\times$'' is a matrix not equal to identity. Based on this observation, correcting $\qmat{H}$ can be excecuted as follows: first write $U=[U_g,U_b]$, with
    $U_g\in\mathbb C^{2m \times 2(s-t)}$ and $U_b\in\bbC^{2m\times 2t}$   corresponding to $S_g$ and $\Sigma_b$ respectively; $V=[V_g,V_b]$ is partitioned accordingly. From \eqref{eq:USVOfchiY_incorrect}  denote
	\begin{align}\label{eq:sudoSVD_partition_chiY}
		\chiQ{Y} := Y_g+Y_b,\quad Y_g=U_gS_gV_g^*,\quad Y_b = U_b\Sigma_bV_b^*.
	\end{align}
	Further partition $U_g$ as in \eqref{eq:sudoSVD_partition_U} and generate $\qmat{H}_g\in\bbQ^{m\times (s-t)}$ from $U_g$ as in   \eqref{eq:ComplexUSVToQuaternion}. Since   $S_g=\diag(\Sigma,\Sigma)$, Lemma \ref{lem:sudoSVD_Usigma_span_invariant_distict_sig} ensures that 
	\begin{align}\label{eq:sudoSVD_Hg_orth_preserve_range_Ug}
		\qmat{H}_g~  {\rm is~orthonormal~and~}~ \rangemat{\chilarger{\qmat{H}_g}} = \rangemat{U_g}.
	\end{align}
		 
For the bad part, %although $\Sigma_b$ is not structured,
 we can still select a representation basis from $U_b$. Write  $U_b =[u_1,\ldots,u_{2t}]\in\mathbb C^{2m\times 2t}$. Then:
	% partition $U_g$ as in \eqref{eq:sudoSVD_partition_U} and generate $\qmat{H}_g$ from $U_g$ as in   \eqref{eq:ComplexUSVToQuaternion}.  can still be used to generate $\qmat{H}_g\in\bbQ^{m\times (s-t)}$ from $U_g$, where $\qmat{H}_g$ is still orthonormal.  Denote $U_b=[u_1,\ldots,u_{2t}]\in\mathbb C^{2m\times 2t}$ the columns of $U$ corresponding to $\Sigma_b$; the following is crucial: 
 \begin{proposition} \label{prop:U_tildeb_span_2t}
  One can find  $U_{\tilde{b}}:=[u_{i_1},\ldots,u_{i_t}] \subset U_b=[u_1,\ldots,u_{2t}]$, such that 
\begin{align}\label{eq:U_tilde_epsilon_dim_2t}
\spanmat{U_{\tilde{b}}, \adjointJ \overline{ U_{\tilde{b}}  }  }  = \spanmat{U_b}= 2t.	
\end{align}
\end{proposition}
% Given this and let $\qmat{H}_{\tilde{\epsilon}}\in\bbQ^{m\times t}$ be constructed from  $U_{\tilde{\epsilon}}$. Then we will have $\mathcal{R}( [\qmat{H}_g,\qmat{H}_{\tilde{\epsilon}}] ) = \mathcal R(\qmat{Y})$.
The proof is given in appendix. Now construct $\qmat{H}_{\tilde{b}}\in\bbQ^{m\times t}$      such that $\chilarger{\qmat{H}_{\tilde b}}=[U_{\tilde b},\adjointJ\overline{U_{\tilde b}}]$. Then 
\begin{proposition}
	If $\qmat{Y}$ has full column rank, then $\mathcal{R}( [\qmat{H}_g,\qmat{H}_{\tilde{b}}] ) = \mathcal R(\qmat{Y})$.
\end{proposition}
\begin{proof}
	By  Lemma \ref{lem:range_equivalent_quaternion_chi_Q} it suffices to prove $\rangemat{ \chilarger{[\qmat{H}_g,\qmat{H}_{\tilde{b}}]}  }  = \rangemat{\chiQ{Y}}$. Note that     $\rangemat{ \chilarger{[\qmat{H}_g,\qmat{H}_{\tilde{b}}]}  } = \rangemat{ [ \chilarger{\qmat{H}_g},\chilarger{\qmat{H}_{\tilde{b}}} ] } = \rangemat{ \chilarger{\qmat{H}_g} } \otimes \rangemat{\chilarger{\qmat{H}_{\tilde{b}}} }$, where the second relation follows from  $\qmat{H}_g\perp \qmat{H}_{\tilde b}$; on the other hand, $\rangemat{\chiQ{Y}} = \rangemat{U} = \range{U_g }\otimes\rangemat{ U_{ b}  }$. 
\eqref{eq:sudoSVD_Hg_orth_preserve_range_Ug}, \eqref{eq:U_tilde_epsilon_dim_2t}, and $\chilarger{\qmat{H}_{\tilde b}}=[U_{\tilde b},\adjointJ\overline{U_{\tilde b}}]$    give the assertion. 
% This can be achieved by verifying that the direct sum   $\mathcal R(\qmat{H}_g)\oplus \mathcal R(\qmat{H}_{\tilde{ b}}) = \mathcal R(\qmat{Y})$,  as $\qmat{H}_g\perp \qmat{H}_{\tilde b}$. As \eqref{eq:sudoSVD_partition_chiY} holds, we only need to verify   $\mathcal R(\chilarger{\qmat{H}_{\tilde b}}) = \mathcal R(U_b)$, which by the construction of $\qmat{H}_{\tilde b}$ is true because $\mathcal R(\chilarger{\qmat{H}_{\tilde b}}) = \spanmat{(\qmat{H}_{\tilde b})_c, \adjointJ \overline{(\qmat{H}_{\tilde b})_c}  } = \spanmat{U_{\overline b},\adjointJ \overline{U_{\tilde b}}  }  = \spanmat{U_b}$ using \eqref{eq:U_tilde_epsilon_dim_2t}.
\end{proof}

Denote $\qmat{H}_{new} = [\qmat{H}_g,\qmat{H}_{\tilde b}]$, it remains to adjust $\qmat{H}_{new}$ such that it is orthonormal. Owing to \eqref{eq:sudoSVD_Hg_orth_preserve_range_Ug} and that $\qmat{H}_g\perp\qmat{H}_{\tilde b}$,  we only need to orthonormalize $\qmat{H}_{\tilde b}$, which in fact can be simultaneously done during the selection of $U_{\tilde b}$ using (modified) Gram-Schmidt orthogonalization \cite{golubMatrixComputations2013}. 

However, in case that %$t\lesssim s$ and
   $\qmat{Y}$ is too ill-conditioning ($\kappaq{Y}>10^{13}$),    the following process for finding $\qmat{H}_{\tilde b}$ is more accurate and efficient. The idea still resorts to complex SVD. Denote $\qmat{H}_b\in\bbQ^{m\times 2t}$  such that $\chilarger{\qmat{H}_b} = [U_b,\adjointJ\overline{U_b}]$  ($\qmat{H}_b$ needs not be explicitly constructed). 
By   Lemma \ref{lem:eigenvectorsSpace} and the definition of $U_b$,   $\chilarger{\qmat{H}_b}$ spans the same invariant subspace as $U_b$, i.e., % Lemma \ref{lem:eigenvectorsSpace} ensures that 
$\rangemat{\chilarger{\qmat{H}_b} } = \rangemat{U_b}$. %, which means that $\rangemat{\qmat{H}_b}$ has dimension $t$. 
   Additionally, let
\begin{align}\label{eq:sudoSVD_Hbdiagf}
\qmat{H}_{b} \leftarrow \qmat{H}_b \diag(f),~ f\in\mathbb R^{2t}~{\rm with}~ f_j \sim{\rm Uniform}(0,1)+1,~j=1,\ldots,2t. 	
\end{align}
This does not change the range of $\qmat{H}_b$ and empirically, multiplying $\diag(f)$ avoids $\qmat{H}_b$ to have duplicated singular values. In this case, applying   complex SVD to $\chilarger{\qmat{H}_b}\in \bbC^{2m\times 4t}$ has the form:
\begin{align}\label{eq:sudoSVD_svd_chi_Hb}
	\chilarger{\qmat{H}_b} = [U_{h1},U_{h2},U_{hb}] \diag(\Sigma_h,\Sigma_h,0) V_h^* , ~\Sigma_h\in\mathbb R^{t\times t}, ~U_{h1}\in\bbC^{2m\times  t},
\end{align}
where $\Sigma_h$ are the positive distinct singular values of $\qmat{H}_b$. Thus \eqref{eq:sudoSVD_svd_chi_Hb}	  reduces to the case of \eqref{eq:USVOfchiY_incorrect}. Note that $\rangemat{[U_{h1},U_{h2}]} = \rangemat{\chilarger{\qmat{H}_b} }=\rangemat{U_b}$. 
 We thus       construct the new $\qmat{H}_{\tilde b}\in\bbQ^{m\times t}$ from   $U_{h1}$ such that $\chilarger{\qmat{H}_{\tilde{ b}}} = [U_{h1},\adjointJ\overline{U_{h1}}]$.   By  Lemma \ref{lem:sudoSVD_Usigma_span_invariant_distict_sig},  $\qmat{H}_{\tilde b}$ is orthonormal and   $\rangemat{\chilarger{\qmat{H}_{\tilde b}}} =  \spanmat{U_b}$. Denote  
$\qmat{H}_{new}=[\qmat{H}_g,\qmat{H}_{\tilde b}] 	$. Finally we summarize that
% \begin{proposition}
% 	  $\rangemat{\qmat{H}_{new}}=\rangemat{\qmat{Y}}$ and $\qmat{H}_{new}$ is orthonormal.
% \end{proposition} 
\begin{align*}
	\rangemat{\qmat{H}_{new}}=\rangemat{\qmat{Y}}~{\rm and}~\qmat{H}_{new}~{\rm is~orthonormal}.
\end{align*} 

The whole computation is summarized in Algorithm \ref{alg:pseudoSVD}.

	% The pseudo-SVD algorithm can naturally extend to a Quaternion SVD algorithm dealing with data in real applications. For a given quaternion matix $A\in\mathbb{Q}^{m\times n}$ We can do SVD for its complex representation: 
	%  \begin{align}
	% 	\chiQ{A}:=U\Sigma V^*\\
	%  \end{align}
	% And SVD triplet $(\qmat{U},\qmat{\Sigma},\qmat{V})$ can be constructed by
	%  \begin{align}
	% 	\qmat{U}&=U(1:2:2s,:)\\
	% 	\qmat{\Sigma}&=\Sigma(1:2:2s,1:2:2s) \\
	% 	\qmat{V}&=V(1:2:2s,:). 
	% \end{align}

	% \begin{remark}
	% 	In practice, the extremely ill-conditioned $Y$ indicates that our sketching size $s$ is too large. As parameters set in the parameter settings of HMT \cite{FindingStructureHalko}, the oversampling parameter $s$ is only slightly larger than the target rank $r$. In practical applications of fixed-rank approximation, we usually have a rough estimate of the "significant part" of the data.
	% \end{remark}

	\begin{algorithm}
		\caption{(pseudo-SVD) quaternion pseudo-SVD implementation}\label{alg:pseudoSVD}
		\begin{algorithmic}[1]
			\Require Sketch matrix $\qmat{Y}\in\bbQms$.%, integer $s$.
			\Ensure Quaternion rangefinder $\qmat{H}\in\bbQms$.
			% \State Draw a  test matrix $\bdOmega\in \bbQ^{n\times s}$ and compute $\qmat{Y}=\qmat{A}\bdOmega$.
			\State Compute complex SVD $[U,\sim,\sim]=\texttt{svd}(\chiQ{Y},0)$ and partition $U=[U_g,U_b]$ with $U_b\in\bbC^{2m\times 2t}$ corresponding to the ``bad'' part. 
			\State  Construct $\qmat{H}_g = H_0+H_1\bj$ from $U_g$ such that $\chilarger{\qmat{H}_g}=[U_g,\adjointJ\overline{U_g}]$.
			% \State if $t\ll s$, then construct $\qmat{H}_b$ from   \eqref{eq:sudoSVD_Hbdiagf}. Compute   $[U_h,\sim,\sim]=\texttt{svd}(\chilarger{\qmat{H}_b},0)$ with $U_h=[U_{h1},U_{h2},U_{hb}]$ as \eqref{eq:sudoSVD_svd_chi_Hb}, and construct $\qmat{H}_{\tilde b}$ such that $\chilarger{\qmat{H}_{\tilde b}}=[U_{h1},\adjointJ\overline{U_{h1}}]$.  
\If{$t\lesssim s$}
\State Construct $\qmat{H}_b$ from   \eqref{eq:sudoSVD_Hbdiagf}. Compute   $[U_h,\sim,\sim]=\texttt{svd}(\chilarger{\qmat{H}_b},0)$ with $U_h=[U_{h1},U_{h2},U_{hb}]$ as \eqref{eq:sudoSVD_svd_chi_Hb}, and construct $\qmat{H}_{\tilde b}$ such that $\chilarger{\qmat{H}_{\tilde b}}=[U_{h1},\adjointJ\overline{U_{h1}}]$.
\Else 
\State Sequentially select linearly independent $u_{i_1},\ldots,u_{i_t}$ from $U_b$ and simultaneously do orthogonalization such that $u_{i_j}\perp [u_{i_1},\adjointJ\overline{u_{i_1}},\ldots,u_{i_{j-1}},\adjointJ\overline{u_{i_{j-1}}}]$, $j=2,\ldots,t$. Let $U_{\tilde b}=[u_{i_1,\ldots,u_{i_t}}]$ and construct $\qmat{H}_{\tilde b}$ such that $\chilarger{\qmat{H}_{\tilde b}}=[U_{\tilde b},\adjointJ\overline{U_{\tilde b}}]$.
% \State Initial $U_h=[U_b(:,1)]$. 
% \For{$i \gets 1~\text{to} ~t$}
% 	\State $k=0$.
% 	\State $\hat{u}=U_b(:,k)-[U_h,\adjointJ\overline{U_h}][U_h,\adjointJ\overline{U_h}]^*U_b(:,k)$
% 	\While{$\normSpectral{\hat{u}}>0$}
% 	\State $U_h=[U_h,\hat{u}/\normSpectral{\hat{u}}]$.
% 	\State $k\leftarrow k+1$.
% 	\EndWhile
% \EndFor
 \EndIf
 \State $\qmat{H}=[\qmat{H}_g,\qmat{H}_{\tilde b}]$. 
\end{algorithmic}
	\end{algorithm}
	
	\subsection{Comparsions} \label{sec:time_comparison_rangefinders}
% For $\qmat{Y}\in\bbQms$, the theoretically computational complexity of Pseudo-QR and Pseudo-SVD is of order $O(ms^2 + s^3)$ and $O(ms^2)$ respectively. Practically, they fit into different criterion. Pseudo-QR runs well when $\kappaq{Y}<10^8$. Pseudo-SVD on the other hand is more accurate when $\qmat{Y}$ is too ill-conditioned; however, it needs  to compute the SVD of $\chiQ{Y}$, which requires two times larger the memory than Pseudo-QR. %In terms of time, both are fast.   
%  In what follows, we will compare them with other   rangefinders. 

% Next, we compare our rangefinders with \texttt{qr} in QTFM, the structure-preserving Quaternion Householder QR (QHQR) \cite{jiaNewRealStructurepreserving2018}, and the modified Gram-Schmidt QR(QMGS) \cite{wei2018quaternion}.
%    Figure \ref{fig:fiveAlgs} shows that our  rangefinders cost much less than the \texttt{qr} in QTFM and QHQR. The difference in their time costs can reach two to three orders of magnitude and increase with large sketch size from $600$ to $2000$. As the scale of the sketches continues to increase, figure \ref{fig:threeAlgs} shows that our two rangefinder perform better than QMGS significantly, when \texttt{qr} in QTFM and $QHQR$ cost too much time. At last, figure \ref{fig:twoAlgs} indicates that pseudo-QR works faster than pseudo-SVD slightly when data size is over than $10^4$.The main reason is that compact \texttt{qr} is more efficient than compact SVD  in MATLAB.     

For a sketch \( \mathbf{Y} \in \mathbb{Q}^{m \times s} \), the   computational complexity of   Pseudo-QR and Pseudo-SVD   is %\( O (ms^2 + s^3) \) and
 \( O(ms^2) \) (assuming $s\ll m$). In practice, these algorithms are suited to different criteria.   Pseudo-QR   performs well  when the condition number \( \kappa(\mathbf{Y}) < 10^8 \). Conversely,   Pseudo-SVD   is more accurate for highly ill-conditioned sketch; however, it necessitates to compute the SVD of $\chiQ{Y}$, which demands   twice the memory   compared to   Pseudo-QR. % In terms of computational speed, both algorithms are considered efficient.

Subsequently, we will juxtapose these algorithms with other rangefinders, including        QTFM's \texttt{qr}, the structure-preserving  quaternion Householder QR (QHQR)    \cite{jiaNewRealStructurepreserving2018}, and the structure-preserving modified Gram-Schmidt   (QMGS)   \cite{wei2018quaternion}.
Fig. \ref{fig:fiveAlgs} illustrates that our   rangefinders exhibit significantly lower computational costs compared to        QTFM's \texttt{qr} and   QHQR. The disparity in their time complexity can span two to three orders of magnitude, escalating with an increase in the row size of the sketches from \( 600 \) to \( 2000 \). As the size continues to grow, Fig. \ref{fig:threeAlgs} demonstrates that our   rangefinders outperform   QMGS   markedly, while        QTFM's \texttt{qr} and QHQR become prohibitively time-consuming. Finally, Fig. \ref{fig:twoAlgs} reveals that   Pseudo-QR   operates faster than   Pseudo-SVD   when $m$ exceeds \( 10^4 \). 

We then fix  the size of $\qmat{Y}$ to be $1000\times 200$ while vary   $\kappaq{Y}$ and compare their precision as illustrated in Fig. \ref{fig:ComparePrecision}. In terms of the condition number of the rangefinder, all the rangefinders perform well when $\kappaq{Y}<10^{16}$, and   the orthonormal ones (Pseudo-SVD, QMGS, QHQR) keep their orthogonality. In terms of the range precision $\normF{\qmat{H}\qmat{H}^\dagger - \qmat{Y}\qmat{Y}^\dagger }$,     Pseudo-SVD and QMGS outperform  the competetors, while Pesudo-QR is still valuable when $\kappaq{Y}<10^8$.  QHQR performs the worst when $\kappaq{Y}$ increases.

% The primary rationale for this is the superior efficiency of the   QR   over  SVD within the MATLAB environment.
  
   % Thus, previous quaternion rangefinder may not suitable when data size exceeds $5e3$.
   % There are both advantages and limitations of the two algorithms.  Pseudo-QR cost less memory and it works more efficiently when size of sketch is too large. But it has limitation about input data. Original condition number should not be larger than $1e-8$, which may affect the precision of linear equations solver. Pseudo-SVD can compute a strictly orthogonal rangefinder with high-probability. It is more efficient when sketch size $s$ is no more than $1e3$ in experience. However, pseudo-SVD utilize the full complex representation of quaternion matrix to compute SVD. Thus, error out of memory may occur when sketch size is too large.
	% \begin{figure}
	% 	\centering
	% 	% \begin{subfigure}{0.3\textwidth}
	% 	% 	\includegraphics[width=\textwidth]{CompareTime_5.eps}
	% 	% 	\caption{m=5n}
	% 	% 	\label{fig:Timem=5n}
	% 	% \end{subfigure}
	% 	% \hfill
	% 	\begin{subfigure}{0.4\textwidth}
	% 		\includegraphics[width=\textwidth]{CompareTime_10.eps}
	% 		\caption{m=10n}
	% 		\label{fig:Timem=10n}
	% 	\end{subfigure}
	% 	% \hfill
	% 	% \begin{subfigure}{0.3\textwidth}
	% 	% 	\includegraphics[width=\textwidth]{CompareTime_20.eps}
	% 	% 	\caption{m=20n}
	% 	% 	\label{fig:Timem=20n}
	% 	% \end{subfigure}
	% 	\caption{Rangefinder time cost in matrix size setting:  $m\in[300,5000]$ with n=0.2m, n=0.1m and n=0.05m respectively.}
	% 	\label{fig:TimeCostOfRangeFinder}
	% \end{figure}
	
	\begin{figure}
		\centering
		% \begin{subfigure}{0.3\textwidth}
		% 	\includegraphics[width=\textwidth]{CompareTime_5.eps}
		% 	\caption{m=5n}
		% 	\label{fig:Timem=5n}
		% \end{subfigure}
		% \hfill
		\begin{subfigure}{0.3\textwidth}
			\includegraphics[width=\textwidth]{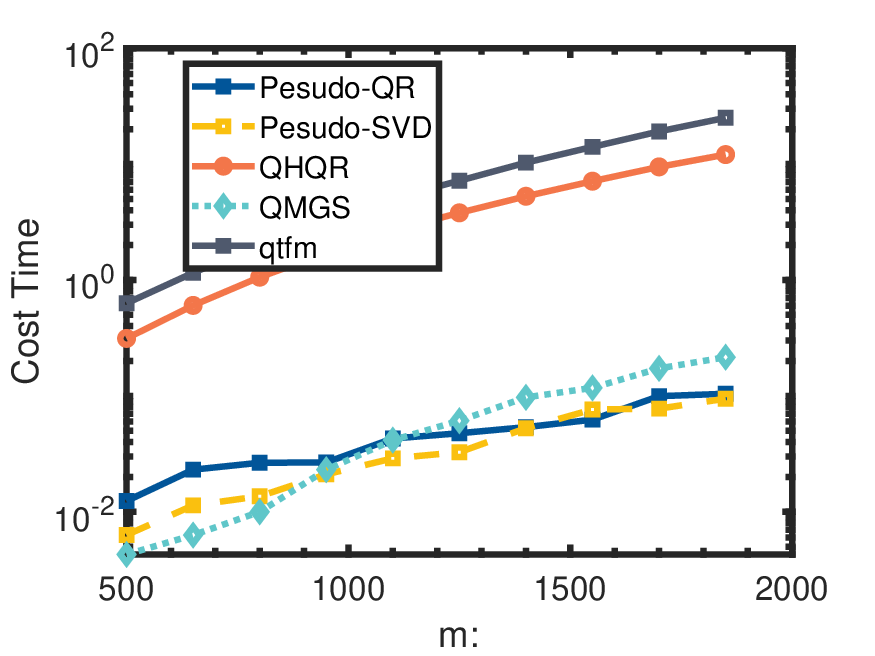}
			\caption{All rangefinders}
			\label{fig:fiveAlgs}
		\end{subfigure}
		\hfill
		\begin{subfigure}{0.3\textwidth}
			\includegraphics[width=\textwidth]{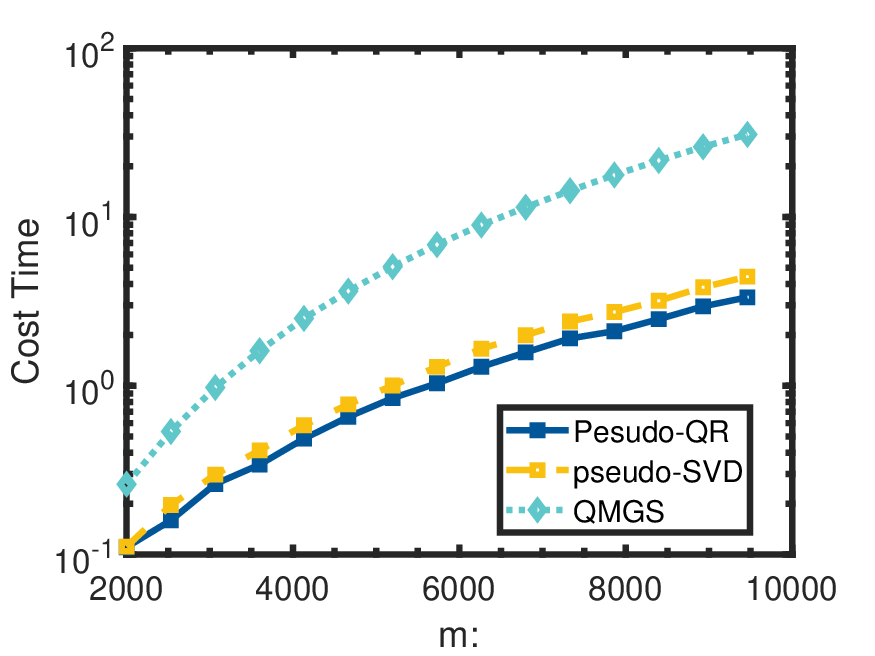}
			\caption{QMGS, our rangefinders}
			\label{fig:threeAlgs}
		\end{subfigure}
		\hfill
		\begin{subfigure}{0.3\textwidth}
			\includegraphics[width=\textwidth]{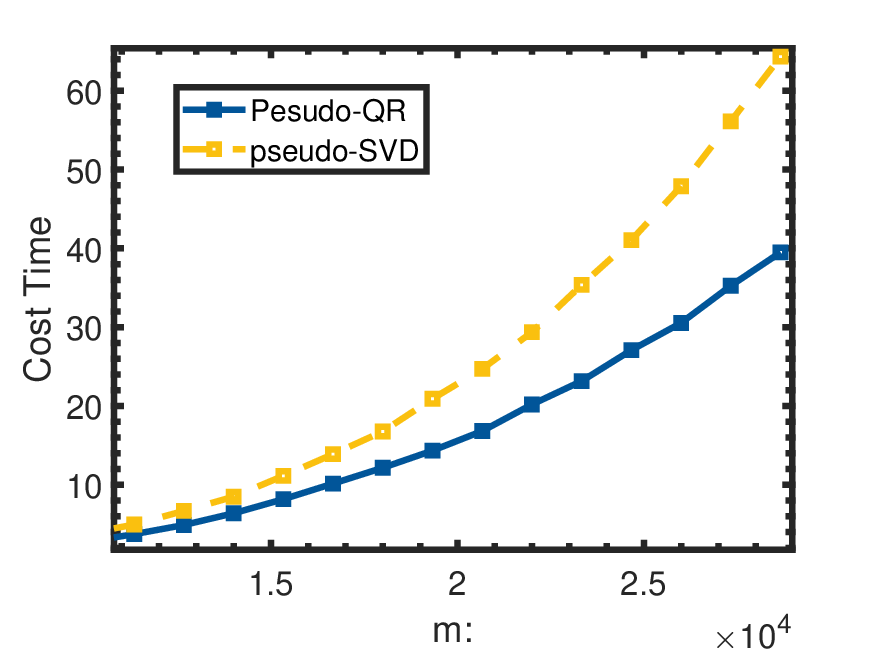}
			\caption{our rangefidners}
			\label{fig:twoAlgs}
		\end{subfigure}
		% \hfill
		% \begin{subfigure}{0.3\textwidth}
		% 	\includegraphics[width=\textwidth]{CompareTime_20.eps}
		% 	\caption{m=20n}
		% 	\label{fig:Timem=20n}
		% \end{subfigure}
		\caption{\small Running Time when $m=20s$: $x$-axis is $m$ and $y$-axis is the running time with log  scale (the first two subfigures).}
		\label{fig:CompareTime}
	\end{figure} 
	\begin{figure}
		\centering
		% \begin{subfigure}{0.3\textwidth}
		% 	\includegraphics[width=\textwidth]{CompareTime_5.eps}
		% 	\caption{m=5n}
		% 	\label{fig:Timem=5n}
		% \end{subfigure}
		% \hfill
		\begin{subfigure}{0.45\textwidth}
			\includegraphics[width=\textwidth]{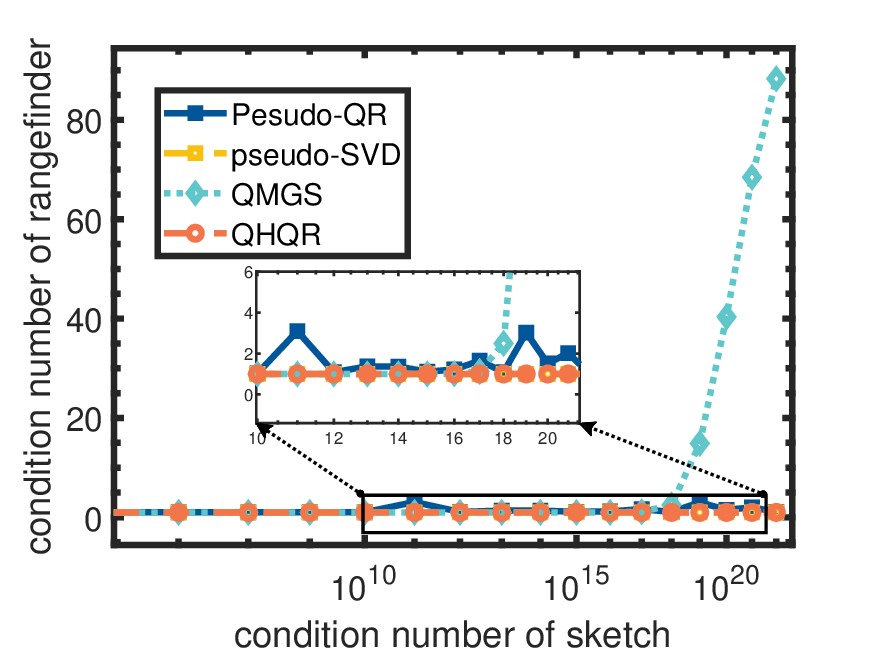}
			\caption{$y$-axis: condition number of the rangefinder}
			\label{fig:conditionThree}
		\end{subfigure}
		\hfill
		\begin{subfigure}{0.45\textwidth}
			\includegraphics[width=\textwidth]{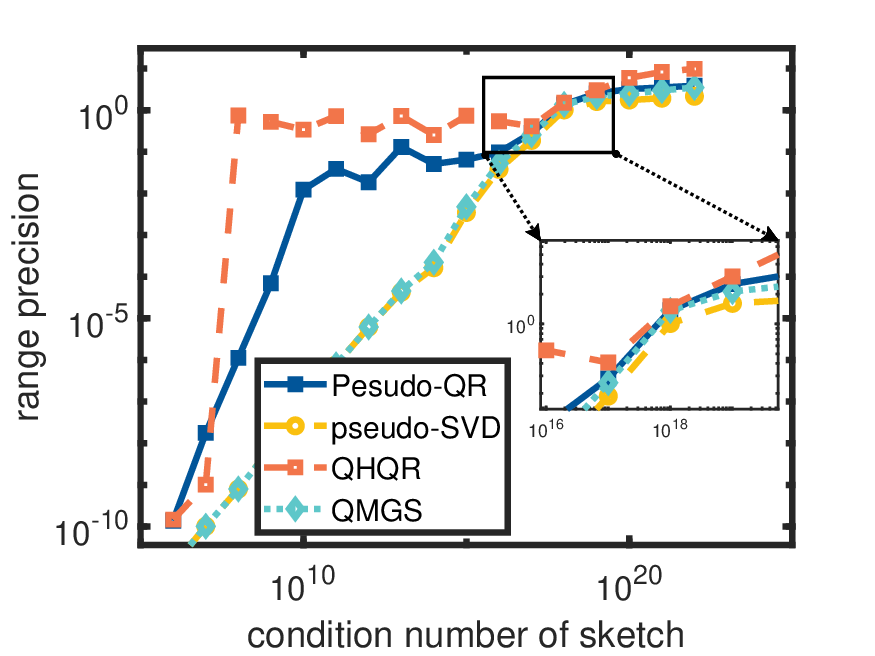}
			\caption{$y$-axis: range Precision}
			\label{fig:rangePrecisionThree}
		\end{subfigure}
		% \hfill
		% \begin{subfigure}{0.3\textwidth}
		% 	\includegraphics[width=\textwidth]{CompareTime_20.eps}
		% 	\caption{m=20n}
		% 	\label{fig:Timem=20n}
		% \end{subfigure}
		\caption{\small $m=1000$, $n=200$, $\kappaq{Y}$ from $1e6$ to $1e22$.}
		\label{fig:ComparePrecision}
	\end{figure}

\subsection{Remarks} \label{sect:remarks_rangefinder}
Very recently, based on randomized subspace embeddings \cite{martinsson2020RandomizedNumerical,nakatsukasa2022FastAccurate}, efficient randomized orthogonalization techniques  have been proposed  in the real domain \cite{balabanov2022randomized,balabanov2022RandomizedCholesky}, which guarantee to generate  a non-orthonormal yet well-conditioned matrix with high probability. The design of our non-orthonormal rangefinders is in the same spirit of   \cite{balabanov2022randomized,balabanov2022RandomizedCholesky}, while note that our computations are deterministic (if $\qmat{Y}$ is given). It is feasible to combine  Pseudo-QR, Pseudo-SVD, and any  QB approximation   with     randomized embedding techniques to yield possibly more practical non-orthonormal rangefinders. However, a fast deterministic quaternion QB approximation   still serves as a backbone.

	\section{One-Pass Algorithm}\label{sec:one-pass-alg-error-anal}
	In this section, we consider     the one-pass randomized   algorithm proposed by Tropp et al. \cite{Practical_Sketching_Algorithms_Tropp}  while with a   possibly non-orthonormal rangefinder (originally it only takes an orthonormal rangefinder). 
	%  In the QB stage, 
	% A left sketch to store sampling data and a right sketch to generate rangefinder are employed to obtain a primary projection onto the subspace $\mathcal{R}(\qmat{H})$. Then, the truncated SVD is performed on the $B$ matrix to obtain a better approximation accuracy. 
	% % More precise fixed-rank approximation from truncated QSVD, as the best rank-$r$ approximation of projection in randomized subspace, is a better approximation of origin data $A$ than randomized QB decomposition.
	 The one-pass algorithm can reduce storage cost and ensure linear update of streaming data, where the latter can save multiplication flops during sketching \cite{Practical_Sketching_Algorithms_Tropp,troppStreamingLowRankMatrix2019}.

	Our   theoretical result in Section \ref{subsec:ErrorAnalysis} shows that the accuracy loss of the truncation approximation is proportional to the condition number of the rangefinder $\kappa(\qmat{H})$. %Another is offering rough estimation of the approximation error in subsection \ref{subsec:ErrorAnalysis}.
	 Probabilistic bounds  with Gaussian and sub-Gaussian embeddings will be given in    Sections \ref{sec:Gaussian} and \ref{sec:sub-gaussian}.

	\subsection{The algorithm}
	For a given data matrix $\qmat{A}\in\mathbb{Q}^{m\times n}$ and the target rank $r$, the purpose is to find a rank-$r$ approximation. First 
	draw two random quaternion   $\bdOmega \in\mathbb{Q}^{n\times s}$ and $\boldsymbol{\bdPsi}\in\mathbb{Q}^{l\times m} $ independently where the sketch size satisfy: $r\leq s\leq l \ll \min\{m,n\}$. The  main information can be captured by two sketches:
	\begin{align}\label{eq:Y=AOmegaW=PsiA}
		\qmat{Y}=\qmat{A}\bdOmega\in\mathbb{Q}^{m\times s} \quad\text{and}\quad \qmat{W}=\bdPsi \qmat{A}\in\mathbb{Q}^{l\times n},
	\end{align}
	where $\qmat{Y}$ is used to generate the range   $\qmat{H}\in\bbQms$ by using  any range-preserving rangefinders such as   pesudo-QR or pesudo-SVD. %, quaternion QR, or QMGS \cite{wei2018quaternion}.
 Then, a rank-$s$   approximation of $\qmat{A}$ is given by	   
	\begin{align}\label{eq:reconstructionOfTwoSketch}
		\qmat{A}\approx\qmat{H}\qmat{X},\quad \qmat{X}=\left(\bdPsi \qmat{H}\right)^\dagger \qmat{W}\in\bbQ^{s\times n}.
	\end{align}
	Finally, a truncated QSVD is applied to $\qmat{X}$ to further obtain the final  rank-$r$ approximation. 
	In the whole algorithm, only $\qmat{Y}$, $\qmat{W}$, and $\bdPsi$ are required, which means that $\qmat{A}$ is not exposed to the   process. This is suitable for the scenorio that $\qmat{A}$ is too large to be read in memory. The pseudocode is given in Algorithms \ref{alg:S2S} and \ref{alg:FixedRankTwoSketches}. \color{black} Note that when $\bdPsi$ takes the identity matrix and $\qmat{H}$ is orthonormal,   the algorithm is exactly the randomized QSVD \cite{RandomizedQSVD}. That is to say, the one-pass algorithm takes an additional sketch $\qmat{W}=\bdPsi\qmat{A}$ than randomized QSVD, thereby trading accuracy for efficiency as well as the usage in the limit storage scenorio. 
	\color{black}
	\begin{algorithm}
		\caption{\textbf{QB Stage}}\label{alg:S2S}
		\begin{algorithmic}[1]
			\Require Quaternion sketch $\qmat{Y}\in\mathbb{Q}^{m\times s},\qmat{W}\in\mathbb{Q}^{l\times n}$, test matrices $\bdPsi\in\mathbb{Q}^{l\times m}$,
			\Ensure Rank-$s$ approximation of the form $\hat{\qmat{A}}=\qmat{H}\qmat{X}\in\mathbb{Q}^{m\times n}$ with   $\qmat{H}\in\mathbb{Q}^{m\times s}$ and $\qmat{X}\in\mathbb{Q}^{s\times n}$
			\State $\qmat{H}\leftarrow \mathcal{F}(\qmat{Y})$ \Comment{$\mathcal{F}$ is a  range-preserving rangefinder map and $\qmat{H}$ is well-conditioned.}
			\State $\qmat{X}\leftarrow\left(\bdPsi \qmat{H}\right)\backslash \qmat{W}$ \Comment{Solve   overdetermined quaternion linear equations using Algorithm \ref{alg:qequation_cplx_solver}.}
			\State return $\left(\qmat{H},\qmat{X}\right)$
		\end{algorithmic}
	\end{algorithm}
	% \begin{algorithm}
		%     \caption{Accelerated Two-Sketches Quaternion Low-Rank approximation(AS2S)}\label{alg:AS2S}
		%     \begin{algorithmic}[1]
			%         \Require Quaternion sketch $Y\in\mathbb{Q}^{m\times s},W\in\mathbb{Q}^{l\times n}$, test matrices $\bdPsi\in\mathbb{Q}^{l\times m}$,
			%         \Ensure Rank-s approximation in the form $\hat{A}=QX\in\mathbb{Q}^{m\times n}$ with orthonormal $Q\in\mathbb{Q}^{m\times s}$ and $X\in\mathbb{Q}^{s\times n}$
			%         %\State $Q\leftarrow orth(Y)$ \Comment{Q is from structure-preserving Quaternion QR algorithm}
			%         \State $X\leftarrow\left(\bdPsi Y\right)\backslash W$ \Comment{Solving an overdetermined quaternion least square problem}
			%         \State return $\left(Y,X\right)$
			%     \end{algorithmic}
		% \end{algorithm}
	
	% Moreover, a fixed rank approximation algorithm is shown in Algorithm
	\begin{algorithm}
		\caption{\textbf{Truncation Stage}}\label{alg:FixedRankTwoSketches}
		\begin{algorithmic}[1]
			\Require $(\qmat{H},\qmat{X})$ from Algorithm \ref{alg:S2S}, $\qmat{H}\in\bbQms,~\qmat{X}\in\bbQ^{s\times n}$,
			\Ensure Rank-$r$ approximation $\hat{\qmat{A}}=\qmat{H}\Sigma \qmat{V}^*$ with  $\qmat{H}\in\mathbb{Q}^{m\times r}$, orthonormal $\qmat{V}\in\mathbb{Q}^{n\times r}$, and real diagonal matrix $\Sigma\in\mathbb{R}^{r\times r}$.
			\State $(\qmat{U},\Sigma,\qmat{V})={\rm QSVD}(\qmat{X})$
			\State $\Sigma=\Sigma(1:r,1:r)$, $\qmat{U}=\qmat{U}(:,1:r)$, $\qmat{V}=\qmat{V}(:,1:r)$
			\State $\qmat{H}=\qmat{H} \qmat{U}$
			\State return $\left(\qmat{H},\Sigma, \qmat{V}\right)$
		\end{algorithmic}
	\end{algorithm}

	% \begin{remark}
	% To solve  $\left(\bdPsi \qmat{H}\right)\backslash \qmat{W}$, one can use Algorithm \ref{alg:QLEQ}. QSVD in the truncation stage can be any quaternion SVD solvers.% \cite{li2016RealStructurepreserving,sangwine2006QuaternionSingular,jia2019LanczosMethod,li2014FastStructurepreserving} or directly computing the SVD of $\chiQ{X}$.
	% \end{remark}
	\begin{remark}
Even though $\qmat{H}$ may be non-orthonormal, empirically, we find that a truncation stage is still helpful in improving the recovery accuracy than only doing a QB approximation. This is the same as the orthonormal case \cite{Practical_Sketching_Algorithms_Tropp,FindingStructureHalko,RandomizedQSVD}.  
		% Despite rangefinder $H$ is not orthonormal, apply the partial SVD to $X$ is reliable to avoid the approximation error degrade with incurring additional errors no more than tail energy, which discussed in theorem \ref{thm:truncationError} in detail. In practical, proper oversampling $p=s-r$ play an important role on more accurate fixed-rank approximation, which usually bring a significant improvement. This also verified in our numerical experiment.
	\end{remark}

	\subsection{Deterministic error}\label{subsec:ErrorAnalysis}
	% This part we will give some results of theoretical analysis of our algorithms. First, we will display a theorem to give a general evaluation of expected error of algorithm \ref{alg:S2S}. Error analysis of two fixed rank approximation algorithms are not analyzed in detail here. Finally, we will complete the proof of algorithm \ref{alg:S2S}.
	% This subsection answer the last question section \ref{sec:motivation}. We give the following results:
	% \begin{enumerate}
	% 	\item The error bound of approximation in stage A is same to orthonormal rangefinder implementation.
	% 	\item Truncation error generated in the stage B can be controlled by $\kappa(H)$, error in stage A and tail energy $\tau(A)$.  
	% \end{enumerate}
 This subsection is   concerned with the deterministic error  without assuming the distribution of the test matrices at first. Probabilistic bounds will be detailed in the next two subsections. We do not assume a specific rangefinder $\qmat{H}$, but
	the following basic requirements on   $\qmat{H}$   and the sketch size are made througout this subsection:
	\begin{align}
		 \rangemat{\qmat{H}} &=\rangemat{\qmat{Y}},~{\rm and}~\qmat{H}~{\rm has~full~column~rank}; \label{eq:rangefinder_basic_requirement}\\
		 r \leq s &\leq l\leq \min\{m,n\}.\label{eq:sketch_size_basic_requirement}
	\end{align}
To analyze the approximation error with a general $\qmat{H}$, the idea is to use QB decomposition as a bridge such that existing   error analysis can be applied. One can represent $\qmat{H}$ as:
\begin{align}\label{eq:QB_representation_H}
	\qmat{H} = \qmat{Q}\qmat{B} , ~\qmat{Q}\in\bbQ^{m\times s},~\qmat{Q}^*\qmat{Q}=I_s.
\end{align} Under \eqref{eq:rangefinder_basic_requirement}, $\rangemat{\qmat{Q}}=\rangemat{\qmat{H}}=\rangemat{\qmat{Y}}$   and $\qmat{B}\in\bbQ^{s\times s}$ is invertible. Here  \eqref{eq:QB_representation_H} can be thin QR, compact QSVD, or polar decomposition.
% , where $\qmat{B}=\qmat{R}$ with $\qmat{R}$   upper triangular, or   compact SVD, where $\qmat{B}=S\qmat{V}^*$ with $S$ real diagonal and $\qmat{V}$ orthonormal.

	% In the remainder of this subsection, denote $\qmat{H}=\qmat{Q}\qmat{B}$ as any QB decomposition of $\qmat{H}$ (e.g., quaternion QR or polar decomposition), where $\qmat{Q}\in\bbQ^{m\times s}$ is partially orthonormal, and $\qmat{B}\in\bbQ^{s\times s}$. If $\qmat{H}$ has full column rank, then $\qmat{B}$ is invertible, and  $\mathcal R(\qmat{Q})= \mathcal R(\qmat{H})=\mathcal R(\qmat{Y})$.     $\qmat{Q}$ plays an important role in our analysis.

	We     introduce notations used in this subsection. Define the partially orthonormal matrix $\qmat{Q}_{\bot}\in\bbQ^{m\times (m-s)}$ with $\qmat{Q}_{\bot}\qmat{Q}_{\bot}^* := I_m-\qmat{Q}\qmat{Q}^*$  such that $\rangemat{\qmat{Q}}\perp\rangemat{\qmat{Q}_{\bot}}$. 
	% \begin{align}
	% 	\qmat{H}_{\bot}\qmat{H}_{\bot}^* := I_m-\qmat{H}\qmat{H}^\dagger.
	% \end{align}
	% The columns of $H_{\bot}$ is orthonormal. 
	Denote 
	\begin{align}\label{eq:LambdaNotation}
		\bdPsi_1=\bdPsi \qmat{Q}_{\bot}\in\mathbb{Q}^{l\times \left(m-s\right)},\quad \bdPsi_2=\bdPsi \qmat{Q}\in\mathbb{Q}^{l\times s}. % \color{red}should~it~be~\bdPsi_2=\bdPsi\qmat{Q}?\color{black}
	\end{align}	
	Let the QSVD of $\qmat{A}\in\bbQmn$ be
	 \begin{align}\label{eq:A_QSVD} \small
		\qmat{A}=\qmat{U}\Sigma \qmat{V}^*=\begin{bmatrix}\qmat{U}_1&\qmat{U}_2\end{bmatrix}\left[\begin{array}{cc}\Sigma_1&0\\0&\Sigma_2\end{array}\right]\left[\begin{array}{c}\qmat{V}_1^*\\\qmat{V}_2^*\end{array}\right],~\Sigma_1\in\mathbb{R}^{r\times r}, ~ \qmat{U}_1\in\bbQ^{m\times r},~ \qmat{V}_1\in\mathbb{Q}^{n\times r};
	\end{align}
	and denoote:
	\begin{align}\label{eq:GammaOmegaNotation}
		\bdOmega_1=\qmat{V}_1^* \bdOmega\in\mathbb{Q}^{r\times s},&\quad \bdOmega_2=\qmat{V}_2^*\bdOmega\in\mathbb{Q}^{\left(n-r\right)\times s}.
	\end{align}

	% The error of the QB stage is provided.  
	\begin{theorem}[QB error]\label{thm:errorAnalysisOfTwoSketch}
		Let $(\qmat{H},\qmat{X})$   be    generated by Algorithm \ref{alg:S2S} and \eqref{eq:rangefinder_basic_requirement} holds. With notations in \eqref{eq:LambdaNotation} and \eqref{eq:GammaOmegaNotation}, if $\bdOmega_1^*$ and $\bdPsi_2$ have full column rank, 
		we have the following Frobenius error estimation,
		% \begin{align}
		% 	\|\qmat{H}\qmat{X}-\qmat{A}\|_a&\leq \normA{\qmat{A}-\qmat{Q}\qmat{Q}^*\qmat{A} } + \normA{\bdPsi_2^\dagger\bdPsi_1 (\qmat{Q}_\bot^* \qmat{A} )  }	\label{eq:thm_QB_error_for_Guassian} \\
		% 	&\leq\left(1+\normSpectral{\bdPsi_2^\dagger}\normSpectral{\bdPsi_1}\right)\left(1+\normSpectral{\bdOmega_2}\normSpectral{\bdOmega_1^\dagger}\right)\|\Sigma_2\|_a; \label{eq:thm_QB_error_for_sub_Gaussian}
		% \end{align}
		\begin{align}
			\normF{\qmat{A}-\qmat{H}\qmat{X}}^2 &= \|\qmat{A}-\qmat{Q}\qmat{Q}^*\qmat{A}\|_F^2 + \|\bdPsi_2^\dagger \bdPsi_1\left(\qmat{Q}_{\bot}^*\qmat{A}\right)\|_F^2 \label{eq:thm_QB_error_for_Guassian} 
			% \\&\leq \normFSquare{\qmat{A}-\qmat{Q}\qmat{Q}^*\qmat{A}}+\normFSquare{\bdPsi_2^\dagger \bdPsi_1 \left(\qmat{A}-\qmat{Q}\qmat{Q}^*\qmat{A}\right)}
			% \label{eq:thm_QB_error_for_sub_Gaussian}
			\\ &\leq  \left(1+\normSpectral{\bdPsi_2^\dagger\bdPsi_1}^2\right)\left(1+\normSpectral{\bdOmega_2 \bdOmega_1^\dagger}^2\right)\|\Sigma_2\|^2_F. \label{eq:thm_QB_error_for_sub_Gaussian}
		\end{align}
		% and
		% \begin{align}
		% 	\|\qmat{Q}_\bot^*\qmat{A}\|_F^2=\|\qmat{A}-\qmat{Q}\qmat{Q}^*\qmat{A}\|_F^2;
		% \end{align}
		% in particular,
		% \begin{align}\label{eq:thm_QB_error_for_Gaussian_Fnorm}
		% 	\normF{\qmat{A}-\qmat{H}\qmat{X}}^2= \|\qmat{A}-\qmat{Q}\qmat{Q}^*\qmat{A}\|_F^2 + \|\bdPsi_2^\dagger \bdPsi_1\left(\qmat{Q}_{\bot}^*\qmat{A}\right)\|_F^2 .
		% \end{align}
		% and 
		% \begin{align}
		% 	\normFSquare{\qmat{A}-\qmat{Q}\qmat{Q}^*\qmat{A}}\leq \normFSquare{\Sigma_2}+\normFSquare{\Sigma_2\bdOmega_2\bdOmega_1^\dagger}
		% \end{align}
	\end{theorem}
	\eqref{eq:thm_QB_error_for_sub_Gaussian} is a rough estimation, whose purpose is to demonstrate that   the QB error  is independent of $\kappaq{H}$. More refined probabilistic bounds will be given in Sections \ref{sec:Gaussian} and \ref{sec:sub-gaussian} based on \eqref{eq:thm_QB_error_for_Guassian}.

	Recall that $\qmat{X}=\left(\bdPsi \qmat{H}\right)^\dagger \bdPsi\qmat{A}$ in Algorithm \ref{alg:S2S}. Denote $\lfloor \qmat{X}\rfloor_r$ as the best rank-$r$ approximation of $\qmat{X}$. Analygous to the real/complex counterpart, $\lfloor \qmat{X}\rfloor_r$ is also given by the rank-$r$ turncated QSVD \cite{lebihan2004SingularValue}.  	The truncation error is estimated in the following theorem:
	\begin{theorem}[Truncation error]\label{thm:truncationError}
		Let $\qmat{H}$, $\qmat{X}$ be generated by Algorithm \ref{alg:S2S} with \eqref{eq:rangefinder_basic_requirement} hold, and $\lfloor \qmat{X}\rfloor_r$ be truncated by Algorithm \ref{alg:FixedRankTwoSketches}. Then
		\begin{align}
			\normF{\qmat{H}\lfloor \qmat{X}\rfloor_r-\qmat{H}\qmat{X}}\leq \kappa(\qmat{H})\left(\normF{\qmat{A}-\qmat{H}\qmat{X}}+ \|\Sigma_2\|_F  %\tau_r(\qmat{A}) 
			\right).
		\end{align}
	\end{theorem}
	
	To make the proofs of the theorems  clear, we devide them in a series of lemmas. 
	
	% \begin{theorem}\label{thm:errorAnalysisOfThreeSketch}
		%     Let $\hat{A}$ be low-rank approximation generated by algorithm \ref{alg:SimplestThreeSketchLowRankApprox}. If $\bdOmega_1$, $\Gamma_1$, $\Phi_1$, $\bdPsi_1$ have full column rank, we have the following estimation for the Frobenius error
		%     \begin{align}
			%         \normFSquare{\hat{A}-A}\leq\left(\xi^2+1\right)\left(1+\normSpectral{\left(\bdOmega_1^*\right)^\dagger}^2\normSpectral{\bdOmega_2^*}^2+\normSpectral{\Gamma_1^\dagger}^2\normSpectral{\Gamma_2}^2\right)\normFSquare{\Sigma_2}
			%     \end{align}
		%     where $\xi=\max\{\normSpectral{\Phi_1^\dagger}^2\normSpectral{\Phi_2}^2,\quad\normSpectral{\bdPsi_2^*}^2\normSpectral{(\bdPsi_1^*)^\dagger}^2\}$. 
		% \end{theorem}
	
	% Theorem \ref{thm:accelerateTheorem} and lemma \ref{cor:AcceleratedTheorem} indicate that the accelerated algorithms has the same theoretical error with origin case.
 
	\begin{lemma}\label{lem:ABdagger=BdaggerAdagger}
		Let $\qmat{A}\in\mathbb{Q}^{m\times n}$ and $\qmat{B}\in\mathbb{Q}^{n\times p}$; then $(\qmat{A}\qmat{B})^\dagger\neq \qmat{B}^\dagger \qmat{A}^\dagger$ in general. However, we have the following   sufficient conditions for  $(\qmat{A}\qmat{B})^\dagger= \qmat{B}^\dagger \qmat{A}^\dagger$:
		\begin{enumerate}
			\item $\qmat{A}$ has orthonormal columns ($\qmat{A}^*\qmat{A}=\qmat{A}^\dagger \qmat{A}=I$) or
			\item $\qmat{B}$ has orthonormal rows ($\qmat{B}\qmat{B}^*=\qmat{B}\qmat{B}^\dagger=I$) or
			\item $\qmat{A}$ has full column rank ($\qmat{A}^\dagger \qmat{A}=I$) and $\qmat{B}$ has full row rank ($\qmat{B}\qmat{B}^\dagger=I$) or
			\item $\qmat{B}=\qmat{A}^*$ or
			  $\qmat{B}=\qmat{A}^\dagger$.
		\end{enumerate}
	\end{lemma}
	\begin{proof}
		By Lemma \ref{lem:AdaggerComplex}, it suffices to prove $\left(\chi_{\qmat{A}}\chi_{\qmat{B}}\right)^\dagger=\chi_{\qmat{B}}^\dagger\chi_{\qmat{A}}^\dagger$. %By this way, we can transform the problem into the complex case. 
		For two complex matrices $E$ and $F$,  \cite{grevilleNoteGeneralizedInverse1966}  proved that $(EF)^\dagger= F^\dagger E^\dagger$   holds if and only if: 
		\begin{align}\label{eq:1966GI}
			E^\dagger EFF^*E^*EFF^\dagger=FF^*E^*E.
		\end{align}   
		Let $E={\chi}_\qmat{A}$, $F=\chi_\qmat{B}$. Using the properties of generalized inverse and complex representation of quaternion matrices, one can check that each condition  in this lemma   can make (\ref{eq:1966GI}) hold.% which complete our proof.
	\end{proof}

	\begin{lemma}\label{lem:HToQ}
		Let $\qmat{H}\in \bbQ^{m\times s}~(m\geq s)$ have full column rank,  $\qmat{Q}\in\mathbb{Q}^{m\times s}$ is given by \eqref{eq:QB_representation_H}, and $\bdPsi$ is an arbitrary $l\times m~(s\leq l\leq m)$ quaternion matrix with full row rank. Then we have:
		\begin{align}
			\qmat{H}\qmat{X}=\qmat{H}\left(\bdPsi \qmat{H}\right)^\dagger \bdPsi \qmat{A}&=\qmat{Q}(\bdPsi \qmat{Q})^\dagger\bdPsi \qmat{A}\label{eq:Leftaccelerate}.
		\end{align}
	\end{lemma}
	\begin{proof}
		Let $\qmat{H}=\qmat{Q}\qmat{B}$ where $\qmat{B}\in\bbQ^{s\times s}$ is invertible.   Then  
		\begin{align*}
			\qmat{H}\left(\bdPsi \qmat{H}\right)^\dagger \bdPsi \qmat{A}&=\qmat{Q}\qmat{B}(\bdPsi \qmat{Q}\qmat{B})^\dagger\bdPsi \qmat{A}\\
			\text{\tiny(Lemma \ref{lem:ABdagger=BdaggerAdagger}, point 3)}&=\qmat{Q}\qmat{B}\qmat{B}^{-1}(\bdPsi \qmat{Q})^\dagger\bdPsi \qmat{A}\\
			&=\qmat{Q}(\bdPsi \qmat{Q})^\dagger\bdPsi \qmat{A}.
		\end{align*}
		The second identity is from Lemma \ref{lem:ABdagger=BdaggerAdagger} because   $\bdPsi \qmat{Q}\in\bbQ^{l\times s}(l\geq s)$ has full column rank ($\bdPsi^*$ and $\qmat{Q}$ are both of full column rank) and $\qmat{B}$ is invertible.
	\end{proof}

	% \begin{lemma}\label{cor:AcceleratedTheorem}
		%     If we use the notation in algorithm \ref{alg:SimplestThreeSketchLowRankApprox} and \ref{alg:AcceleratedThreeSketchLowRankApprox}, we have the following equation:
		%     \begin{align}
			%         Y\left(\Phi Y\right)^\dagger \Phi A \bdPsi^* \left(\left(\bdPsi W^*\right)^\dagger\right)^*=Q\left(\Phi Q\right)^\dagger \Phi A \bdPsi^* \left(\left(\bdPsi P\right)^\dagger\right)^*
			%     \end{align}
		% \end{lemma}
	% The lemma is a direct result of theorem \ref{thm:accelerateTheorem}.

	% The accuracy lost in truncated SVD is easy to be evaluated. The result and deduction is similar to \cite[section 6]{Practical_Sketching_Algorithms_Tropp} and \cite[section 5.3]{troppStreamingLowRankMatrix2019} because that quaternion matrix has real singular values. Truncation error is always smaller than tail energy and usually not affect our parameter selection.
	
	%The $\qmat{Q}$ appearing in the following lemmas is just the $\qmat{Q}$ defined in Lemma \ref{lem:HToQ}. 
	 
		The following lemma is a trivial  quaternion version of \cite[Lemma A.4]{Practical_Sketching_Algorithms_Tropp}.
	\begin{lemma}\label{lem:X-Q^*A}
	 	Assume that $\bdPsi_2$ has full column rank; then 
		\begin{align*}
			\left(\bdPsi \qmat{Q}\right)^\dagger \bdPsi\qmat{A}-\qmat{Q}^*\qmat{A}=\bdPsi_2^\dagger \bdPsi_1\left(\qmat{Q}_\bot^*\qmat{A}\right).
		\end{align*}
	\end{lemma}
	
	% \begin{lemma}{(Decomposition of core matrix in \ref{eq:CoreC})}\label{lem:C-Q^*AP}
		%     Assume that the matrices $\Phi_1$ and $\bdPsi_1$ in \ref{eq:PHIPSINotaion} has full columns, then 
		%     \begin{align}
			%         C-Q^{*}AP &=\Phi_1^\dagger\Phi_2(Q_\bot^*AP)+(Q^*AP_\bot)\bdPsi_2^*(\bdPsi_1^\dagger)^* \\
			%         &+\Phi_1^\dagger\Phi_2(Q_\bot^*AP_\bot)\bdPsi_2^*(\bdPsi_1^\dagger)^*. 
			%     \end{align}
		% \end{lemma}
	% \begin{theorem}\label{thm:errorAnalysisAll}
		%     If $\hat{A}$ is generated by algorithm (\ref{alg:S2S}), $P$ is a basis of orthogonal complement space of $span(Q)$, QSVD of A is \begin{align*}
			%         A=U\Sigma V^*=U\left[\begin{array}{cc}\Sigma_1&0\\0&\Sigma_2\end{array}\right]\left[\begin{array}{c}V_1^*\\V_2^*\end{array}\right],\quad\Sigma_1\in\mathbb{R}^{k\times k},\quad V_1\in\mathbb{Q}^{n\times k},
			%     \end{align*}
		%     $\bdPsi_1=\bdPsi P$, $\bdPsi_2=\bdPsi Q$, $\bdOmega_1=V_1^* \bdOmega$, $\bdOmega_2=V_2^*\bdOmega$, and following assumptions are satisfied:
		%     \begin{enumerate}
			%         \item $k< s< l<\min\{m,n\}$, the specific conditions depend on the selection of the test matrix.
			%         \item $\bdOmega_1$ has full row rank, $\bdPsi_2$ has full column rank.
			
			%     \end{enumerate}
		%     We have the following estimation for the Frobenius error or spectral error
		%     \begin{align}
			%         \|\hat{A}-A\|_a\leq\left(1+\normSpectral{\bdPsi_2^\dagger}\normSpectral{\bdPsi_1}\right)\left(1+\normSpectral{\bdOmega_2}\normSpectral{\bdOmega_1^\dagger}\right)\|\Sigma_2\|_a
			%     \end{align}
		% \end{theorem}

The follow lemma comes from \cite[Section 4.3]{RandomizedQSVD}; see also \cite[Theorem 9.1]{FindingStructureHalko}.
	\begin{lemma}\label{lem:A-QQ^*A}
		With the notations in Algorithm \ref{alg:S2S} and in  \eqref{eq:GammaOmegaNotation}, and $\qmat{Q}$ is   in \eqref{eq:QB_representation_H}, if $\bdOmega_1$ has full row rank, then 
		\begin{align}
		  	\normF{\qmat{A}-\qmat{Q}\qmat{Q}^*\qmat{A}}^2 = \normF{\qmat{A}-\qmat{Y}\qmat{Y}^\dagger\qmat{A}}^2 &\leq\normA{\Sigma_2}^2+\normF{\Sigma_2\bdOmega_2 \bdOmega_1^\dagger}^2 \label{eq:lem:A-QQ^*A:1}\\
			&\leq \left(1+\normSpectral{ \bdOmega_2 \bdOmega_1^\dagger}^2\right)\normF{\Sigma_2}^2. \label{eq:lem:A-QQ^*A:2} 
		\end{align}
	\end{lemma}

	\begin{proof}[Proof of Theorem \ref{thm:errorAnalysisOfTwoSketch}]
		By Lemma \ref{lem:HToQ}, it holds that
		\begin{align}
		\normFSquare{\qmat{A}-\qmat{H}\qmat{X}} = 	\normFSquare{\qmat{A}-\qmat{H}\left(\bdPsi \qmat{H}\right)^\dagger \bdPsi \qmat{A}}=\normFSquare{\qmat{A}-\qmat{Q}(\bdPsi \qmat{Q})^\dagger\bdPsi \qmat{A}}.\label{eq:H2Q}
		\end{align}
	 Thus, it suffices to evaluate the right part of  \eqref{eq:H2Q}. By Pythagorean identity,
		\begin{align}
			\|\qmat{A}-\qmat{Q}(\bdPsi \qmat{Q})^\dagger\bdPsi \qmat{A}\|_F^2&= \|\qmat{A}-\qmat{Q}\qmat{Q}^*\qmat{A}+  \qmat{Q}\qmat{Q}^*\qmat{A}- \qmat{Q}(\bdPsi \qmat{Q})^\dagger\bdPsi \qmat{A}\|_F^2 \nonumber\\
			&= \|\qmat{A}-\qmat{Q}\qmat{Q}^*\qmat{A}\|_F^2+\|\qmat{Q}(\bdPsi \qmat{Q})^\dagger\bdPsi \qmat{A}-\qmat{Q}\qmat{Q}^*\qmat{A}\|_F^2 \nonumber\\
			&= \|\qmat{A}-\qmat{Q}\qmat{Q}^*\qmat{A}\|_F^2 + \|\bdPsi_2^\dagger \bdPsi_1\left(\qmat{Q}_{\bot}^*\qmat{A}\right)\|_F^2 , \label{eq:proof:qb_error:1}
		\end{align}
	where the  last equality uses    Lemma \ref{lem:X-Q^*A}. 
	Note that  
	$  %\label{eq:P*A=QQ*A-A}
			\|\qmat{Q}_\bot^*\qmat{A}\|_F=\|\qmat{Q}_\bot \qmat{Q}_\bot^*\qmat{A}\|_F=\|\qmat{A}-\qmat{Q}\qmat{Q}^*\qmat{A}\|_F
	$; thus $\|\bdPsi_2^\dagger \bdPsi_1\left(\qmat{Q}_{\bot}^*\qmat{A}\right)\|_F \leq \| \bdPsi_2^\dagger \bdPsi_1\|_2\|\qmat{A}-\qmat{Q}\qmat{Q}^*\qmat{A}\|_F $, which together with \eqref{eq:proof:qb_error:1} and Lemma \ref{lem:A-QQ^*A} yields \eqref{eq:thm_QB_error_for_sub_Gaussian}. 	 
		% it follows from   Lemma \ref{lem:A-QQ^*A} that
		% \begin{align*}
		% 	\|\qmat{A}-\qmat{Q}\qmat{Q}^*\qmat{A}\|_a\leq\left(1+\normSpectral{\bdOmega_1^\dagger}\normSpectral{\bdOmega_2}\right)\normA{\Sigma_2}.
		% \end{align*}
		% In particular,   using $\bdPsi_2^\dagger \bdPsi_1\left(\qmat{Q}_{\bot}^*\qmat{A}\right) = \qmat{Q}(\bdPsi \qmat{Q})^\dagger\bdPsi \qmat{A}-\qmat{Q}\qmat{Q}^*\qmat{A}$ and    Pythagorean identity,  
		% \begin{align*}
		% 	\normA{\qmat{A}-\qmat{H}\qmat{X}}^2=\normF{\qmat{A}-\qmat{Q}(\bdPsi \qmat{Q})^\dagger\bdPsi \qmat{A}}^2 
		% 	&=\normFSquare{\qmat{A}-\qmat{Q}\qmat{Q}^*\qmat{A} + \qmat{Q}\qmat{Q}^*\qmat{A} - \qmat{Q}(\bdPsi \qmat{Q})^\dagger\bdPsi \qmat{A}   } \\
		% 	&= \|\qmat{A}-\qmat{Q}\qmat{Q}^*\qmat{A}\|_F^2 + \|\bdPsi_2^\dagger \bdPsi_1\left(\qmat{Q}_{\bot}^*\qmat{A}\right)\|_F^2 .
		% \end{align*}
		\end{proof}

	\begin{proof}[Proof of Theorem \ref{thm:truncationError}]
		% \color{red}
		%  If we denote the approximation of algorithm \ref{alg:S2S} by $\tilde{A}=H(\bdPsi H)^\dagger \bdPsi A$, then $X=H^\dagger\tilde{A}$ and $\tilde{A}$ is already in columnspace of $H$. The output of algorithm \ref{alg:FixedRankTwoSketches} is just $H\lfloor X\rfloor_r $. Now our work is to estimate the $\normA{\tilde{A}-H\lfloor X\rfloor_r}$.
		%  Let $A_r$ be the best-r rank approximation to $A$, we can find that $H^\dagger A_r$ is also a rank-r approximation. We know that $\lfloor H^\dagger\tilde{A}\rfloor_r$ is the closest rank-r matrix to $H^\dagger \tilde{A}$, $H^\dagger A_r$ must be no closer to $H^\dagger \tilde{A}$. Thus, we have:
		%  \begin{align}
			%     \normA{\tilde{A}-H\lfloor X\rfloor_r}&\leq \normA{\tilde{A}-HH^\dagger A_r}\\
			%     &\leq \normA{\tilde{A}-HH^\dagger A}+\normA{HH^\dagger A-HH^\dagger A_r}\\
			%     &\leq \normA{\tilde{A}-HH^\dagger A}+\normA{A-A_r}\\
			%     &\leq \normA{HX-HH^\dagger A}+\tau_r(A)
			%  \end{align}
		\color{black}
		Let $\qmat{H}=\qmat{Q}\qmat{B}$ as \eqref{eq:QB_representation_H} where $\qmat{Q}$ is partially orthonormal and $\qmat{B}$ is   invertible. Then $\qmat{X}=\left(\bdPsi \qmat{H}\right)^\dagger \bdPsi\qmat{A} =\left(\bdPsi \qmat{Q}\qmat{B}\right)^\dagger \bdPsi\qmat{A}=\qmat{B}^{-1}\left(\bdPsi \qmat{Q}\right)^\dagger \bdPsi\qmat{A}$, where the last equality follows from the proof of Lemma \ref{lem:HToQ} and that a random fat matrix $\bdPsi \in\bbQ^{l\times m}$ drawn from continuous distribution has full row rank genericly. Denote $\qmat{A}_{in}:=\qmat{Q}(\bdPsi \qmat{Q})^\dagger \bdPsi \qmat{A}$.
		We have:
		\begin{align*}
			\normF{\qmat{H}\qmat{X}-\qmat{H}\lfloor \qmat{X}\rfloor_r}
			&\leq\normSpectral{\qmat{H}}\normF{\qmat{X}-\lfloor \qmat{X}\rfloor_r}\\
			 &=\normSpectral{\qmat{H}}\normF{\qmat{B}^{-1}(\bdPsi \qmat{Q})^\dagger \bdPsi \qmat{A}-\lfloor \qmat{B}^{-1}(\bdPsi \qmat{Q})^\dagger \bdPsi \qmat{A}\rfloor_r}\\		
			&\leq \normSpectral{\qmat{H}}\normF{\qmat{B}^{-1}(\bdPsi \qmat{Q})^\dagger \bdPsi \qmat{A}-\qmat{B}^{-1}\lfloor (\bdPsi \qmat{Q})^\dagger \bdPsi \qmat{A}\rfloor_r}	\\
			&\leq\normSpectral{\qmat{H}}\normSpectral{\qmat{B}^{-1}}\cdot\normF{(\bdPsi \qmat{Q})^\dagger \bdPsi \qmat{A}-\lfloor (\bdPsi \qmat{Q})^\dagger \bdPsi \qmat{A}\rfloor_r}\\
	 \text{\tiny(\cite[(6.3)]{Practical_Sketching_Algorithms_Tropp}, $\qmat{Q}\lfloor (\bdPsi \qmat{Q})^\dagger \bdPsi \qmat{A}\rfloor_r = \lfloor\qmat{A}_{in}\rfloor_r)$}		&=\kappa(\qmat{H}) \normF{ \qmat{A}_{in} - \lfloor \qmat{A}_{in}  \rfloor_r   } ~\leq  \kappaq{H}\normF{ \qmat{A}_{in} - \lfloor \qmat{A} \rfloor_r }\\
	 &\leq \kappaq{H}(\normF{ \qmat{A}_{in} - \qmat{A} } + \normF{\qmat{A}-\lfloor \qmat{A}\rfloor_r }  )\\
			&=\kappa(\qmat{H})\left(\normF{\qmat{A}-\qmat{H}\qmat{X}}+\|\Sigma_2\|_F\right),
		\end{align*}
		% where the third inequality uses that $\qmat{X}=\left(\bdPsi \qmat{Q}\qmat{R}\right)^\dagger \bdPsi\qmat{A} = \qmat{R}^{-1}(\bdPsi \qmat{Q})^\dagger \bdPsi \qmat{A} $ and that $\bdPsi \qmat{Q} \in\bbQ^{l\times s}~(l\geq s) $ has full column rank. 
		where the second   inequality is because that $\qmat{B}^{-1}\lfloor (\bdPsi \qmat{Q})^\dagger \bdPsi \qmat{A}\rfloor_r$ has rank at most  $r$, which is no closer than $\lfloor \qmat{B}^{-1}(\bdPsi \qmat{Q})^\dagger \bdPsi \qmat{A}\rfloor_r$ to $\qmat{B}^{-1}(\bdPsi \qmat{Q})^\dagger \bdPsi \qmat{A}$. The last equality uses Lemma \ref{lem:HToQ} that $\qmat{H}\qmat{X}=  \qmat{H}(\bdPsi \qmat{H})^\dagger \bdPsi \qmat{A} = \qmat{Q}(\bdPsi \qmat{Q})^\dagger \bdPsi \qmat{A}=\qmat{A}_{in}$.
		% The last inequality is deducted similar to \cite[Section 6.1]{Practical_Sketching_Algorithms_Tropp} and noting that $\kappa(\qmat{H})=\normSpectral{\qmat{H}}\normSpectral{\qmat{B}^{-1}}$.
		\color{black}
		
	\end{proof}
	\subsection{Probabilistic error: Guassian embedding}\label{sec:Gaussian}

	% %%%% real representation from Preliminaries

	% In fact, the quaternion matrix, as an operator, has a more essential representation that establishes a bridge between it and a real or complex matrix.
	% For $\mathbf{Q}\in\mathbb{Q}^{m\times n}$, define the real counterpart $\Upsilon_{\mathbf{Q}}$ and its first column block $\mathbf{Q}_r$ as 
	% $$\Upsilon_{\mathbf{Q}}=
	% \begin{bmatrix}
	% 	Q_0&-Q_1&-Q_2&-Q_3\\
	% 	Q_1&Q_0&-Q_3&Q_2\\
	% 	Q_2&Q_3&Q_0&-Q_1\\
	% 	Q_3&-Q_2&Q_1&Q_0
	% \end{bmatrix}.\quad
	% \mathbf{Q}_r=\begin{bmatrix}
	% 	Q_0\\
	% 	Q_1\\
	% 	Q_2\\
	% 	Q_3
	% \end{bmatrix}$$
	% A number of facts of the real counterpart are as follows:
	% \begin{gather}
	% 	\Upsilon_{k_1\mathbf{P}+k_2\mathbf{Q}}=k_1\Upsilon_\mathbf{P}+k_2\Upsilon_\mathbf{Q}(k_1,k_2\in\mathbb{R}),\quad\Upsilon_\mathbf{Q^*}=\Upsilon_\mathbf{Q}^T,\\ \quad\Upsilon_\mathbf{QS}=\Upsilon_\mathbf{Q}\Upsilon_\mathbf{S}\quad
	% 	(\mathbf{QS})_c=\Upsilon_\mathbf{Q} \mathbf{S}_c
	% \end{gather}
	
	%%%%%%%%%%%%%%%%%%% end of this part
	We further quantify the probabilistic estimation of the QB error in Theorem \ref{thm:errorAnalysisOfTwoSketch} with   Guassian test matrices. The fixed-rank error can be   derived  accordingly. 
	Using the statistical properties of  quaternion Gaussian matrices established in \cite{RandomizedQSVD}, the estimation can be proved using a similar deduction as in \cite{Practical_Sketching_Algorithms_Tropp}. We first recall some results from \cite{RandomizedQSVD}. 
	% In this section, we will estimate the average error of Algorithm \ref{alg:S2S} when the test matrices are quaternion Gaussian matrices. It means that all entries of the first column block in its real counterpart are independent and follow the standard normal distribution. Due to the beautiful marginal property of the Gaussian distribution, we can derive some results using a similar deduction as in \cite{Practical_Sketching_Algorithms_Tropp}. Some results on quaternion Gaussian matrices are also provided in \cite{RandomizedQSVD}. It is sufficient for us to estimate the average Frobenius and spectral errors generated in QB stage.

	\begin{definition}(c.f. \cite{RandomizedQSVD}) \label{def:q_gaussian_mat}
	Let $\qmat{A}\in\mathbb{Q}^{m\times n}$ and write $\qmat{A}=A_w+A_x\bi+A_y\bj+A_z\bk$. $\qmat{A}$ is Gaussian if all entries of  $A_k\in\mathbb{R}^{m\times n},k\in\{w,x,y,z \}$ 
	are  i.i.d   standard  Gaussian random variables.
\end{definition}
we call a quaternion matrix $\qmat{A}$   standard Gaussian   if it satisfies Definition \ref{def:q_gaussian_mat}.
 
	\begin{lemma}{(\cite{RandomizedQSVD})}\label{lem:normOfSGT}
		Let $\qmat{G}\in\mathbb{Q}^{m\times n}$ be   standard Guassian    and $\qmat{S}\in\mathbb{Q}^{l\times m}$, $\qmat{T}\in\mathbb{Q}^{n\times s}$ be fixed. Then 
		% \begin{align*}
			$\mathbb{E}\normFSquare{\qmat{SGT}} =4\normFSquare{\qmat{S}}\normFSquare{\qmat{T}}$.
	% 		\\
	% \color{blue}		\mathbb{E}\normSpectral{\qmat{SGT}}&\leq 3\left(\normSpectral{\qmat{S}}\normF{\qmat{T}}+\normF{\qmat{S}}\normSpectral{\qmat{T}}\right). \color{black}
		% \end{align*}
	% \end{lemma}
	% \begin{lemma}{(\cite{RandomizedQSVD})}\label{lem:normOfpseudo}
		% Let   $\qmat{G}\in\mathbb{Q}^{m\times n}\left(m\leq n\right)$ be a quaternion Guassian matrix. Then
		Furthermore if $m\leq n$, then
		% \begin{align*}
			$\mathbb{E}\normFSquare{\qmat{G}^\dagger}=\frac{m}{4\left(n-m\right)+2}$.
			% \color{blue}\quad{\rm and}\quad 				\mathbb{E}\normSpectral{\qmat{G}^\dagger}=\frac{e\sqrt{4n+2}}{2n-2m+2}. \color{black}
		% \end{align*}
	\end{lemma}
	% \begin{lemma}\label{lem:QSVDMainThm}
	% 	For    $\qmat{A} \in\bbQmn (m\geq n )$, let $(\qmat{H},\qmat{X})$ be generated by Algorithm \ref{alg:S2S} with \eqref{eq:rangefinder_basic_requirement} hold,  and let $\qmat{H}=\qmat{Q}\qmat{B}$   as in \eqref{eq:QB_representation_H}. Then 
	% 	% Let $\bdOmega$ be an $n\times s$ quaternion guassian test matrix, $\qmat{Q}$ is computed by Algorithm \ref{alg:S2S}, then the expected approximation for the $\qmat{Q}\qmat{Q}^*A$ satisfies
	% 	\begin{align}
	% 		\mathbb{E}\normF{\qmat{A}-\qmat{Q}\qmat{Q}^*\qmat{A}}\leq\left(1+\frac{4k}{4\left(s-k\right)+2}\right)^{1/2}\tau_k(\qmat{A})
	% 	\end{align}
	% 	\begin{align}
	% 		\mathbb{E}\normSpectral{\qmat{A}-\qmat{Q}\qmat{Q}^*\qmat{A}}\leq\left(1+3\sqrt{\frac k{4\left(s-k\right)+2}}\right)\sigma_{k+1}+\frac{3e\sqrt{4s+2}}{2\left(s-k\right)+2}\tau_k(\qmat{A}).
	% 	\end{align}
	% \end{lemma}
	% \begin{proof} Using Lemma \ref{lem:A-QQ^*A}, it suffices to estimate $\mathbb{E}\normA{\Sigma_2\bdOmega_2^*\left(\bdOmega_1^*\right)^\dagger}$, which has been  
	% \end{proof}
	In what follows, we denote $f(n,m):= \frac{4n}{4(m-n)+2}$.
	\begin{theorem}[Probabilistic QB error]\label{thm:prob_qb_error_gaussian}
		For    $\qmat{A} \in\bbQmn \color{blue}(m\geq n )\color{black}$, let $(\qmat{H},\qmat{X})$ be generated by Algorithm \ref{alg:S2S} with \eqref{eq:rangefinder_basic_requirement} and \eqref{eq:sketch_size_basic_requirement} hold. If $\bdPsi$ and $\bdOmega$ in  \eqref{eq:Y=AOmegaW=PsiA} are   standard Gaussian, then
		\begin{small}
		\begin{align}
		&	\mathbb{E}\normFSquare{\qmat{HX}-\qmat{A}}\leq (1+f(s,l))(1+f(r,s))\|\Sigma_2\|_F^2= \left(\frac{2l+1}{2\left(l-s\right)+1}\right) \left(\frac{2s+1}{2\left(s-r\right)+1}\right) \normFSquare{\Sigma_2} \label{eq:thm:prob_QB_error_Fnorm_error}.
		\end{align}
	\end{small}
	\end{theorem}

	% The theorem is a special situation of the theorem \ref{thm:errorAnalysisOfTwoSketch}. Its proof is similar, and marginal property of quaternion Gaussian embeddings offers more accurate error bounds.
	\begin{proof}
		Write $\qmat{H}=\qmat{Q}\qmat{B}$   as in \eqref{eq:QB_representation_H} with $\qmat{Q}$ orthonormal and $\qmat{B}$ invertible.  From  \eqref{eq:thm_QB_error_for_Guassian}  of Theorem \ref{thm:errorAnalysisOfTwoSketch}, it suffices to respectively estimate $\normFSquare{\bdPsi_2^\dagger \bdPsi_1\left(\qmat{Q}_\bot^*\qmat{A}\right)}$ and $\normFSquare{\qmat{A}-\qmat{Q}\qmat{Q}^*\qmat{A}}$.   
	 Owing to the marginal property of the standard normal distribution, $\bdPsi_2=\bdPsi \qmat{Q}\in\bbQ^{l\times s}$ and $\bdPsi_1=\bdPsi \qmat{Q}_\bot$ are statistically independent standard Guassian matrices \cite{RandomizedQSVD}. \color{black}Note that $\bdPsi_2^* \in\bbQ^{s\times l}(s\leq l)$ is still standard Gaussian and it holds that $\bdPsi_2^\dagger = ((\bdPsi_2^\dagger)^*)^* = ((\bdPsi_2^*)^\dagger)^*$; thus by Lemma \ref{lem:normOfSGT}, $\mathbb E\|\bdPsi_2^\dagger\|_F^2=\mathbb E\|(\bdPsi_2^*)^\dagger\|_F^2 = s/(4(l-s)+2) =f(s,l)/4 $.  \color{black} 
Therefore, 
		\begin{align}
			\mathbb{E}_{\bdPsi}\|\bdPsi_2^\dagger \bdPsi_1\left(\qmat{Q}_\bot^*\qmat{A}\right)\|_F^2&=\mathbb{E}_{\bdPsi_2}\mathbb{E}_{\bdPsi_1}\|\bdPsi_2^\dagger \bdPsi_1\left(\qmat{Q}_\bot^*\qmat{A}\right)\|_F^2 \nonumber\\
		\text{\tiny (Lemma  \ref{lem:normOfSGT})}	&=4\mathbb{E}\normFSquare{\bdPsi_2^\dagger}\normFSquare{\qmat{Q}_\bot^*\qmat{A}}\nonumber\\
		% \text{\tiny (Lemma  \ref{lem:normOfpseudo})}
			&= f(s,l)\normFSquare{\qmat{Q}_\bot^*\qmat{A}} 
			 = f(s,l)\normFSquare{(I-\qmat{Q}\qmat{Q}^*)\qmat{A}}. \label{eq:estimate_Psi2Psi1}
		\end{align}	
It follows from  \eqref{eq:thm_QB_error_for_Guassian} and  the independence of $\bdOmega$ and $\bdPsi$ that
		\begin{align}\label{eq:estimate_HX_A_fnorm_first_step}
			\mathbb{E}\normFSquare{\qmat{H}\qmat{X}-\qmat{A} } &= \bigxiaokuohao{1+ f(s,l)}\mathbb{E}_{\bdOmega}  \normFSquare{ (I-\qmat{Q}\qmat{Q}^*)\qmat{A}  }. 
		\end{align}
		\eqref{eq:lem:A-QQ^*A:2} of	Lemma \ref{lem:A-QQ^*A} shows that 
		\begin{align}			
		\mathbb{E}_{\bdOmega}\normFSquare{ (I-\qmat{Q}\qmat{Q}^*)\qmat{A}  }	&\leq  \normFSquare{\Sigma_2}+\mathbb{E}_{\bdOmega}\normFSquare{\Sigma_2\bdOmega_2\bdOmega_1^\dagger}\nonumber\\
		&\leq \left(1+f(r,s)\right)\normFSquare{\Sigma_2},\label{eq:normF_A-QQ*A}
		\end{align}
		where the deduction of the second inequality is similar to that of \eqref{eq:estimate_Psi2Psi1}. Plugging this into \eqref{eq:estimate_HX_A_fnorm_first_step} yields \eqref{eq:thm:prob_QB_error_Fnorm_error}.
	\end{proof}

\begin{theorem}[Probabilistic fixed-rank error] 
	\label{thm:prob_fix_rank_error_gaussian} Under the setting of Theorem \ref{thm:prob_qb_error_gaussian}, let $\hat{\qmat{A}}=\qmat{H}\lfloor \qmat{X}\rfloor_r$ be generated by Algorithm \ref{alg:FixedRankTwoSketches}. Then
	\begin{small}
	\[
		\mathbb{E}\normF{\hat{\qmat{A}}-\qmat{A}}\leq \bigxiaokuohao{ (1+\kappaq{H})\sqrt{(1+f(s,l)\cdot(1+f(r,s)))  }+\kappaq{H} }\|\Sigma_2\|_F.  
	\]
	\end{small}
\end{theorem}
\begin{proof}
This is obtained by	  $\|\hat{\qmat{A}}-\qmat{A}\|_F \leq \|\qmat{A}-\qmat{H}\qmat{X}\|_F + \|\qmat{H}\qmat{X} - \qmat{H}\lfloor\qmat{X}\rfloor_r\|_F$ together with Theorems \ref{thm:truncationError}, \ref{thm:prob_qb_error_gaussian}, and Jensen inequality.
\end{proof}

The above two probabilistic errors generalize \cite[Theorem 4.3, Corollary 6.4]{Practical_Sketching_Algorithms_Tropp} to the quaternion and non-orthonormal rangefinders settings. 

	\subsection{Probabilistic error: sub-Guassian embedding}\label{sec:sub-gaussian}

		A random variable $\xi$ is called   sub-Gaussian   if it satisfies \cite{vershynin_2012_Introduction_non-asymptotic} 
	 $
			\mathbb{P}\left\{\left|\xi \right|>t\right\}\leq \exp\left(1-ct^2\right)
	 $	for all $t>0$ and some constant $c>0$.
	The sub-Gaussian norm of $X$    is defined as:
	 $
			\|\xi\|_{\psi_2}=\sup_{p>=1} p^{-1/2}\left(\mathbb{E}\left|\xi\right|^p\right)^{1/p} 
	 $. 
	 Common sub-Gaussian variables include Gaussian, sparse Gaussian, Radmacher,  and   bounded random  variables. Using sub-Gaussian  such as sparse Gaussian or Radmacher gives more flexibility or improves the speed of generating random test matrices and sketches, We define sub-Gaussian quaternion matrices as follows:

	\begin{definition}[Entry-independent sub-Gaussian quaternion matrix] Let $\qmat{A}\in\mathbb{Q}^{m\times n}$ and write $\qmat{A}=A_w+A_x\bi+A_y\bj+A_z\bk$. %with $A_w, A_x, A_y, A_z\in\mathbb{R}^{m\times n}$. 
		If all entries of  the real matrix $ \cptQr{A}:= \begin{bmatrix}
			A_w^T, A_x^T, A_y^T,  A_z^T
		\end{bmatrix}^T\in \mathbb R^{4m\times  n}$
		are  independent   sub-Gaussian random variables, we call   $\qmat{A}$ a quaternion sub-Gaussian matrix. %In particular, when every entry of $\cptQr{A}$   has \color{black}the same \color{black}sub-Gaussian norm $K$, we say that $\qmat{A}$ has  sub-Gaussian norm $K$. 
		If  every entry of $\cptQr{A}$   is centered and has unit variance, we say that $\qmat{A}$ is a centered and unit variance quaternion sub-Gaussian matrix. 
		 \label{def:sub_gaussian_q_mat}
	\end{definition}

We   establish the probabilistic QB error    with sub-Gaussian test matrices, while the deduction for fixed-rank error is ommited as it is similar to Theorem \ref{thm:prob_fix_rank_error_gaussian}. 
\begin{theorem}[Probabilistic QB error]\label{thm:ErrorAnalysisOfSubGaussian}
	 Assume that the sketch   parameters satisfy    $r< s< l< min\{m,n\}$. \color{black} Draw   sub-Gaussian quaternion test matrices $\bdOmega\in\mathbb{Q}^{n\times s}$ and $\bdPsi\in\mathbb{Q}^{l\times m}$ according to Definition \ref{def:sub_gaussian_q_mat} and assume that every entry of $\bdOmega_r$ and $\bdPsi_r$ has the same sub-Gaussian norm $K$. Let $\qmat{H},\qmat{X}$ be   generated by Algorithm \ref{alg:S2S}; then
	 \begin{small}
	\begin{align}
		\normF{\qmat{HX}-\qmat{A}}\leq&\left(1+\frac{K\sqrt{l}\sqrt{2\bigxiaokuohao{\log(\min\{m-s,n\})+\log(4l)+1}/c_1}}{\sqrt{l}-C_K\sqrt{s}-t }\right)\nonumber\\
		&\cdot\left(1+\frac{K\sqrt{s}\sqrt{2\bigxiaokuohao{\log(min\{m,n\}-r)+\log(4s)+1} /c_1}}{\sqrt{s}-C_K\sqrt{r}-t }\right) \normF{\Sigma_2} \label{eq:qb_error_sub_gaussian}
	\end{align}
\end{small}
\noindent 	with probability at least $1-2\exp\left(-\frac{2c_2 t^2}{K^4}\right)-\frac{1}{\min\{m-s,n\}}-\frac{1}{\min\{m,n\}-r}$, where $c_1>0,c_2>0$ are absolulte constants and $C_K>0$    only depends on  $K$. 
	%where $f(p,q,r,t)=\frac{2\sqrt{p}+2C_K\sqrt{q-r}+t}{2\sqrt{p}-2C_K\sqrt{r}-t}$.
\end{theorem}

Comparing with the Gaussian bound \eqref{eq:thm:prob_QB_error_Fnorm_error},  an extra factor of order $O\bigxiaokuohao{ \log(\min\{m,n \} )   }$ appears in   \eqref{eq:qb_error_sub_gaussian}. The main reason is that for quaternion  sub-Gaussian, we do not have a sharp estimation of $\|\qmat{S}\qmat{G}\qmat{T}\|_F$ as that in Lemma \ref{lem:normOfSGT}. 
The main tool for proving Theorem \ref{thm:ErrorAnalysisOfSubGaussian}  is the following result, which generalizes the   deviation bound of \cite{vershynin_2012_Introduction_non-asymptotic} for extream singular values of a   sub-Gaussian real matrix to the quaternion realm.  

% The deviation bound for extreme singular values of a quaternion sub-Gaussian matrix is given as follows. The idea of the proof follows from \cite[Theorem 39]{vershynin_2012_Introduction_non-asymptotic}.
\begin{theorem}[Deviation bound]\label{thm:ConclusionOfSingularValueEstimation}
	Let $\qmat{A}\in\mathbb{Q}^{N\times n}(N>4n)$. If each row $\qmat{A}_i$  of $\qmat{A}$ are   independent quaternion sub-Gaussian isotropic   vectors, then for every $t> 0$, with probability at least $1-\exp\left(-{c t^2}/{K^4}\right)$ one has
	\begin{align}\label{eq:ConclusionOfSingularValueEstimation}
		2\sqrt{N}-2C_K\sqrt{n}-2t\leq \sigma_{\min}\left(\qmat{A}\right)\leq \sigma_{\max}\left(\qmat{A}\right)\leq 2\sqrt{N}+2C_K\sqrt{n}+2t.
	\end{align}
	Here  $c>0$ is an absolulte constant and $C_K>0$ \color{black}   only depends on the max sub-Gaussian norm $K=\max_i{\|\qmat{A}_i\|_{\psi_2}}$ of the rows.
\end{theorem}

	% \begin{definition}{(sub-gaussian random variable, \cite{vershynin_2012_Introduction_non-asymptotic})}
	% 	A random variable $X$ is called a sub-gaussian random variable if satisfying:
	% 	\begin{align*}
	% 		\mathbb{P}\left\{\left|X\right|>t\right\}\leq \exp\left(1-ct^2\right)
	% 	\end{align*}
	% 	for all $t>0$.
	% 	And the sub-gaussian norm of $X$, denoted $\|X\|_{\bdPsi_2}$, is defined as:
	% 	\begin{align*}
	% 		\|X\|_{\bdPsi_2}=\sup_{p>=1} p^{-1/2}\left(\mathbb{E}\left|X\right|^p\right)^{1/p}
	% 	\end{align*}
	% \end{definition}
% 	\color{red}

% \color{black}

Theorem \ref{thm:ConclusionOfSingularValueEstimation} relies on the    following   definitions; see  \cite[Def. 5.19]{vershynin_2012_Introduction_non-asymptotic} for their real counterparts.
 	\begin{definition}[Quaternion sub-Gaussian and isotropic vectors] \label{def:q_subG_vector_isotropc}
		A quaternion random vector $\qmat{v}$ in $\bbQ^n$ is   called  a sub-Gaussian vector if the one-demensional marginals $\innerprod{\qmat{v}_r}{x}$ are sub-Gaussian random variables for all $x\in\mathbb{R}^{4n}$, where $\qmat{v}_r = [v_w^T,v_x^T,v_y^T,v_z^T]^T\in\mathbb R^{4n}$. The sub-Gaussian norm of $\qmat{v}$ is defined as:
		\begin{align*}
			\|\qmat{v}\|_{\psi_2}:=\|\qmat{v}_r\|_{\psi_2}=\sup_{x^Tx=1,x\in \mathbb R^{4n}}\|\innerprod{\qmat{v}_r}{x}\|_{\psi_2}.
		\end{align*}
		$\qmat{v}$ is called isotropic if $\mathbb E \qmat{v}_r\qmat{v}_r^T = I_{4n}$.
		% \color{red}\color{black}
	\end{definition}

	The following proposition is also crucial for proving Theorem \ref{thm:ErrorAnalysisOfSubGaussian}.
	\begin{proposition}\label{prop:AVOmega}
		Let $\qmat{A}\in\mathbb{Q}^{N\times m}$, $\qmat{V}\in\mathbb{Q}^{m \times n}$ ($m\leq n)$ be row-orthonormal, and $\bdOmega\in\mathbb{Q}^{n\times s}$ be a quaternion centered sub-Gaussian matrix (Def. \ref{def:sub_gaussian_q_mat})  with    entries of $\bdOmega_r$ all having the same sub-Gaussian norm $K$. 
		Then 
		\begin{align}
			\normF{\qmat{A} \qmat{V}\bdOmega}\leq 2\sqrt{s}t\normF{\qmat{A}}
		\end{align}
		with probability at least $1-4e\cdot \exp\left(\log (s\min\{N,m\}  )-c t^2/K^2\right)$, where $c>0$ is an absolulte constant. 
	\end{proposition}

	The detailed proofs in this subsetion are left to the appendix, and   we give a	sketch of the proof of Theorem \ref{thm:ErrorAnalysisOfSubGaussian} here: observing \eqref{eq:thm_QB_error_for_Guassian}, we first   derive   $\|\bdPsi_2^\dagger \bdPsi_1\left(\qmat{Q}_{\bot}^*\qmat{A}\right)\|_F \leq \|\bdPsi_2^\dagger\|_2 \|\bdPsi \qmat{Q}_{\bot}\left(\qmat{Q}_{\bot}^*\qmat{A}\right)\|_F$; then respectively bound $\|\bdPsi_2^\dagger\|_2$ by Theorem \ref{thm:ConclusionOfSingularValueEstimation} and $\|\bdPsi \qmat{Q}_{\bot}\left(\qmat{Q}_{\bot}^*\qmat{A}\right)\|_F\leq O( \sqrt{ l\log( \min\{m,n \} ) })\|\qmat{Q}_{\bot}^*\qmat{A}\|_F$ by Proposition \ref{prop:AVOmega}   (observe that $\bdPsi$ is sub-Gaussian, $\qmat{Q}_{\bot}$ is orthonormal, and $\qmat{Q}_{\bot}^*\qmat{A}$ is fixed). Next, use \eqref{eq:lem:A-QQ^*A:1} to obtain $\|\qmat{Q}_{\bot}^*\qmat{A}\|_F^2 =  \normF{\qmat{A}-\qmat{Q}\qmat{Q}^*\qmat{A}}^2 \leq\normF{\Sigma_2}^2+\normF{\Sigma_2\bdOmega_2 \bdOmega_1^\dagger}^2$; the second term $\normF{\Sigma_2\bdOmega_2 \bdOmega_1^\dagger} \leq \|\Sigma_2\qmat{V}_2^*\bdOmega\|_F \|\bdOmega_1^\dagger\|_2$; respectively bound $\|\bdOmega_1^\dagger\|_2$ by Theorem \ref{thm:ConclusionOfSingularValueEstimation} and $\|\Sigma_2\qmat{V}_2^*\bdOmega\|_F \leq O(\sqrt{s \log(\min\{m,n \}) } )\|\Sigma_2\|_F$ by Proposition \ref{prop:AVOmega}. Finally, combining the above pieces and using   union bound or total probability, one obtains   \eqref{eq:qb_error_sub_gaussian}.

	\section{Numerical Experiments} \label{sec:numerical}
	In this section, we tested  our quaternion one-pass algorithm with the devised rangefinders on several examples. We used randomized quaternion SVD (RQSVD) \cite[Algorithm 3.1]{RandomizedQSVD} as a baseline to evaluate the performance of our algorithm for   synthetic data in example \ref{eg:syntheticData}. We  then   handled large-scale data processing tasks of  scientific simulation output compression and color image processing using our algorithm. All the experiments were carried out in MATLAB 2023b on a personal computer with an Intel(R) CPU i7-12700 of 2.10 GHz and 64GB of RAM and work in double-precesion arithmetic. The QTFM   \cite{sangwine2020qtfm} was employed for basic quaternion operations. 
	% with $\texttt{eps}=2.22e-16$. For these example, relative error are defined as : 
	% and each experiment use start stopwatch timer \texttt{tic} and \texttt{toc} to calculate the elapsed time. 
	This section only presents limited selection of results and more detail   results are conatined in supplementary materials.

\subsection{Synthetic data}
\begin{example}\label{eg:syntheticData}
	In this example, %we tested the quaternion low-rank approximation algorithm by synthetic data 
we	constructed $A=U\Sigma V^*\in\mathbb{Q}^{2000\times 1600}$ where $U$, $V$ are partial unitary and $\Sigma$ were generated as follows:
\begin{enumerate}
	\item Low rank plus noise:  $\Sigma=\diag{\left(1,\ldots,1,0,\ldots,0\right)+\xi n^{-1}\qmat{E}} \in \bbQ^{1600\times 1600}$.   $\qmat{E}$ is a quaternion standard Gaussian matrix;
	\item Polynomial decay spectrum (pds):  $\Sigma=\diag{\left(1,\ldots,1,2^{-p},3^{-p},\ldots,\left(n-R+1\right)^{-p}\right)}\in \bbQ^{1600\times 1600}$;
	\item Exponential decay spectrum (eds):  $\Sigma=\diag{\left(1,\ldots,1,10^{-q},10^{-2q},\ldots,10^{-\left(n-R\right)q}\right)}\in \bbQ^{1600\times 1600}$.
\end{enumerate}

We   adjusted the parameter $\xi,p,q$ to obtain different decay rate. For every synthetic data case, we   evaluated the performance of RQSVD and one-pass algorithm with different rangefinders: 1) RQSVD with QHQR, 2) RQSVD with QMGS, 3) RQSVD with our pseudo-SVD, 4) one-pass with pseudo-QR, and 5) one-pass with pseudo-SVD. The deterministic QSVD used in RQSVD, as suggested in \cite[Remark 3.2]{RandomizedQSVD}, were from \cite{li2016RealStructurepreserving,wei2018quaternion}, whose code is also provided in Jia's homepage; that used in one-pass were implemented by ourselves\footnote{Roughly speaking, our QSVD is based on  computing the   SVD of the full complex representation; however, as pointed out in Sect. \ref{sec:pseudo-svd}, we need to correct the approach such that it still works in the scenorio that   there exist duplicated singular values and the matrix is very ill-conditioned. The idea is similar to pseudo-SVD but the detail is more involved, which may be reported in an independent manuscript. The code is with the name \texttt{ModifiedcsvdQ77} in our github webpage.}. We tested different target rank $r$ varying from $100$ to $500$. For one-pass algorithm we set $s=r+5$ and $l=2s$; the oversampling parameter of RQSVD was set to  $5$, the same as one-pass algorithm. The relative error is defined as $
% \begin{align}
	\label{def:relativeError}
 {\normF{\qmat{A}-\hat{\qmat{A}}}}/{\normF{\qmat{A}}}
% \end{align}
$.
\end{example}

\begin{figure}
	\centering
	\begin{subfigure}[b]{0.48\textwidth}
		\includegraphics[width=\linewidth]{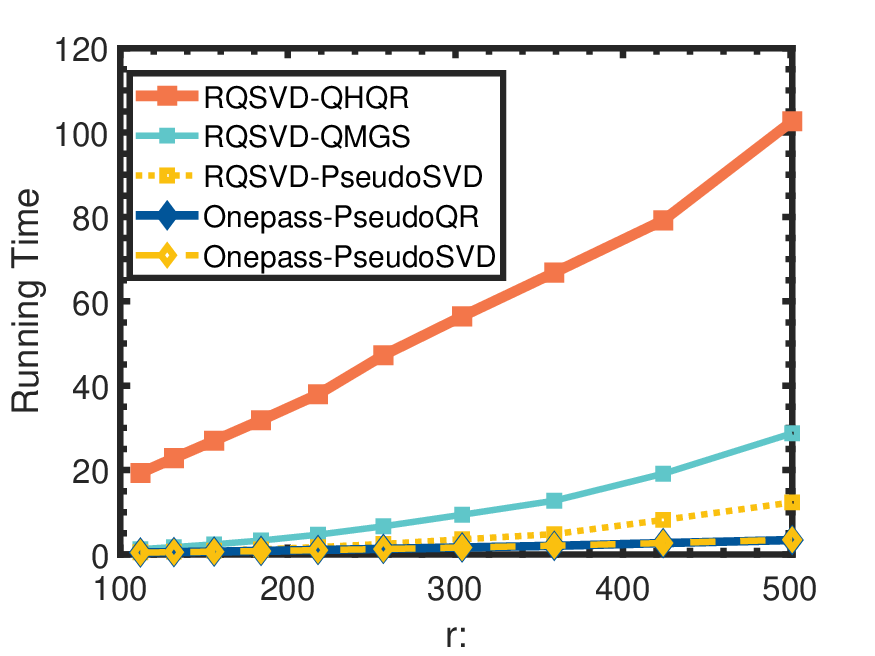}
		\caption{low-rank plus noise and pds}
	 \label{fig:timeNoise}
	\end{subfigure}
	\hfill
	\begin{subfigure}[b]{0.48\textwidth}
		\centering
		\includegraphics[width=\linewidth]{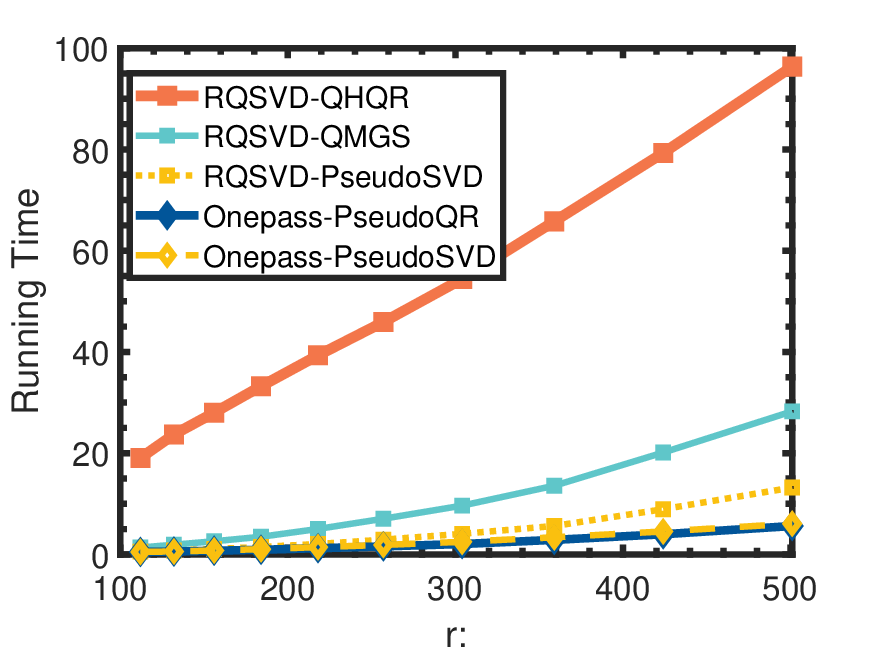}
		\caption{eds}
		\label{fig:timePds}
	\end{subfigure}
	\caption{\tiny $\xi=0.01,p=2,q=0.25$. Time cost of RQSVD (QHQR, QMGS, Pseudo-SVD) and algorithm \ref{alg:FixedRankTwoSketches} (pseudo-QR, pseudo-SVD) of synthetic data. $x$-axis is the target rank; $y$-axis is the running time in seconds. }
	\label{fig:timeSynthetic}
\end{figure}

Fig. \ref{fig:timeSynthetic} shows that RQSVD with QHQR is the slowest one, followed by that with QMGS; RQSVD with Pseudo-SVD is about two times faster than that with QMGS. This shows that even   Pseudo-SVD is useful for accelerating RQSVD. 
One-Pass with Pseudo-QR and Pseudo-SVD are the most fastest ones, which are about 3.5x faster than RQSVD with Pseudo-SVD and are about 8x faster than RQSVD with QMGS. This is also benifitial from the efficiency of our deterministic QSVD function \texttt{ModifiedcsvdQ77}. 

\begin{figure}
	\centering
	\begin{subfigure}[b]{0.32\textwidth}A
		\includegraphics[width=\linewidth]{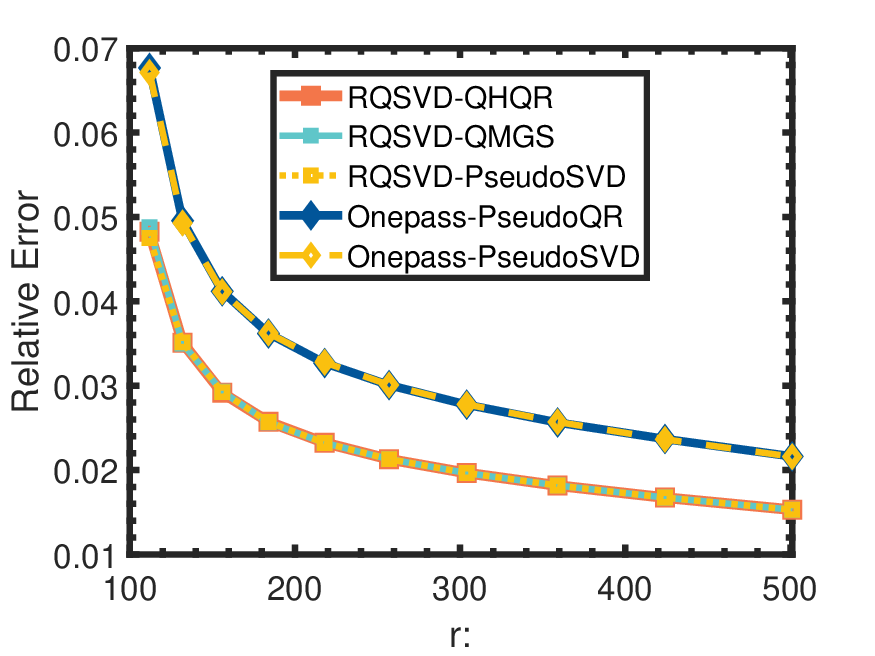}
		\caption{Low rank plus noise}
	 \label{fig:normNoise}
	\end{subfigure}
	\hfill
	\begin{subfigure}[b]{0.32\textwidth}
		\centering
		\includegraphics[width=\linewidth]{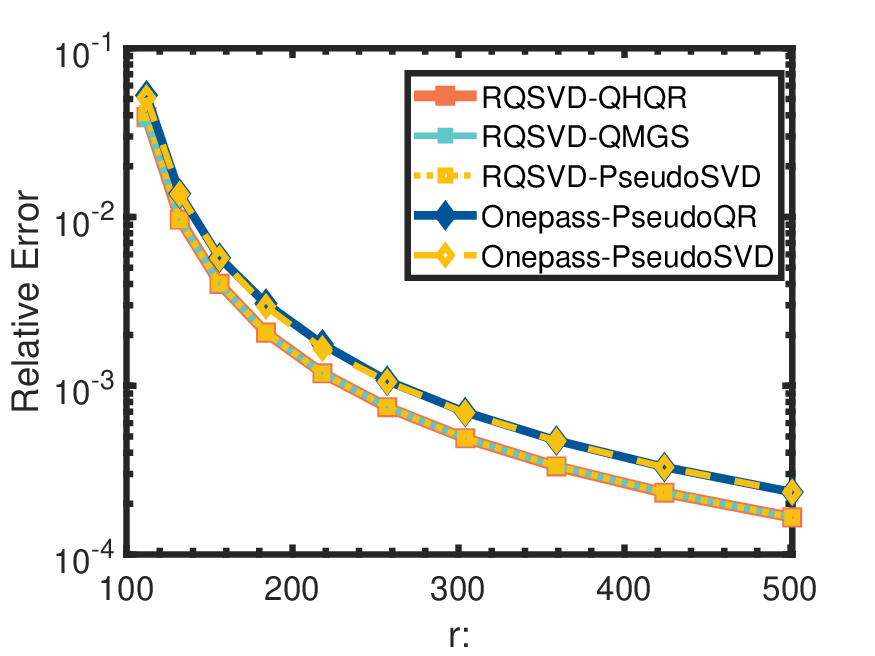}
		\caption{pds}
		\label{fig:normPds}
	\end{subfigure}
	\hfill
	\begin{subfigure}[b]{0.32\textwidth}
		\centering
		\includegraphics[width=\linewidth]{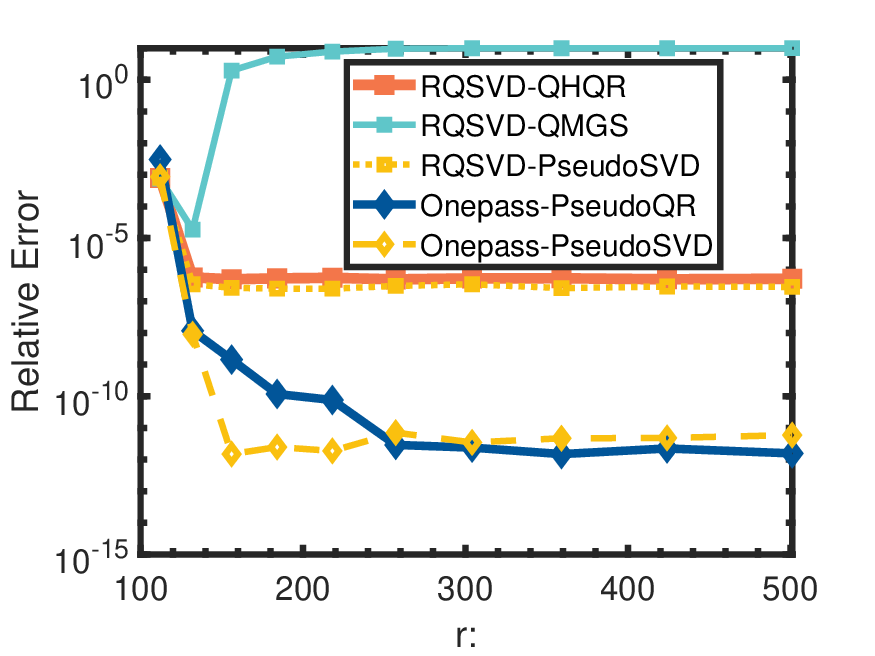}
		\caption{eds}
		\label{fig:normEds}
	\end{subfigure}
	\caption{\tiny Relative Errors of RQSVD (QHQR, QMGS, Pseudo-SVD) and algorithm \ref{alg:FixedRankTwoSketches} (pseudo-QR, pseudo-SVD) of synthetic data. x-label is target rank, y-label is relative $\normF{A-\hat{A}}/\normF{A}$.}
	\label{fig:normSynthetic}
\end{figure}

In terms of the relative error, Fig. \ref{fig:normSynthetic} show that one-pass algorithm is   worse  than RQSVD. This is understandable, as commented in the pagragraph below \eqref{eq:reconstructionOfTwoSketch}.
%, the one-pass algorithm takes an additional sketch $\qmat{W}=\bdPsi\qmat{A}$ than RQSVD, thereby trading accuracy for efficiency and the usage in the limite storage scenorio. 
We also remark that   their performance discrepancy diminishes when the   spectral decays    more rapidly. When the original data is extremely ill-conditioned, as in  Fig. \ref{fig:normEds}, some needs to be interpreted particularly. We first observe that RQSVD with QMGS is inaccuracy when $r$ is near   $120$; this is due to that     QMGS's orthonormality is more sensitive to the condition number of the sketch, as illustrated in Fig. \ref{fig:conditionThree}.  At the same time, RQSVD with QHQR holds the orthogonality while preserving less accurate range (Fig. \ref{fig:rangePrecisionThree}), and Pseudo-SVD   takes both aspects into account, therefore, RQSVD with Pseudo-SVD is slightly better than that with QHQR. On the other hand, it is interesting to see that one-pass with both Pseudo-QR and Pseudo-SVD perform better than the others; in fact, this is due to that in this very ill-conditioned scenorio, our deterministic QSVD function \texttt{ModifiedcsvdQ77} is more accurate than  that used in RQSVD \cite{RandomizedQSVD}.

% when original data has fast decay spectrum. When data has slow decay spectrum, algorithm \ref{alg:FixedRankTwoSketches} performs worse than RQSVD, which can also seen in \cite[Section 1.6]{Practical_Sketching_Algorithms_Tropp}.
%  However, it is imperative to underscore that the one-pass algorithm mainly designed for single pass access data or limitation storage or memory. Thus, the approximate precesion above is usually acceptable. More important thing is that it often works much faster than RQSVD, especially for those using our practical rangefinder pseudo-QR or pseudo-SVD. Figure \ref{fig:timeSynthetic} indicates that algorithm \ref{alg:FixedRankTwoSketches} is much faster than RQSVD with QHQR, which recommend in \cite[Remark 3.1]{RandomizedQSVD}. Even if comparing to QMGS-RQSVD, it cost less than $1/10$ time. 

\subsection{Scientific data}	
	\begin{example}
		In this example, we applied the one-pass algorithm  with pseudo-QR and pseudo-SVD to compress the output of a computational fluid dynamics (CFD) simulation. We have obtained a  numerical simulation on a finite elements model of the unsteady 3D Navier-Stokes equations for microscopic natural convection about biological research applications using the \href{https://old.quickersim.com/cfdtoolbox/}{QuickerSim CFD Toolbox for MATLAB}. 
		There are $20914$ nodes to characteristic the velocity, each node having three directions and can be represented as a pure quaternion. Each element represents the space velocity field by:
		\begin{align}
			a_{p,t}=v_x(p,t)\bi+v_y(p,t)\bj+v_z(p,t)\bk,
		\end{align}
		where $v=(v_x,v_y,v_z)$ is a velocity vector field in the model.
		The velocity field is time-dependent  and so 
		we collect it at $20000$ time instants, resulting into  a pure quaternion matrix \emph{microConvection} of size $20914\times 20000$.
	Due to the periodicity, the data matrix \emph{microConvection} generated by  this type of models has rapidly decay spectrum,  as shown in figure \ref{fig:SingularValueCFDData}. 
	% Usually, the large-scale data cannot be put into memory once, thus through linear updating process, we can obtain two small sketches whose sketch size less than $200$. Then algorithm \ref{alg:FixedRankTwoSketches} with pseudoQR or pseudoSVD generates a low-rank approximation. 

		%  Velocity $(u,v,w)$ and pressure field $p$ are computed by QuickerSim CFD Toolbox for MATLAB. Then shear rate on each node can be computed by velocity field. 
		
		% The origin data contains 20914 nodes and 20000 time instants, forming a quaternion matrix \emph{microConvection} where each element represents the space velocity field by:
		% \begin{align}
		% 	a_{p,t}=v_x(p,t)\bi+v_y(p,t)\bj+v_z(p,t)\bk
		% \end{align}
		% where $v=(v_x,v_y,v_z)$ is velocity vector field on the $\mathcal{M}\times \mathbb{R}$, $p\in\mathcal{M}$ is discrete point on vessel surface. A stable solution usually has rapidly decay spectrum as shown in figure \ref{fig:SingularValueCFDData} (An example which has $20914$ nodes and $20000$ instants). In linear updating process, we obtained the  sketches which have size $20914\times 110$ and $220\times 20000$. Then algorithm \ref{alg:FixedRankTwoSketches} with two sketches and a test matrix generates low-rank approximation. 
	\end{example}
	
	Fig. \ref{fig:CFD} show that the algorithm work efficiently and obtain highly accurate approximation for this example. In particular, if we set $k=50$, the relative   error is under $10^{-5}$ and is acceptable to be used in PDE numerical solving, while the time costs of the whole proceedure, including sketching and randomized  approximation, are less than $10$ seconds. 
	% In fact, in most practical applications, solution of CFD can be seen as a streaming model, so that sketches can be linear updated by last sketch and output of simulation in this instant. It means that our algorithm can complete compression process before next input, which allows us to compress on-the-fly large-scale quaternion output of scientific simulations.
	\begin{figure}
		\centering
		\begin{subfigure}[b]{0.32\textwidth}
			\includegraphics[width=\linewidth]{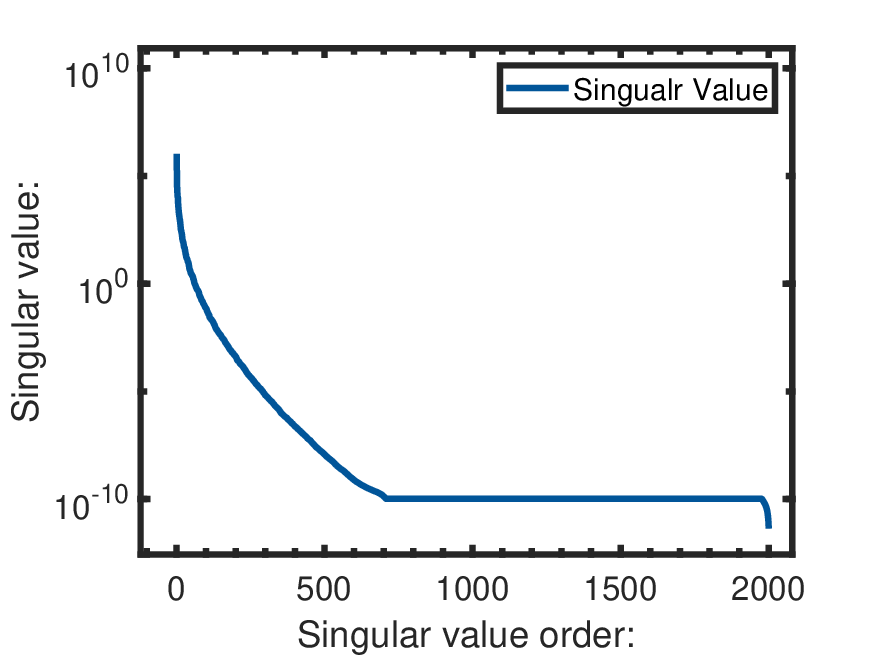}
			\caption{Singular values}
		 \label{fig:SingularValueCFDData}
		\end{subfigure}
	% \end{figure}
	% \begin{figure}
		\hfill
		\begin{subfigure}[b]{0.32\textwidth}
			\centering
			\includegraphics[width=\linewidth]{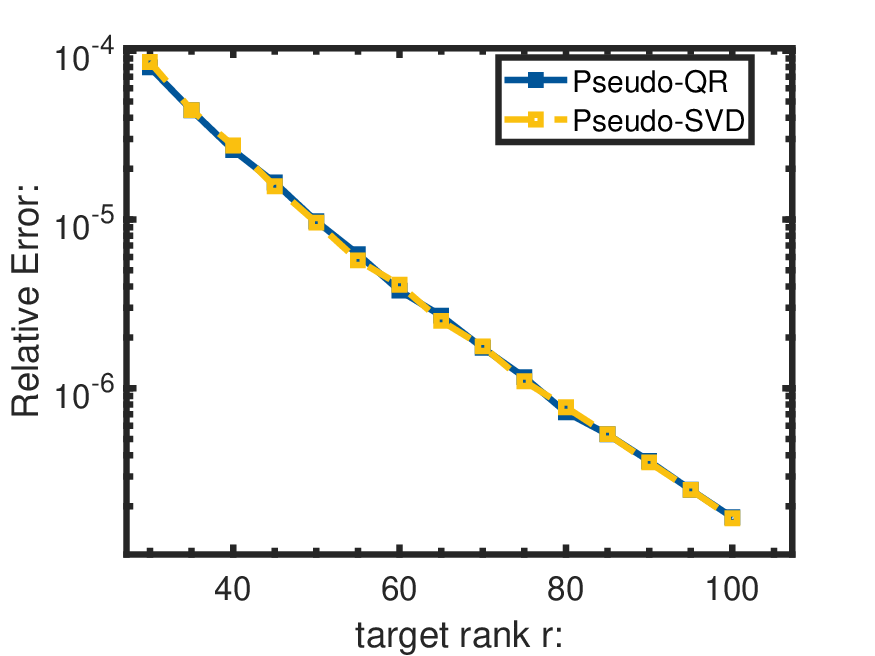}
			\caption{Relative Frobenius Error}
			\label{fig:CFDRelativeError}
		\end{subfigure}
		\hfill
		\begin{subfigure}[b]{0.32\textwidth}
			\centering
			\includegraphics[width=\linewidth]{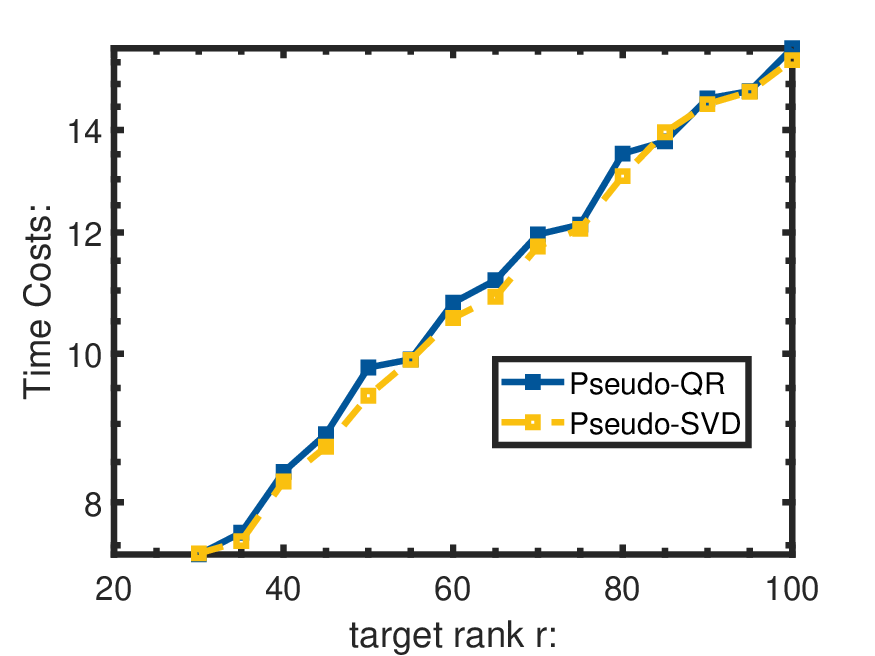}
			\caption{Running Time}
			\label{fig:CFDTime}
		\end{subfigure}
		\caption{Relative Errors and Running Time of CFD simulation with different target rank $k$}
		\label{fig:CFD}
	\end{figure}

	Fig. \ref{fig:ComparsionOfCFD} graphically  illustrates that the compression can highly approximate the origin velocity field: the figures in the second row are the shear rates computed from velocity field ($k=50$), which   closely match the real data, as shown in the first row. In this case, our algorithm compressed the original data from $5.22$GB to $47.8$MB, whose compression ratio is   $99.11\%$.  
	% In the typical server-client architecture of large-scale computer clusters, our compression algorithm only requires transmitting the linearly updated parts of sketches from the server side, which significantly saves local storage space and network transmission bandwidth.  
	% In practical, it allows us to free more Random Access Memory (RAM) or Video Random Access Memory (VRAM) to solve PDEs. 

	\begin{figure}
		\begin{subfigure}[b]{0.32\textwidth}
			\centering
			\includegraphics[width=\linewidth]{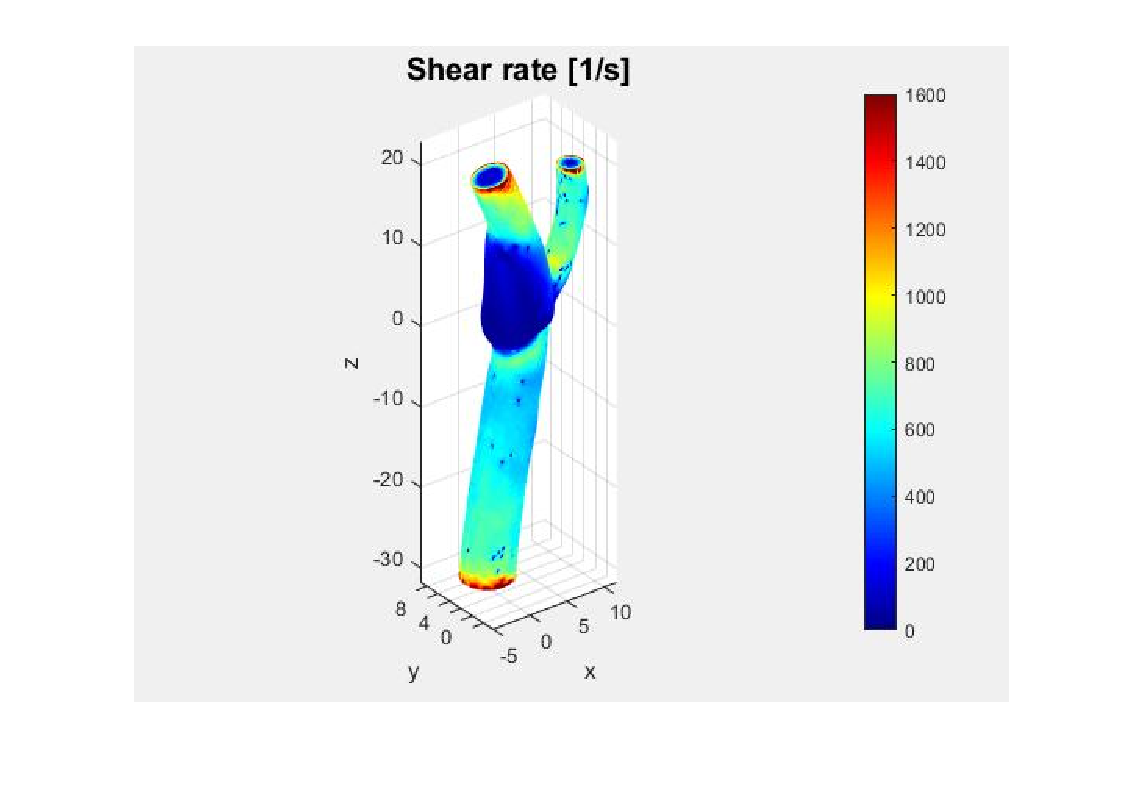}
			\caption{Origin at 8th step}
			\label{fig:Origin8}
		\end{subfigure}
		\hfill
		\begin{subfigure}[b]{0.32\textwidth}
			\centering
			\includegraphics[width=\linewidth]{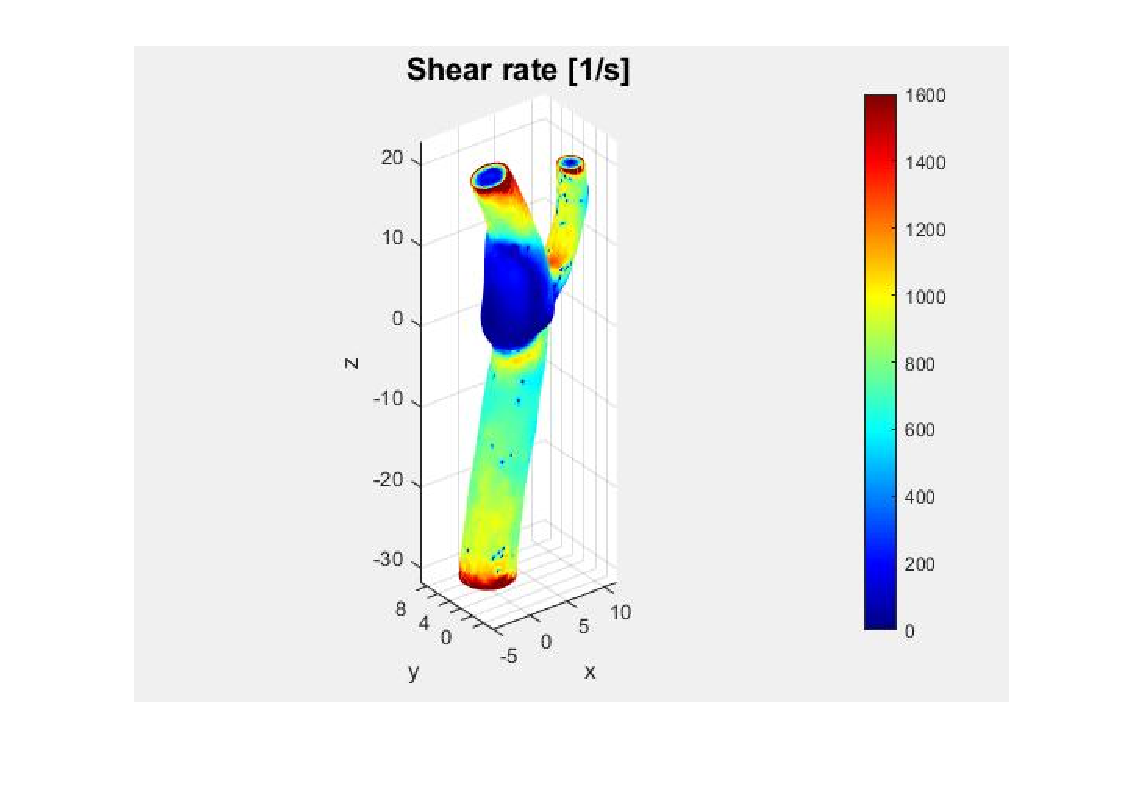}
			\caption{Origin at 14th step}
			\label{fig:Origin14}
		\end{subfigure}
		\hfill
		\begin{subfigure}[b]{0.32\textwidth}
			\centering
			\includegraphics[width=\linewidth]{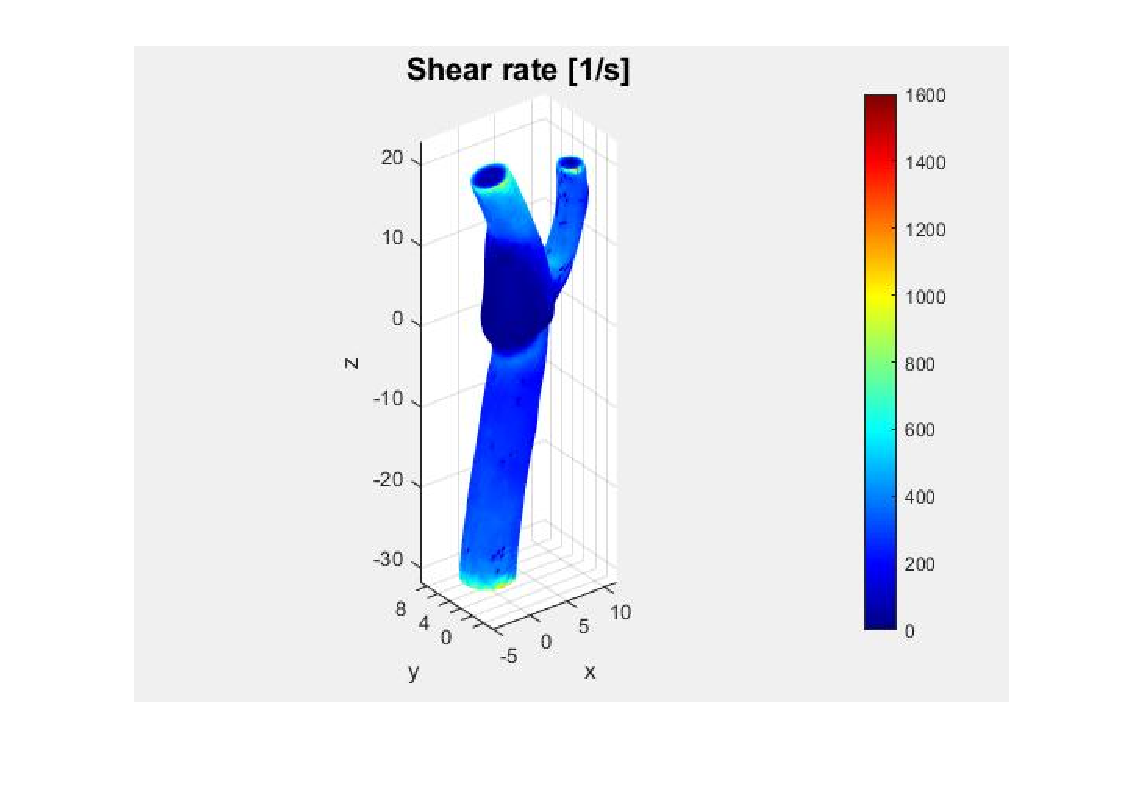}
			\caption{Origin at 23rd step}
			\label{fig:Origin23}
		\end{subfigure}
		\hfill
		\begin{subfigure}[b]{0.32\textwidth}
			\centering
			\includegraphics[width=\linewidth]{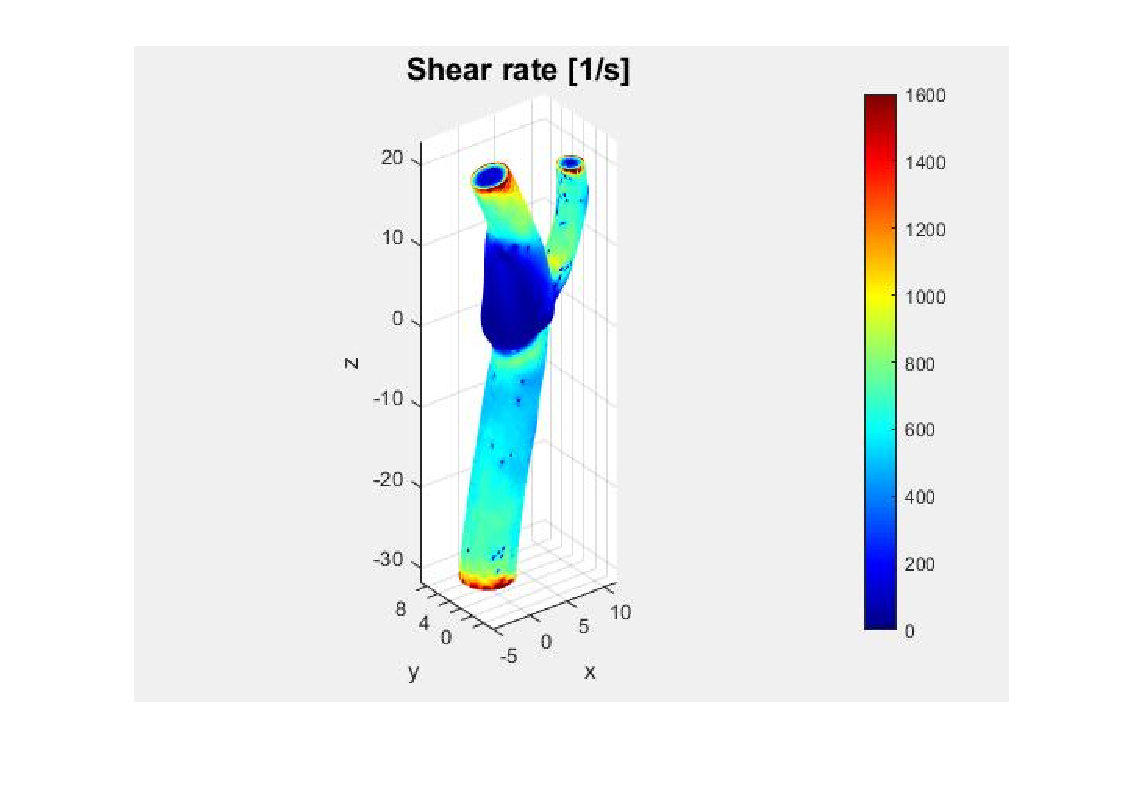}
			\caption{Recovered at 8th step}
			\label{fig:Approx8}
		\end{subfigure}
		\hfill
		\begin{subfigure}[b]{0.32\textwidth}
			\centering
			\includegraphics[width=\linewidth]{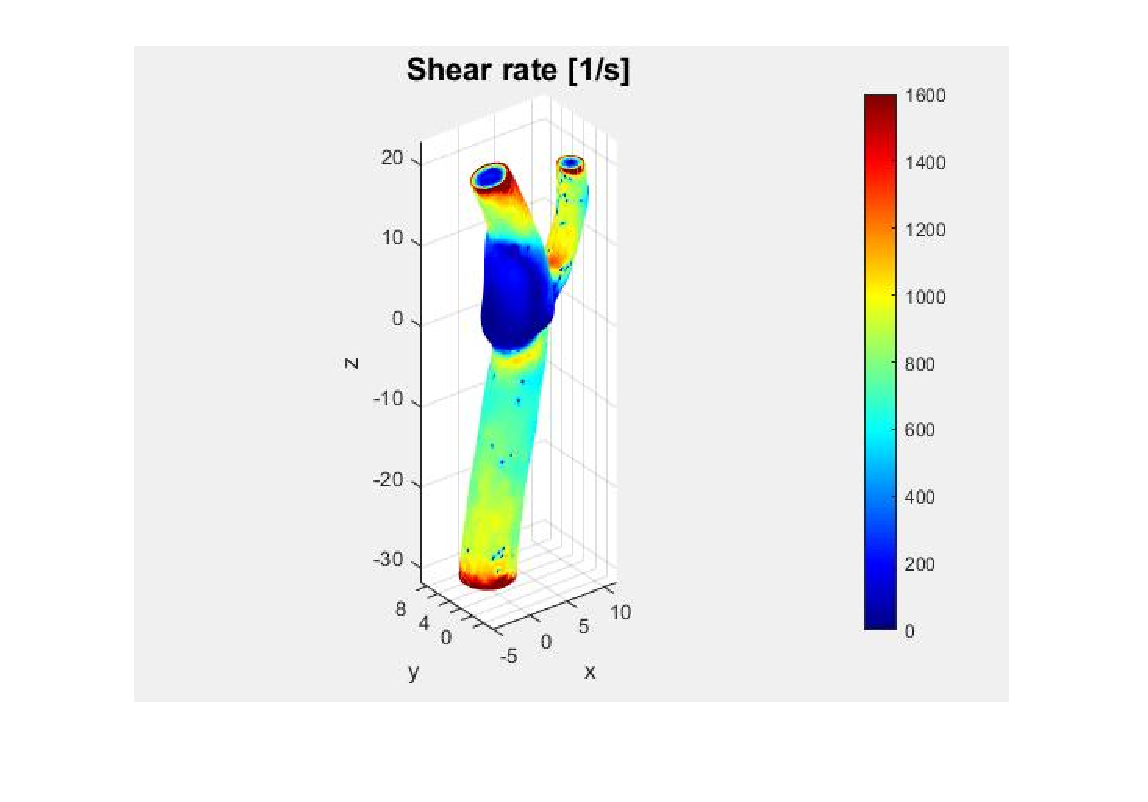}
			\caption{Recovered at 14th step}
			\label{fig:Approx14}
		\end{subfigure}
		\hfill
		\begin{subfigure}[b]{0.32\textwidth}
			\centering
			\includegraphics[width=\linewidth]{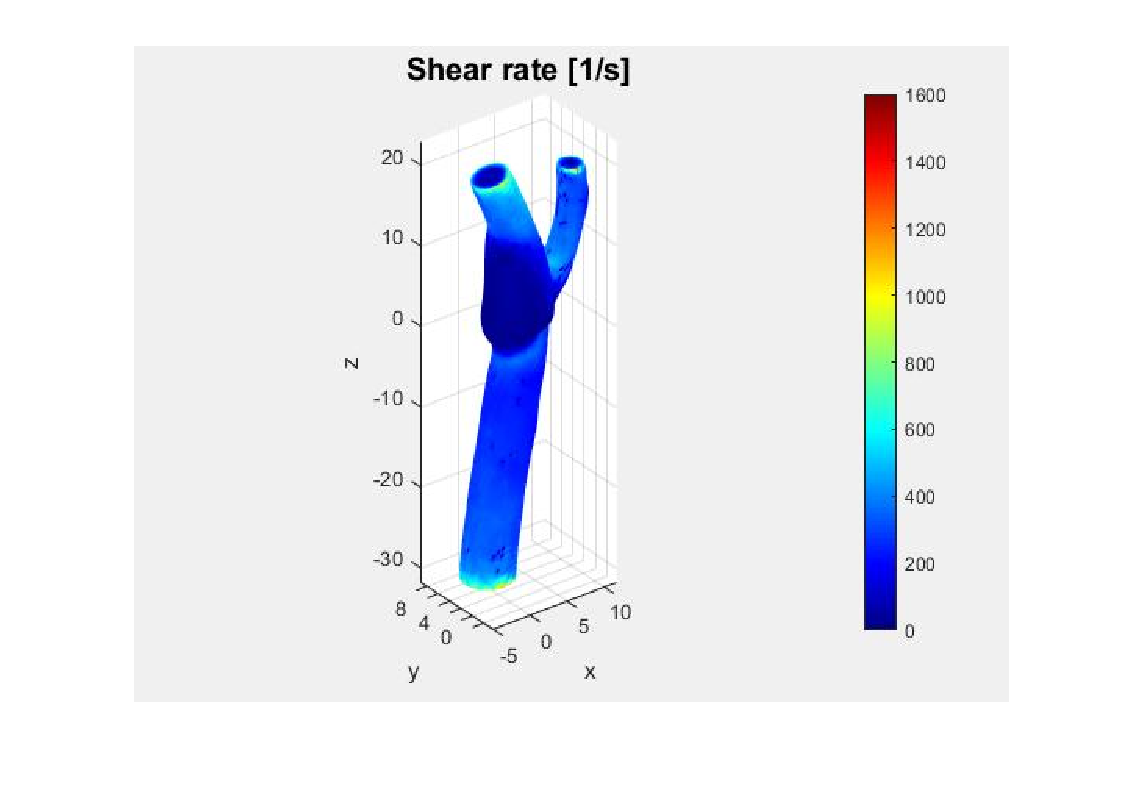}
			\caption{Recovered at 23rd step}
			\label{fig:Approx23}
		\end{subfigure}
		\caption{Shear rate computed from origin data and compressed data ($k=50$)}
		\label{fig:ComparsionOfCFD}
	\end{figure}
\begin{example}
	We   tested   compressing the output of a 4D Lorenz-type chaotic system simulation.  The dynamical system is shown as follows:
	$$\begin{cases}\frac{dx}{dt}=a(y-x),\\\frac{dy}{dt}=cx-y-xz+w,\\\frac{dz}{dt}=-bz+xy,\\\frac{dw}{dt}=(c-1)y+w-\frac{x^3}{b},\end{cases}$$ 
	where $x$, $y$, $z$ and $w$ are state variables and $a$, $b$, $c$, $d$, $e$, $h$ are positive parameters of system. In our simulation, we set $a=15$, $b=2$, $c=28$ and select $10000$ initial states $(x,y,z,w)$ randomly selected from sphere $\normSpectral{(x,y,z,w)}=20$. By using \texttt{ode45} to compute the system at $20000$ time instances, we   obtained a $20000\times 10000$ quaternion matrix which records the information of solutions. The distinction from the previous CFD simulation example is that, this 4D Lorenz-type system   is a hyperchaotic one, and so the quaternion matrix has more flatten spectrum decay, as shown in figure \ref{fig:SingularValueLorenz}. Specifically, in this hyperchaotic system, the presence of attractors causes a significant drop in singular values from the 700th to the 800th. 
	% Dealing with the large-scale data, we also use the same linear updating technology as in CFD simulation to generate two sketches and then use algorithm \ref{alg:FixedRankTwoSketches} with pseudoQR or pseudoSVD to obtain low-rank approximation.  
\end{example}

Figure \ref{fig:SingularValueLorenz} indicates a larger sketch size is necessary to obtain a high accuracy approximation. One-pass algorithm with pseudo-QR and pseudo-SVD  have similar performance in accuracy. When the target rank $r>1000$, the relative error is less than $10^{-3}$. Our algorithm compresses the original data form $5.74$GB to $998$MB and the compression ratio reaches to $85\%$ ($r=1000$). We can observe that pseduo-QR is faster when the target rank is larger.

\begin{figure}
	\begin{subfigure}[b]{0.32\textwidth}
		\includegraphics[width=\linewidth]{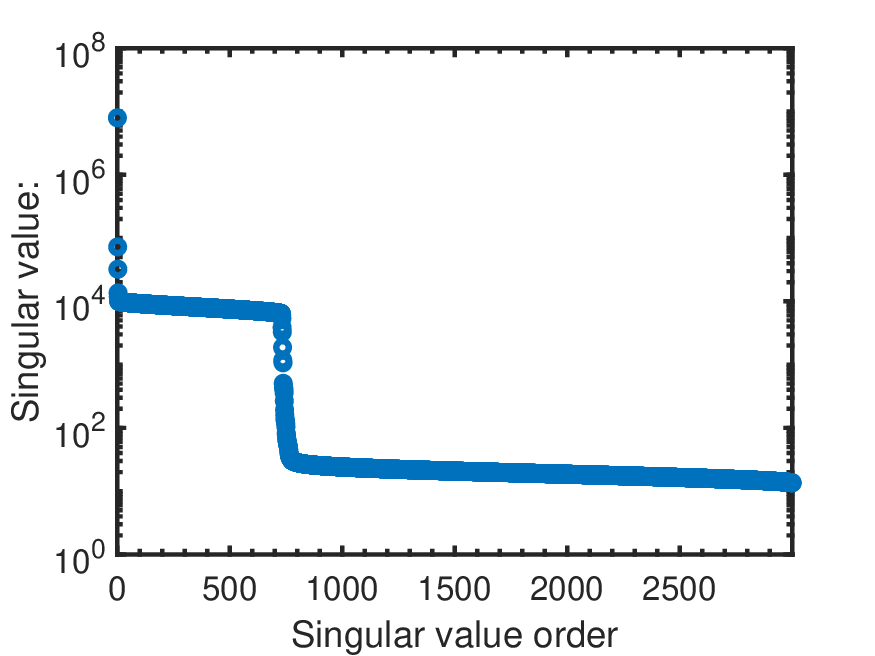}
		\caption{Singular values}
	 \label{fig:SingularValueLorenz}
	\end{subfigure}
	\hfill
	\begin{subfigure}[b]{0.32\textwidth}
		\centering
		\includegraphics[width=\linewidth]{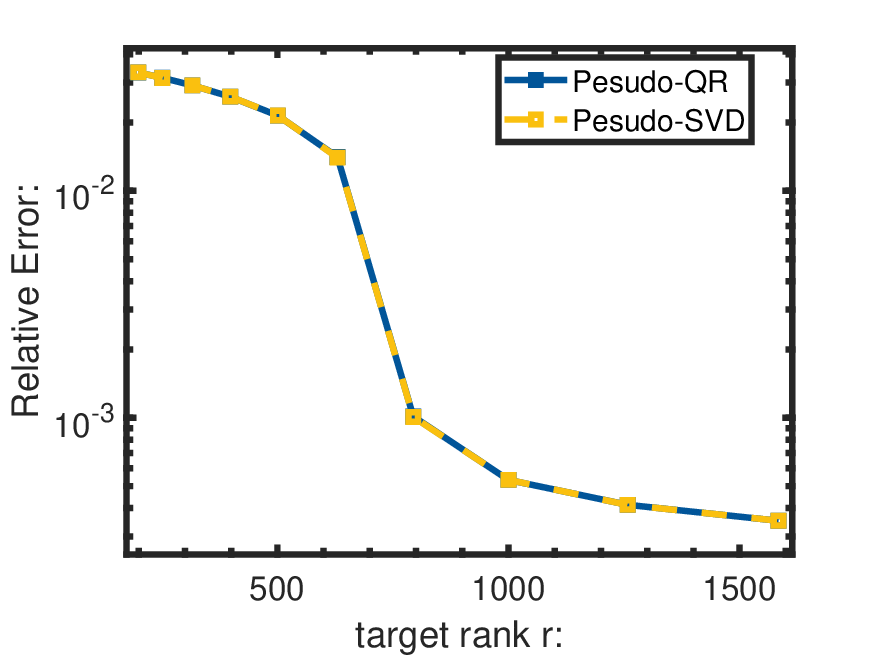}
		\caption{Relative Frobenius}
		\label{fig:LorenzRelativeError}
	\end{subfigure}
	\hfill
	\begin{subfigure}[b]{0.32\textwidth}
		\centering
		\includegraphics[width=\linewidth]{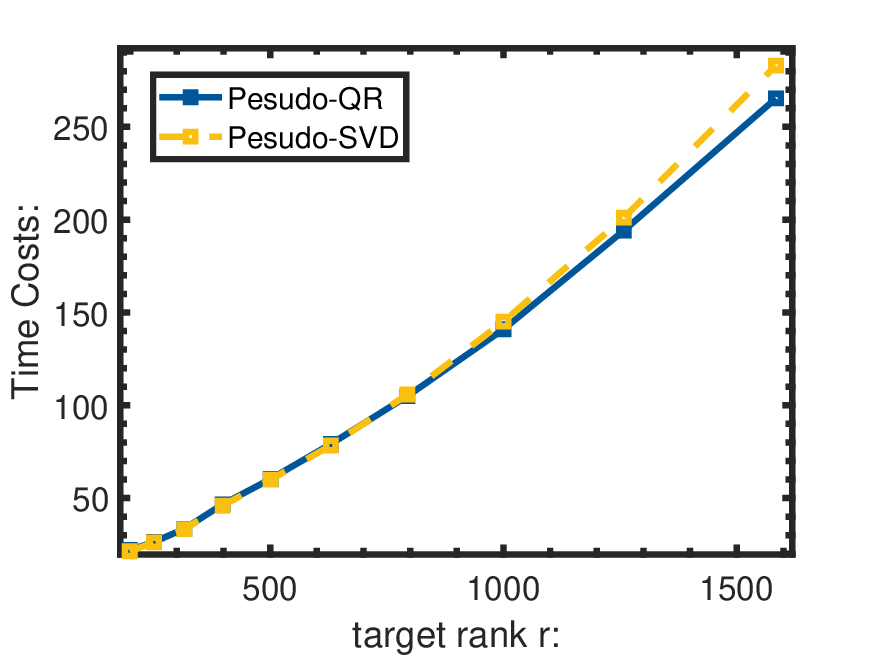}
		\caption{Running Time}
		\label{fig:LorenzTime}
	\end{subfigure}
	\caption{Relative Errors and Running Time of Lorenz system with different target rank $k$}
	\label{fig:Lorenz}
\end{figure}

\subsection{Image compression}
\begin{example}
	We compressed a large-scale color image of high-resolution map, which has $31365\times 27125$ pixels and encompasses a diverse range of environments, including clusters of buildings, factories, fields, and bodies of water. Considering the limitation of memory capabilities of local machines, streaming data processing can significantly reduce the consumption of time and space. We have performed linear updates on the sketches similar to \cite[Algorithm 2]{Practical_Sketching_Algorithms_Tropp}, thus avoiding revisiting the original data and additional transport flops. Through linear updating, we have obtained two sketches whose sizes are $31365\times s$ and $l\times 27125$ respectively. One-pass algorithm with pseudoQR or pseudo-SVD was then used to generate a low-rank approximation.

\end{example}

\begin{figure}
	\centering
	\begin{subfigure}[b]{0.4\textwidth}
		\includegraphics[width=\linewidth]{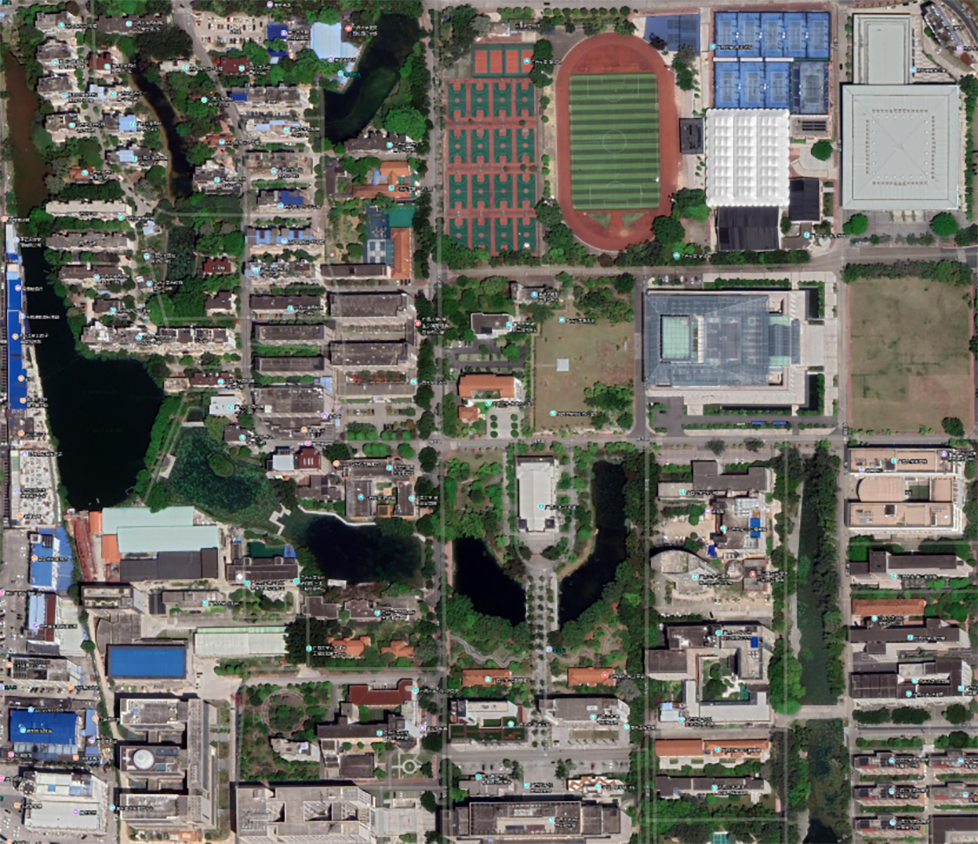}
		\caption{\small Original structures in Guangxi Univ.}
		\label{fig:SchoolOrigin}
	\end{subfigure}
	\hfill
	% \begin{subfigure}[b]{0.45\textwidth}
	% 	\centering
	% 	\includegraphics[width=\linewidth]{factory.png}
	% 	\caption{factories}
	% 	\label{fig:factory}
	% \end{subfigure}
	% \hfill
	\begin{subfigure}[b]{0.4\textwidth}
		\centering
		\includegraphics[width=\linewidth]{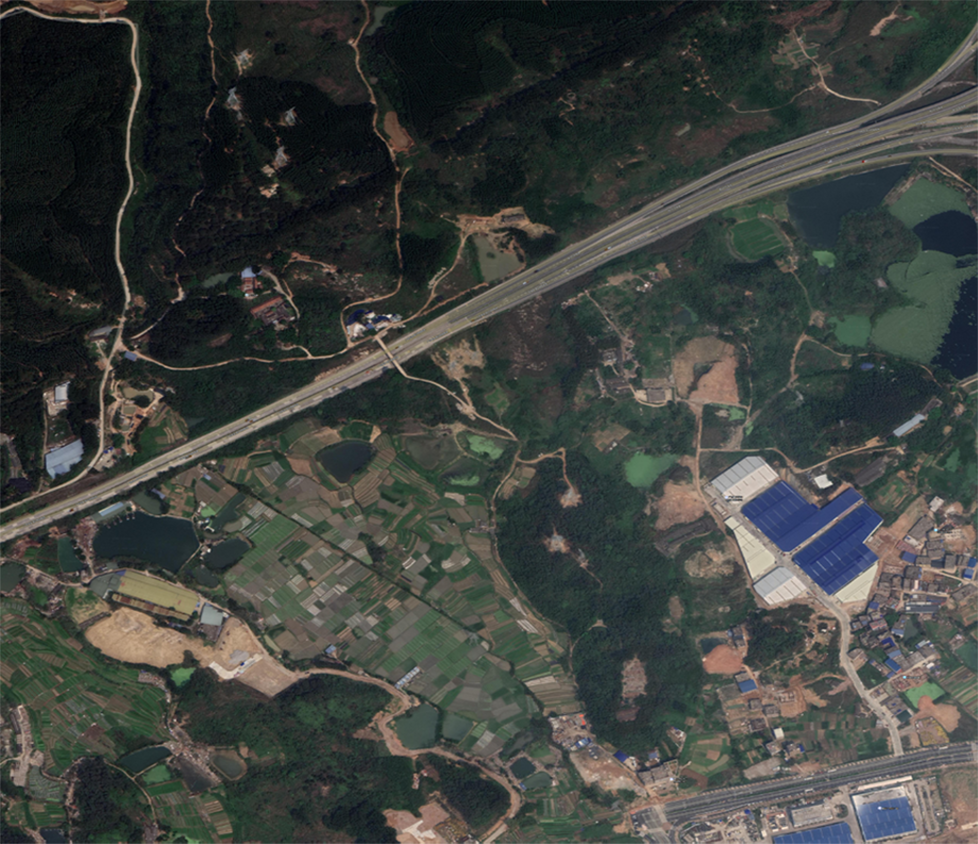}
		\caption{\small Original fields}
		\label{fig:fieldOrigin}
	\end{subfigure}
	\hfill
	\begin{subfigure}[b]{0.4\textwidth}
		\includegraphics[width=\linewidth]{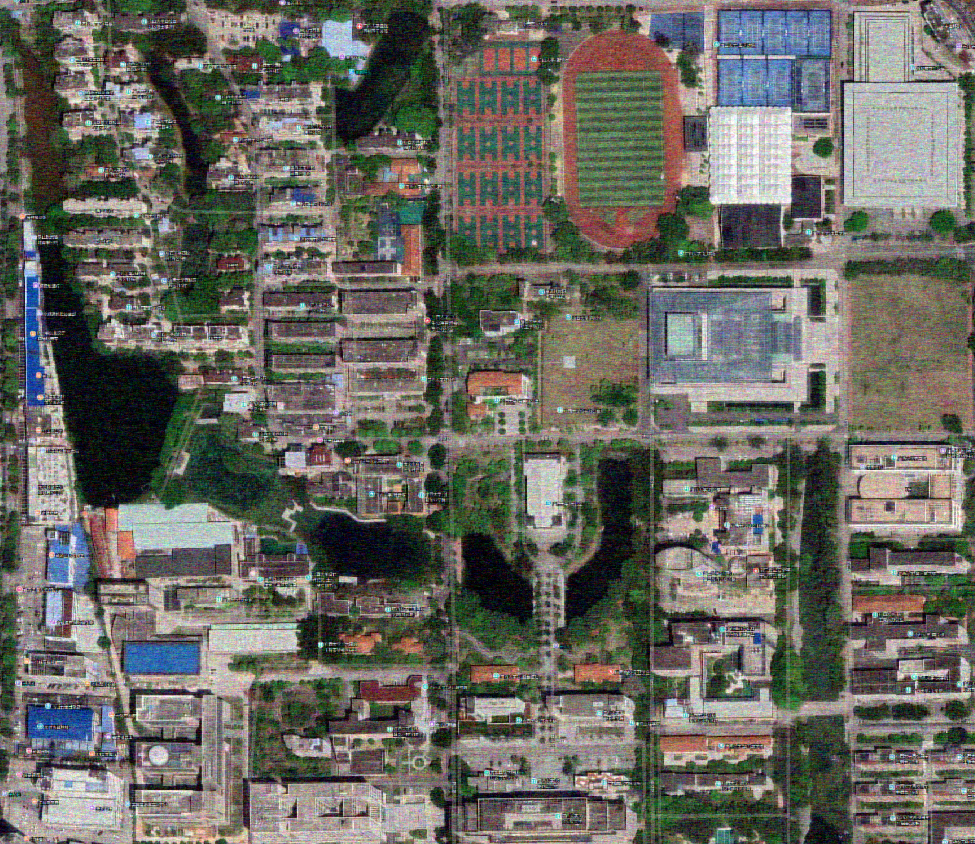}
		\caption{\small Comressed structures  in Guangxi Univ.}
		\label{fig:School}
	\end{subfigure}
	\hfill
	\begin{subfigure}[b]{0.4\textwidth}
		\centering
		\includegraphics[width=\linewidth]{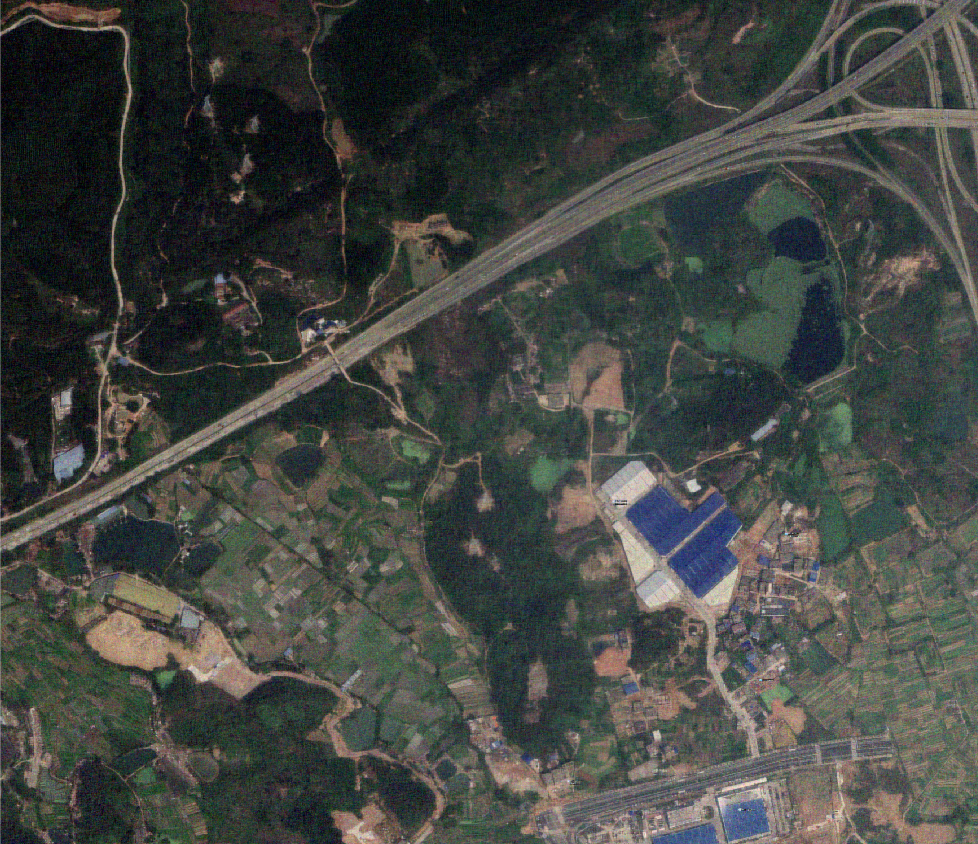}
		\caption{\small Compressed fields}
		\label{fig:field}
	\end{subfigure}
	% \hfill
	% \begin{subfigure}[b]{0.45\textwidth}
	% 	\centering
	% 	\includegraphics[width=\linewidth]{waters.png}
	% 	\caption{Waters}
	% 	\label{fig:waters}
	% \end{subfigure}
	\caption{\tiny Top two images are two representative areas selected from the orginal iamge \emph{XixiangtangMap} which has $31365\times 27125$ pixels. Down two images are corresponding parts of the approximation with target rank $r=3900$. }
	\label{fig:ApproximationImage}
\end{figure}

\begin{table}[]
	\centering
	\begin{mytabular1}{ccccc}
		\hline
		$r$    & rangefinder & Relative Error & PSNR    & Cost Time \\ \hline
		$2000$  & pseudo-QR    & $0.2613$         & $18.0186$ & $372.62$    \\
		     & pseudo-SVD   & $0.2614$         & $18.0149$ & $400.71$    \\ \hline
		$2900$ & pseudo-QR    & $0.2252$         & $19.3095$ & $764.04$    \\
		     & pseudo-SVD   & $0.2253$         & $19.3066$ & $871.75$    \\ \hline
		$3900$ & pseudo-QR    & $0.1966$         & $20.4905$ & $1433.66$   \\
		     & pseudo-SVD   & $0.1966$         & $20.4885$ & $1675.73$   \\ \hline
	\end{mytabular1}
	\caption{\small Experiment   of image compression, where $r=2000,2900,3900,s=r+50,l=2s$.}\label{table:ImageCompression}
\end{table}

Table \ref{table:ImageCompression} indicates that, as target rank varies from $2000$ to $3900$, the relative error of the low-rank approximation gradually decreases, and the PSNR (Peak Signal-to-Noise Ratio) and running time progressively increase. Comparing between pseudo-QR and pseudo-SVD, they have similar performance with the one-pass  algorithm. However, in most cases, the time consumed by pseudo-QR is   less than that by pseudo-SVD. In image compression, pseudo-QR may be a more suitable rangefinder.

Fig. \ref{fig:ApproximationImage} more intuitively reflects the approximation quality of the algorithm. We select two representative area including structures and fields in the original image; see Fig. \ref{fig:SchoolOrigin} and \ref{fig:fieldOrigin}. The former contains more architectural details, while the latter is more open. We can see that after approximating with target rank $3900$, the main elements of the map are still distinguishable, and the color channels have been well preserved when compression ratio reaches to $71.4\%$.

	\section{Conclusions} \label{sec:conclusions}
 \color{black}
Existing quaternion rangefinders based on quaternion orthogonalization  may    be inefficient for large-scale problems. Based on the strategy of trading accuracy or space for speed, this work presented two practical rangefinders, which may not be orthonormal but still well-conditioned. The proposed rangefinders were then incorporated into the   the one-pass algorithm proposed by Tropp et al. \cite{Practical_Sketching_Algorithms_Tropp} for  low-rank approximation to quaternion matrices. Throughout the whole algorithm, heavy quaternion computations has been transformed to   QR, SVD, and solving linear equations in the complex arithmetic, such that mature scientific computing libraries or advanced algorithms can be employed to accelerate the computations. Theoretically, the probabilistic error bound was established for both quaternion Gaussian and sub-Gaussian test matrices; in particular, it was demonstrated that the       error is proportional to the rangefinder's condition number. 
% In addition, we established a deviation bound for the extreme singular values of a quaternion sub-Gaussian matrix, giving theoretical support of using a sub-Gaussian test matrix. 
The efficiency and effectiveness of the algorithm were verified on  large-scale    experiments such as scientifc data compression and color image compression,  which consist of quaternion matrices of dimension $>20000$.   \color{black} 

\color{black}Our experiments also demonstrate that the devised rangefinder pseudo-SVD can accelerate   randomized QSVD \cite{RandomizedQSVD}. Other quaternion randomized algorithms \cite{renRandomizedQuaternionQLP2022,liRandomizedBlockKrylov2023,liu2024FixedprecisionRandomized,xu2024RandomizedQuaternion} may also benefit from pseduo-SVD. Besides, it is also helpful to study how to incorporate a non-orthonormal rangefinder such as pseudo-QR into these algorithms, and it is interesting to devise more practical rangefinders, as that mentioned in Sect. \ref{sect:remarks_rangefinder}. \color{black}

% The theoretical results in Section \ref{subsec:ErrorAnalysis} indicate  that one can  design  more practical non-orthonormal yet well-conditioned rangefinders. For example, any  QB decomposition  (such as Pseudo-SVD) can be combined with the subspace embedding techniques \cite{martinsson2020RandomizedNumerical,nakatsukasa2022FastAccurate}  to yield well-conditioned rangefinders with high probability.    In any case, a fast quaternion QB decomposition is still needed.   
 
%	In this paper, we propose a fast quaternion rangefinder framework and present two implementations: pseudoQR and pseudoSVD. Then we apply the fast quaternion rangefinder to a one-pass randomized quaternion algorithm for low-rank approximation and provide a theoretical analysis of the accuracy loss of the approximation error and the deviation in the space of left singular vectors. Besides, average Frobenius and spectral error bounds of fixed-rank approximation under Gaussian quaternion embedding are provided to assist users in determining the parameters of sketching size. Furthermore, we provide probability bounds for derivatives under sub-Gaussian randomized embeddings to ensure the flexibility of other prevalent randomized embeddings. Numerical experiments demonstrate that our algorithms work efficiently with less storage costs. Finally, we test two practical numerical examples, including image dimensionality reduction and online compression of scientific simulation data, to verify the effectiveness of our fast quaternion low-rank approximation algorithm.
	
\bibliographystyle{plain}
\bibliography{sn-bibliography}

\appendix
\small

\section{Proofs in the Pseudo-QR subsection}
 
	\begin{proof}[Proof of Proposition \ref{thm:CondReductionSpeed}]
		Let $\qmat{H} = \qmat{U}\Sigma\qmat{V}^*$ be a compact QSVD of $\qmat{H}$
		  where $\Sigma=\texttt{diag}(\sigma_1,\ldots,\sigma_s)$ with $\sigma_i$ arranged in a decreasing order. Denote %$\qmat{H}_{new}$ as the updated $\qmat{H}$ in the LHS of \eqref{eq:CondNumReductionIter}, and 
		  $$f(x):= (1-\epsilon)x+\epsilon x^{-1}.$$
		   Then
		  \eqref{eq:CondNumReductionIter} shows that $\qmat{H}_{new} $ has singular values $f(\sigma_i) = (1-\epsilon)\sigma_i+\epsilon {\sigma_i}^{-1}$, $i=1,\ldots,s$.  
		  We first consider the upper bound of $f$ on $[\sigma_s,1]$ and $[1,\sigma_1]$, respectively. 

		For any    $\epsilon$ and $x\in [1,\sigma_1]$, 
		$$f(x)=(1-\epsilon)x+\epsilon {x}^{-1}\leq(1-\epsilon)x+\epsilon x\leq x,$$
		 meaning that $f(x)\leq \sigma_1$ on $[1,\sigma_1]$. %Thus it is sufficient to consider $x<1$ below.

		%when there exists $k$ such that $M>\sigma_k>1$, we have the following inequality:
		% \begin{align*}
		% 	(1-\epsilon)\sigma_k+\epsilon\frac{1}{\sigma_k}\leq(1-\epsilon)\sigma_k+\epsilon\sigma_k\leq \sigma_k.
		% \end{align*}
	%	where the identity holds if and only if $\sigma_k=1$. 
	%	This property shows that it is sufficient to discuss the cases when $\sigma_i<1$. 
		%In second step,
		Next, we consider $x\in [\sigma_s,1]$. Observe that $f(x)$ is convex when $x>0$, which achieves the maximal value on the boundary. It follows from $\epsilon\geq \sigma_s$ that $f(\sigma_s) \geq (1-\epsilon)\sigma_s + \sigma_s \sigma_s^{-1} \geq 1 = f(1)$, i.e., $f$ is upper bounded by $f(\sigma_s)=\left(1-\epsilon\right)\sigma_s+\epsilon/\sigma_s$ on $[\sigma_s,1]$. Furthermore, as $\sigma_s\leq 1$ and $\epsilon\leq \delta\sigma_s $, 
		$$f(x)\leq \left(1-\epsilon\right)\sigma_s+\epsilon/\sigma_s \leq (1-\delta \sigma_s )\sigma_s+\delta\sigma_s/\sigma_s =\delta+\sigma_s-\delta\sigma_s^2<\delta + 1/(4\delta).$$

		The analysis above together with $\sigma_1\leq \sqrt{2}$ in Proposition \ref{thm:spectralNormProofOfpseudoQR} shows that the largest singular value of $\qmat{H}_{new}$ cannot exceed $\max \{\sqrt{2},\delta + 1/(4\delta)\}$.
		
		On the other hand, 		it follows from the convexity of  $f(x)$ that when $x>0$,  $ f(x)\geq 2\sqrt{\epsilon(1-\epsilon)}$. The range of $\epsilon$ implies $2\sqrt{\epsilon(1-\epsilon)}\geq  2\sqrt{\sigma_s(1-\sigma_s) }$. Thus $f(x)\geq 2\sqrt{\sigma_s(1-\sigma_s) }$.
		  %when $x=\sqrt{\epsilon(1-\epsilon)^{-1}}$ 

		  %If we set $\epsilon=\delta \sigma_n$, the largest singular value of new $H$ has an upper bound $\delta+\sigma_n-\delta\sigma_n^2$.
		 
		We also need a lower bound on $\sigma_s $. When $\kappa(H)>\max \{2\sqrt{2}\delta,2\delta^2+1/2\} = 2\delta \max \{\sqrt{2},\delta + 1/(4\delta)\} \geq 2\delta \sigma_1$, we   obtain that $\sigma_s\leq 1/\left(2\delta\right)$. 

Comparing the two upper bounds on $f$  obtained  previously,			 if $\sigma_1\geq \delta+1/(4\delta)$,
		 \begin{align} 
			\kappa(\qmat{H}_{new})  %&\leq %\frac{\sigma_1}{2\sqrt{\epsilon\left(1-\epsilon\right)}}\\
			&\leq \frac{\sigma_1}{2\sqrt{\sigma_s\left(1-\sigma_s\right)}};
		 \end{align}
		 then 
		 \begin{align}
			\frac{\kappa(\qmat{H}_{new}) }{\sqrt{\kappaq{H}}}\leq\frac{\sqrt{\sigma_1}}{2\sqrt{\left(1-\sigma_s\right)}}\leq\sqrt{\sigma_1}/\sqrt{2} \leq \sqrt{\max \{\sqrt{2},\delta + 1/(4\delta)\}}/\sqrt{2}  \leq 1,
		 \end{align}
		 where the second inequality follows from $\sigma_s\leq 1/\left(2\delta\right)<1/2$, and the last one comes from the range of $\delta$.
		 Similarly, if $\sigma_1<\delta+1/(4\delta)$,
		 \begin{align}
			\kappa(\qmat{H}_{new})  
			&\leq \frac{\delta+1/(4\delta)}{2\sqrt{\sigma_s\left(1-\sigma_s\right)}};
		 \end{align} then 
		 \begin{align}
			\frac{\kappa(\qmat{H}_{new}) }{\sqrt{\kappaq{H}}}\leq\frac{\delta+1/(4\delta)}{2\sqrt{\sigma_1\left(1-\sigma_s\right)}}\leq\frac{\delta+1/(4\delta)}{2\sqrt{1-1/(2\delta)}}<1,
		 \end{align}
	where the second inequality follows from $\sigma_1\geq 1$ and $\sigma_s\leq 1/\left(2\delta\right)$, while the last one comes from that $\frac{\delta+1/(4\delta)}{2\sqrt{1-1/(2\delta)}}$	is non-decreasing on $[1, \sqrt{7}/2]$. The result follows.

		% In third step, combined the two steps above, when $\kappa(H)>2\delta^2+1/2$, the smallest singular value $\sigma_n\leq 1/(2\delta)$. Then the condition numbers of sequence $H_{t}$ satisfy:
		% \begin{align*}
		% 	\kappa(H_{new})&\leq\frac{\delta+1/4\delta}{2}\frac{1}{\sqrt{\delta\sigma_n\left(1-\delta\sigma_n\right)}}\\
		% 	&\leq\frac{1}{2}\frac{1+1/(4\delta^2)}{\sqrt{\sigma_n\left(1/\delta-\sigma_n\right)}}\\
		% 	&\leq\frac{(1+1/(4\delta^2))\kappa(H)}{2\sqrt{\kappa(H)/\delta-1}}
		% \end{align*} 
		% which yields
		% \begin{align}
		% 	\frac{\kappa(H_{new})}{\kappa(H)^{/2}}&\leq\frac{1+1/(4\delta^2)}{2\sqrt{1/\delta-1/\kappa(H)}}\\
		% 	&\leq \frac{(1+1/(4\delta^2))\delta}{\sqrt{4\delta-2}}\in (0,1).\label{eq:QsuperlinearlyConvergence}
		% \end{align}
		% A proper parameter $\delta$ exists to guarantee (\ref{eq:QsuperlinearlyConvergence}), for example $\delta=2$.
	\end{proof}

\begin{proof}[Proof of Corollary \ref{col:correction_H_iteratively}]
Using the analysis in Proposition \ref{thm:CondReductionSpeed}, one can  use the induction method to show that if all the singular values $\sigma_i(\qmat{H}_k)$ lies in  $\bigzhongkuohao{ 2\sqrt{\sigma_s(\qmat{H}_{k-1})(1-\sigma_s(\qmat{H}_{k-1}) ) },  \max \{\sqrt{2},\delta + 1/(4\delta)\} }$, then all $\sigma_i(\qmat{H}_{k+1})$ also lie in $[2\sqrt{\sigma_s(\qmat{H}_{k})(1-\sigma_s(\qmat{H}_{k}) ) },  \max \{\sqrt{2},\delta + 1/(4\delta)\}  ]$.  Similar to the proof of   Proposition \ref{thm:CondReductionSpeed},  $\kappa(\qmat{H}_{k+1})< \sqrt{\kappa(\qmat{H}_k)}$.
\end{proof}

\begin{proof}[Proof of Proposition \ref{prop:U_tildeb_span_2t}] The idea of the proof essentially follows from \cite[Theorem 1]{wiegmann1955some}. 
	We first use the induction method to show that $\spanmat{U_{\tilde{b}}, \adjointJ \overline{ U_{\tilde{b}}  }  }=2t$. 	For $t=1$, this claim is true as $u_1 \perp\adjointJ \overline{u_1}$. Suppose now that we have found 
	 $U_{k}:=[u_{i_1},\ldots,u_{i_{k}}] \in U_b~(k<t)$  such that   $\spanmat{U_{k},\adjointJ\overline{U_{k}}}$ has dimension $2k$. Note that $U_b$ is partially orthonormal  and $\spanmat{U_b}$ has dimension $2t>2k$, and so there always exists at least a     $u_{i_{k+1}}\in U_b\setminus U_k$, such that  
	 \begin{align}\label{eq:u_k+1_notin_span_UM}
		u_{i_{k+1}}\notin \spanmat{ U_k,\adjointJ \overline{U_k}}.   
	 \end{align}
	 We first show that \eqref{eq:u_k+1_notin_span_UM} is equivalent to $\adjointJ \overline{u_{i_{k+1}}}\notin \spanmat{ U_k,\adjointJ \overline{U_k}}$. Suppose on the contrary that $\adjointJ \overline{u_{i_{k+1}}}\in \spanmat{ U_k,\adjointJ \overline{U_k}}$, which by Lemma \ref{lem:J-adjoint} is equivalent to
	 \begin{align*}
		 \overline{u_{i_{k+1}}} = \adjointJ^*\adjointJ \overline{u_{i_{k+1}}}\in \spanmat{ \adjointJ^* U_k,\adjointJ^*\adjointJ \overline{U_k}} \Leftrightarrow {u_{i_{k+1}}}    \in \spanmat{ \adjointJ \overline{U_k},  {U_k}},
	 \end{align*}
	 deducing a contradiction. 
	 
	%  Denote $U_M:=[U_k,\adjointJ \overline{U_k} ]$. Then \eqref{eq:u_k+1_notin_span_UM} is equivalent to
	%  \begin{align*}
	% 	(I_{2m}- U_MU_M^\dagger)u_{i_{k+1}}\neq 0,
	%  \end{align*}
	%  which implies
	% 	\begin{align}
	% 		\left(I_{2m}-U_M U_M^\dagger\right)\adjointJ\overline{u_{k+1}}&= \adjointJ\overline{\left(I_{2m}-U_M U_M^\dagger\right)}\adjointJ^*\adjointJ\overline{u_{k+1}}\\
	% 		&=\adjointJ\overline{\left(I_{2m}-U_M U_M^\dagger\right)u_{k+1}}\\
	% 		&=\left(I_{2m}-U_M U_M^\dagger\right)u_{k+1}>0.
	% 	\end{align}
	
	 Next, we will prove that $\spanmat{U_{k},\adjointJ\overline{U_{k}},u_{i_{k+1}},\adjointJ\overline{u_{i_{k+1}}}}$ has dimension $2(k+1)$. 
	%  Consider
		% \begin{align}\label{eq:linearIndependentUt}
		% 	\sum_{j=1}^{k}\nolimits{u_{i_{j}}\alpha_j} +u_{i_{k+1}}\alpha_{k+1}+\sum_{j=1}^{k}\nolimits{\adjointJ\overline{u_{i_{j}}}}\tilde{\alpha}_j + \adjointJ\overline{u_{i_{k+1}}}\tilde{\alpha}_{k+1}=0.
		% \end{align}
	Denote   $M:=[U_k,\adjointJ U_k]$, $P_{M}:= MM^\dagger$, and $P_{M^\perp}:= I_{2m} - MM^\dagger$ the orthogonal projection onto $\spanmat{M}$ and $\spanmat{M}^\perp$, respectively. Then $u_{i_{k+1}}$ can be divided into two part:
	\begin{align}
		 P_{M^\perp}u_{i_{k+1}}\perp\spanmat{M}\quad &\text{and}\quad P_M u_{i_{k+1}}\in\spanmat{M}.
	\end{align}
	Correspondingly, $\adjointJ \overline{u_{i_{k+1}}}$ can be divided into $P_{M^\perp}\adjointJ\overline{u_{i_{k+1}}}$ and $P_M \adjointJ \overline{u_{i_{k+1}}}$, where we also notice that 
	\begin{align*}
		\adjointJ\overline{P_{M^\perp}u_{i_{k+1}} }=\adjointJ\overline{ u_{i_{k+1}} } - \adjointJ \overline{P_M}\overline{u_{i_{k+1}}} =\adjointJ\overline{ u_{i_{k+1}} } - P_M\adjointJ\overline{u_{i_{k+1}}}=P_{M^\perp}\adjointJ\overline{u_{i_{k+1}}},
	\end{align*}
	with the second equality   from Lemma \ref{lem:j_pm_jstar}. % Similarly, 
	% \[
	% 	\adjointJ\overline{P_M u_{i_{k+1}}} = P_M \adjointJ \overline{u_{i_{k+1}}} . 
	% 	\]
	
	By Lemma \ref{lem:J-adjoint}, $P_{M^\perp}u_{i_{k+1}} \perp \adjointJ\overline{P_{M^\perp}u_{i_{k+1}} }$, and due to the above relation,   $\adjointJ\overline{P_{M^\perp}u_{i_{k+1}} }\perp M$. Thus,
	\[
		\spanmat{M,P_{M^\perp}u_{i_{k+1}},  P_{M^\perp}\adjointJ\overline{u_{i_{k+1}}} } =\spanmat{M, P_{M^\perp}u_{i_{k+1}},\adjointJ\overline{P_{M^\perp}u_{i_{k+1}} } }  	
	\]
	has dimension $2(k+1)$. Therefore, 
	\begin{align*}
	&\spanmat{U_k,\adjointJ\overline{U_k},u_{i_{k+1}},\adjointJ\overline{u_{i_{k+1}}}} \\
	=& \spanmat{M, P_{M^\perp}u_{i_{k+1}}+P_{M}u_{i_{k+1}},       P_{M}\adjointJ\overline{u_{i_{k+1}}}+P_{M^\perp}\adjointJ\overline{u_{i_{k+1}}} }	
	\end{align*}
	also has dimension $2(k+1)$. Thus the induction method shows that there exists $U_{\tilde{b}} =[u_{i_1},\ldots,u_{i_t}] \in U_b$, such that $\spanmat{U_{\tilde{b}},\adjointJ U_{\tilde{b}}}$ has dimension $2t$.

		 It is obvious that $\range{U_{\tilde{b}}}\subset\range{U_b}$. Furthermore, for any $u$ in $U_{\tilde{b}}$ corresponding to singular value $\sigma$ of $\chiQ{Y}$,  $\adjointJ\overline{u}$ is also corresponding to $\sigma$. which means that  $\range{\adjointJ\overline{U_{\tilde{b}}}}\subset\range{U_b}$. Recall that $\range{U_b}$ also has dimension $2t$; thus   $\range{U_{\tilde{b}},\adjointJ\overline{U_{\tilde{b}}}}=\range{U_b}$.
	\end{proof}

\section{Proofs in the sub-Gaussian subsection}

\subsection{Real representation}
The analysis relies on the real representation of a quaternion matrix. Similar to the complex representation, a quaternion matrix $\qmat{Q}=Q_w+Q_x\bi+Q_y\bj+Q_z\bk \in\bbQ^{m\times n}$ has a (full) real representation \cite{zhangQuaternionsMatricesQuaternions1997,RandomizedQSVD}:
\begin{align*}
	\Upsilon_{\qmat{Q}} := \begin{bmatrix}
		Q_w & -Q_x & -Q_y & -Q_z \\
		Q_x & Q_w & -Q_z & Q_y \\
		Q_y & Q_z & Q_w & -Q_x \\
		Q_z & -Q_y & Q_x & Q_w
		\end{bmatrix} \in\mathbb R^{4m\times 4n}.
\end{align*}
Similar to the complex representation $\chiQ{Q}$, $\Upsilon_{\qmat{Q}}$ also keeps several basic properties of $\qmat{Q}$, such as 
$$\qmat{C}=\qmat{A}\qmat{B} \Leftrightarrow \Upsilon_{\qmat{C}} = \Upsilon_{\qmat{A}}\Upsilon_{\qmat{B}};$$
 $$\qmat{Q} ~\text{is row/column-orthonormal if and only if}~  \Upsilon_{\qmat{Q}}~\text{is row/column-orthonormal}.$$ 

For convenience, we   use $\qmat{Q}_r \in\mathbb R^{4m\times n}$ to denote the first column block of $\Upsilon_{\qmat{Q}}$ as the compact real representation of $\qmat{Q}$, i.e., 
\begin{align*}
	\qmat{Q}_r :=\begin{bmatrix}
		Q_w\\
		Q_x\\
		Q_y\\
		Q_z
	\end{bmatrix} \in\mathbb R^{4m\times n}.%~ \text{and} ~\Upsilon_Q=[\mathcal{J}_0\mathbf{Q}_r,\mathcal{J}_1\mathbf{Q}_r,\mathcal{J}_2\mathbf{Q}_r,\mathcal{J}_3\mathbf{Q}_r]
\end{align*}
% where $\mathcal{J}_0=I_{4m}$ and 
% \begin{align*}
% 	\left.\mathcal{J}_1=\left[\begin{array}{r}-e_2^T\\e_1^T\\e_4^T\\-e_3^T\end{array}\right.\right]\otimes I_m,\quad \mathcal{J}_2=\left[\begin{array}{r}-e_3^T\\-e_4^T\\e_1^T\\e_2^T\end{array}\right]\otimes I_m,\quad \mathcal{J}_3=\left[\begin{array}{r}-e_4^T\\e_3^T\\-e_2^T\\e_1^T\end{array}\right]\otimes I_m.
% \end{align*}
% All of $\mathcal{J}_i(i=0,1,2,3)$ are orthogonal.
Spectral and Frobenius norms of a quaternion matrix $\mathbf{Q}$ can be represented by    the real representations as below:
$$\|\mathbf{Q}\|_2=\|\Upsilon_\mathbf{Q}\|_2\geq\normSpectral{\mathbf{Q}_r},\quad\|\mathbf{Q}\|_F=\frac{1}{2}\|\Upsilon_\mathbf{Q}\|_F=\normF{\mathbf{Q}_r};$$
$$ \normF{\qmat{A}\qmat{B}} = \normF{\Upsilon_{\qmat{A}}\qmat{B}_r  }.  $$

\subsection{Technical lemmas}

\begin{lemma}(\cite[Lemma 5.24]{vershynin_2012_Introduction_non-asymptotic})
	\label{lem:prod_subG}
	Let $x_1, \ldots, x_N$ be real independent centered sub-Gaussian random variables. Then $x =  [x_1, \ldots, x_N]^T$ is a centered sub-Gaussian random vector in $\mathbb{R}^n$, and
	\[
	\|x\|_{\psi_2} \leq C \max_{1\leq i \leq N} \|x_i\|_{\psi_2},
	\]
	where $C$ is an absolute constant.
\end{lemma}

\begin{proposition}(Hoeffding-type inequality, \cite[Proposition 5.10]{vershynin_2012_Introduction_non-asymptotic})
	\label{prop:hoeffding}
	Let $x_1, \ldots, x_N$ be real independent centered sub-Gaussian random variables, and let $K = \max_i \|x_i\|_{\psi_2}$. 	
	Then for every $a = [a_1, \ldots, a_N]^T \in \mathbb{R}^N$ and every $t \geq 0$, we have
\[
\mathbb{P}\left\{ \left| \sum_{i=1}^N a_i x_i \right| \geq t \right\} \leq e \cdot \exp \left( - \frac{ct^2}{K^2 \|a\|^2} \right),
\]
where $c > 0$ is an absolute constant.
\end{proposition}

\begin{definition}(sub-exponential random variable, \cite{vershynin_2012_Introduction_non-asymptotic})
	A real random variable $x$ is called a sub-exponential random variable if satisfying
	\begin{align*}
		\mathbb{P}\left\{\left|x\right|>t\right\}\leq \exp\left(1-t/K_1\right)
	\end{align*}
	for all $t>0$. And the sub-exponential norm of $x$, denoted $\|x\|_{\psi_1}$, is defined as:
	\begin{align*}
		\|x\|_{\psi_1}=\sup_{p\geq1}p^{-1}(\mathbb{E}|x|^p)^{1/p}.
	\end{align*}
\end{definition}

	\begin{lemma}{(\cite[Corollary 5.17]{vershynin_2012_Introduction_non-asymptotic})}\label{lem:SubExpRVSum}
		Let $x_1,\ldots,x_N$ be real independent centered sub-exponential random real variables, and let $K=\max_{i}\|x_i\|_{\psi_1}$. Then, for every $\varepsilon\geq0$, we have:
		\begin{align*}
			\mathbb{P}\left\{\left|\sum_{i=1}^Nx_i\right|\geq\varepsilon N\right\}\leq2\exp\left[-c\min\left(\frac{\varepsilon^2}{K^2},\frac{\varepsilon}{K}\right)N\right],
		\end{align*}
		where $c>0$ is an absolulte constant. 
	\end{lemma}

	\begin{lemma}{(\cite[Lemma 5.36]{vershynin_2012_Introduction_non-asymptotic})}\label{lem:Approximate isometries}
		Consider a real matrix $M$ that satisfies
		\begin{align*}
			\|M^TM-I\|\leq\max(\delta,\delta^2)
		\end{align*}
		for some $\delta>0$. Then 
		\begin{align}\label{eq:s_min_max}
			1-\delta\leq s_{\min}(M)\leq s_{\max}(M)\leq 1+\delta.
		\end{align}
		Conversely, if $M$ satisfies (\ref{eq:s_min_max}) for some $\delta>0$ then $\|M^TM-I\|\leq3\max(\delta,\delta^2)$.
	\end{lemma}

	\begin{lemma}\label{lem:rvsSumIneq}
		Let $x$, $y$ be two  real random variables; we have:
		\begin{align*}
			\mathbb{P}\{x+y>2\varepsilon\}\leq\mathbb{P}\{x>\varepsilon\}+\mathbb{P}\{y>\varepsilon\}
		\end{align*}
		
	\end{lemma}
	\begin{proof}
		Using formula of total probability, 
		\begin{align*}
			\mathbb{P}\{x+y>2\varepsilon\} &=\mathbb{P}\{x+y>2\varepsilon|y>\varepsilon\}\mathbb{P}\{y>\varepsilon\}+\mathbb{P}\{x+y>2\varepsilon|y\leq\varepsilon\}\mathbb{P}\{y\leq\varepsilon\} \\
			&\leq \mathbb{P}\{y>\varepsilon\}+\mathbb{P}\{x+y>2\varepsilon|y\leq\varepsilon\} \\
			&\leq \mathbb{P}\{y>\varepsilon\}+\mathbb{P}\{x>\varepsilon\}. \\
		\end{align*}
	\end{proof}
	
	\begin{lemma}\label{lem:rvsMultiplneq}
		Let $x$, $y$ be positive random variables; we have:
		\begin{align*}
			\mathbb{P}\{xy>\varepsilon^2\}\leq\mathbb{P}\{x>\varepsilon\}+\mathbb{P}\{y>\varepsilon\}
		\end{align*}
	\end{lemma}
	\begin{proof}
		This is by setting  $x_1=\log (x)$, $y_1=\log (y)$ and using   Lemma \ref{lem:rvsSumIneq}.
		%  gives
		% $$\mathbb{P}\{X_1+Y_1>2\ln\varepsilon\}\leq\mathbb{P}\{X_1>\ln\varepsilon\}+\mathbb{P}\{Y_1>\ln\varepsilon\}.$$    
	\end{proof}

	% \begin{theorem}\label{thm:ErrorAnalysisOfSubGaussianOfThreeSketch}
		%     Assume that the sketch size parameters satisfy the $k\ll s\ll l\ll min\{m,n\}$. Draw random test matrices $\bdOmega\in\mathbb{Q}^{s\times n}$, $\Gamma\in\mathbb{Q}^{s\times m}$, $\Phi\in\mathbb{Q}^{l\times m}$ and $\bdPsi\in\mathbb{Q}^{l\times n}$, whose rows are independent sub-gaussian isotropic random quaternion vectros. $\hat{A}$ is generated by algorithm \ref{alg:SimplestThreeSketchLowRankApprox}, then we have the following estiamte for the Frobenius error
		%     \begin{align}
			%         \normF{\hat{A}-A}\leq&\left(1+\frac{2\sqrt{l}+2C_K\sqrt{\max\{m,n\}-s}+t}{2\sqrt{l}-2C_K\sqrt{s}-t}\right)\left(1+
			%         \frac{2\sqrt{s}+2C_K\sqrt{n-k}+t}{2\sqrt{s}-2C_K\sqrt{k}-t}\right.\\&\qquad+\left.\frac{2\sqrt{s}+2C_K\sqrt{m-k}+t}{2\sqrt{s}-2C_K\sqrt{k}-t}\right)\sum_{j>k}{\sigma_j}
			%     \end{align}
		%     with probability at least $1-4\exp\left(-\frac{c_K t^2}{K^4}\right)+\exp\left(-\frac{2 c_K t^2}{K^4}\right)$.
		% \end{theorem}

	\subsection{Proof of Proposition \ref{prop:AVOmega}}
	\begin{proof}[Proof of Proposition \ref{prop:AVOmega}]
	Let $\Upsilon_{\qmat{A}} \in\mathbb R^{4N\times 4m} $ be the real representation of $\qmat{A}$. Let the full SVD of $\Upsilon_{\qmat{A}}$ be $\Upsilon_{\qmat{A}}=U_{\Upsilon_{\qmat{A}}}\Sigma  V_{\Upsilon_{\qmat{A}}}^T$ with $U_{\Upsilon_{\qmat{A}}}\in\mathbb R^{4N\times 4N}, \Sigma\in\mathbb R^{4N\times 4m}, V_{\Upsilon_{\qmat{A}}}\in \mathbb R^{4m\times 4m}$.    Denote $\hat{{V}}=V_{\Upsilon_{\qmat{A}}}^T\Upsilon_{\qmat{V}}$. Since $\qmat{V}$ is row-orthonormal, so is $\Upsilon_{\qmat{V}}$;   then $\hat{{V}}$ is also  row-orthonormal. We have:
		\begin{align}
			\normFSquare{\qmat{A} \qmat{V}\bdOmega}=\normFSquare{\Upsilon_{\qmat{A}} \Upsilon_{\qmat{V}}\bdOmega_r}
			=\normFSquare{U_{\Upsilon_{\qmat{A}}}\Sigma  V_{\Upsilon_{\qmat{A}}}^*\Upsilon_{\qmat{V}}\bdOmega_r}
			=\normFSquare{\Sigma\hat{V}\bdOmega_r}.
		\end{align}
  Write 
		\begin{align}
			\normFSquare{\Sigma\hat{V}\bdOmega_r}=\sum_{1\leq i\leq \min\{ 4N,4m  \},1\leq j\leq s}{{\Sigma}}_{i,i}^2\bigxiaokuohao{{\hat{V}\bdOmega_r}}_{i,j}^2 \bigxiaokuohao{I_s}_{j,j}^2,
		\end{align}
		where   $I_s$ is the identity matrix of size $s\times s$. Then 
		\begin{align}
			\normFSquare{\Sigma\hat{V}\bdOmega_r}\leq \normFSquare{\Sigma}\max_{i,j}|(\hat{V}\bdOmega_r)_{i,j}|^2\normFSquare{I}= s\normFSquare{\Sigma}\max_{i,j}|(\hat{V}\bdOmega_r)_{i,j}|^2.
		\end{align}
		% The final inequality is 
		% because $\Upsilon_{\qmat{A}}$ is real diagonal matrix, $\Upsilon_{V}$ is real orthogonal and $\bdOmega$ is isotropic entry-independent real sub-Gaussian matrix, where each element is i.i.d and has zero mean, unit variance and subGaussian norm $K$.
		For any fixed $i,j$, 
		\begin{align}
			\bigjueduizhi{(\hat{V}\bdOmega_r)_{i,j}}=\bigjueduizhi{\sum_{k}\hat{V}_{i,k}(\bdOmega_r)_{k,j}};
		\end{align}
since each $(\bdOmega_r)_{k,j}$ is an independent centered sub-gaussian random variable 	with the same sub-Gaussian norm $K$, by Hoeﬀding-type inequality (Proposition \ref{prop:hoeffding}), 
		\begin{align}
			\mathbb{P}\left\{\left|\sum_{k}\hat{V}_{i,k}(\bdOmega_r)_{k,j}\right|>t\right\}\leq e\cdot \exp\left(-\frac{ ct^2} {K^2 \sum_{k}\hat{V}_{i,k}^2   }\right) = e\cdot \exp\left(-\frac{ ct^2} {K^2     }\right),
		\end{align}
where the equality is because $\hat{V}$ is real partially orthonormal.		Then we union bound for all $i=1,\ldots ,\min\{4N,4m  \}$, $j=1,\ldots,s$:
		\begin{align}
			\mathbb{P}\left\{\max_{i,j}\left|\sum_{k}\hat{V}_{i,k}{(\bdOmega_r)}_{k,j}\right|>t\right\}&\leq  4e\cdot s\min\{N,m  \}\exp\left( -ct^2/K^2\right) \\
			&=    4e\cdot \exp\left(\log (s\min\{N,m\}  )-ct^2/K^2\right).
		\end{align}
		Combining inequalities above and the fact that $4\normFSquare{\qmat{A}} = \normFSquare{\Upsilon_{\qmat{A}}}= \normFSquare{\Sigma}$, the following inequality holds:
		\begin{align}
			\mathbb{P}\left\{\normFSquare{\qmat{A} \qmat{V}\bdOmega}>4st^2\normFSquare{\qmat{A}}\right\} &= \mathbb{P}\left\{\max_{i,j}\left|\sum_{k}\hat{V}_{i,k}{(\bdOmega_r)}_{k,j}\right|^2>t^2\right\}\\
			&\leq
			4e\cdot \exp\left(\log (s\min\{N,m\}  )-ct^2/K^2\right),
		\end{align}
i.e., 
		\begin{align}
			\mathbb{P}\left\{\normF{\qmat{A}\qmat{V}\bdOmega}>2\sqrt{s}t\normF{\qmat{A}}\right\}\leq
			4e\cdot \exp\left(\log (s\min\{N,m\}  )-ct^2/K^2\right).
		\end{align}
	\end{proof}
 
\subsection{Proof of Theorem \ref{thm:ConclusionOfSingularValueEstimation}}
	\begin{proof}[Proof of Theorem \ref{thm:ConclusionOfSingularValueEstimation}]
		First, from \cite{RandomizedQSVD}, $\sigma_{\min}(\Upsilon_{\qmat{A}}) = \sigma_{\min}(\qmat{A})$ and $\sigma_{\max}(\Upsilon_{\qmat{A}}) = \sigma_{\max}(\qmat{A})$. Thus we focus on the estimation of $\sigma_{\min}(\Upsilon_{\qmat{A}})$ and $\sigma_{\max}(\Upsilon_{\qmat{A}})$.      The conclusion is to prove that
		\begin{align*} 
			1 - C_K\sqrt{ \frac{n}{N} }- \frac{t}{\sqrt{N}}\leq \sigma_{\min}\left( \frac{\Upsilon_{\qmat{A}}}{\sqrt{4N}} \right)\leq \sigma_{\max}\left(\frac{\Upsilon_{\qmat{A}}}{\sqrt{4N}}\right)\leq 1 + C_K\sqrt{ \frac{n}{N} }+ \frac{t}{\sqrt{N}};
		\end{align*} 
		applying Lemma \ref{lem:Approximate isometries}  to   $M:=\Upsilon_{\qmat{A}}/\sqrt{4N}$, it suffices to prove that
		\begin{align}\label{eq:AA*-I_tmp}
			\|\frac{1}{4N}\Upsilon_{\qmat{A}}^T\Upsilon_{\qmat{A}}-I\|_2\leq\max(\delta,\delta^2)=:\varepsilon\quad\mathrm{where}\quad\delta=C\sqrt{\frac {n}{N}}+\frac t{\sqrt{N}}.
		\end{align}
		We can evaluate the spectral norm $\|\frac{1}{4N}\Upsilon_{\qmat{A}}^T\Upsilon_{\qmat{A}}-I\|_2$   on a $\frac{1}{4}$-net $\mathcal{N}$ of the unit sphere $S^{4n-1} = \{x\in\mathbb R^{4n}\mid x^T x = 1  \}$:
		\begin{align}\label{eq:AA*-I}
			\left\|\frac{1}{4N}\Upsilon_{\qmat{A}}^T\Upsilon_{\qmat{A}}-I\right\|_2\leq2\max_{x\in\mathcal{N}}\left|\left\langle(\frac{1}{4N}\Upsilon_{\qmat{A}}^T\Upsilon_{\qmat{A}}-I)x,x\right\rangle\right|=2\max_{x\in\mathcal{N}}\left|\frac{1}{4N}\|\Upsilon_Ax\|_2^2-1\right|.
		\end{align}
		Write $\Upsilon_{\qmat{A}}$ as a block matrix $ \Upsilon_{\qmat{A}}= \begin{bmatrix}
			A_w & -A_x & -A_y & -A_z \\
			A_x & A_w & -A_z & A_y \\
			A_y & A_z & A_w & -A_x \\
			A_z & -A_y & A_x & A_w
			\end{bmatrix} =: \begin{bmatrix}
			{B}_1\\
			{B}_2\\
			{B}_3\\
			{B}_4
		\end{bmatrix}$ with ${B}_i\in\mathbb R^{N\times 4n}$, $i=1,2,3,4$. %; here we call ${B}_i$ the $i$-th block row of the real counterpart $\Upsilon_{\qmat{A}}$.
		%  By Lemma \ref{lem:rvsSumIneq}, we have:
		% \begin{align}
		% 	P\left\{\left|\frac{1}{4N}\normSpectral{\Upsilon_{\qmat{A}} x}-1\right|>\frac{\varepsilon}{2}\right\}
		% 	&\leq P\left\{\left|\frac{1}{4N}\sum_{i=1}^{4} \normSpectral{{\Upsilon_{\qmat{A}}}_i x}-\frac{1}{4}\right|>\frac{\varepsilon}{2}\right\}\\
		% 	&\leq\sum_{i=1}^{4}P\left\{\left|\frac{1}{4N}\normSpectral{{\Upsilon_{\qmat{A}}}_i x}-\frac{1}{4}\right|>\frac{\varepsilon}{2}\right\}
		% \end{align}
	Since each row of $\qmat{A}$  is independent and  isotropic (Def. \ref{def:sub_gaussian_q_mat}),	  each   row   of ${B}_i$ is an independent sub-Gaussian isotropic real vector.
		
		Fix any real vector $x\in S^{4n-1}$,  we will upper bound   $P\left\{\left|\frac{1}{4N}\normSpectral{{B}_i x}-\frac{1}{4}\right|>\frac{\varepsilon}{2}\right\}$ for each fix $i$. The idea comes from the proof of \cite[Thm. 5.39]{vershynin_2012_Introduction_non-asymptotic}. First denote 
		\begin{align*}
			\|{B}_i x\|^2=\sum_{j=1}^{N}{  \bigxiaokuohao{({B}_i)_j x}^2}=:\sum_{j=1}^{N}{Z_j^2},
		\end{align*}
		where $({B}_i)_j$ represents the $j$-th row of ${B}_i$. As $({B}_i)_j~,j=1,\ldots,N$ are independent rows,  $Z_j=({B}_i)_j x$ are independent sub-Gaussian random variables; one can compute $\mathbb{E}Z_j^2=1$ and 
		$$\|Z_j\|_{\psi_2}=\|({B}_i)_jx\|_{\psi_2}\leq\sup_{y\in S^{4n-1}}\|({B}_i)_jy\|_{\psi_2}=\|({B}_i)_j\|_{\psi_2}\leq \max_{1\leq k\leq N}\| \qmat{A}_k\|_{\psi_2} =  K,$$
		where the second inequality follows from the definition of the sub-Gaussian norm of the quaternion vector $\qmat{A}_k$ in Def. \ref{def:q_subG_vector_isotropc}.
		Thus $Z_j^2-1$ are independent centered sub-exponential random variables.
		Using Lemma \ref{lem:SubExpRVSum} to give:
		\begin{align*}
			\mathbb{P}\left\{\left|\frac{1}{N}\|{B}_i x\|_2^2-1\right|\geq\frac{\varepsilon}{2}\right\}& =\mathbb{P}\left\{\left|\frac1N\sum_{j=1}^N {Z_j^2}-1\right|\geq\frac\varepsilon2\right\}\leq2\exp\left[-\frac{c_1}{K^4}\min(\varepsilon^2,\varepsilon)N\right] \nonumber \\
			&=2\exp\left[-\frac{c_1}{K^4}\delta^2N\right]=2\exp\left[-\frac{c_1}{K^4}(C^2n+t^2)\right],
		\end{align*}
		where $c_1$ is an absolute constant and the last two equalities come from the definition of $\delta$ in \eqref{eq:AA*-I_tmp}. By Lemma \ref{lem:rvsSumIneq}, 
		% Observe  For a  each row  of ea $\Upsilon_{\qmat{A}}$ are independent centered sub-gaussian vectors, multiplying by a permutation matrix and changing sign of the elements does not change the distribution and parameters of the random vector. Thus, we have:
		\begin{align*}
			P\left\{\left|\frac{1}{4N}\normSpectral{\Upsilon_{\qmat{A}} x}^2-1\right|>\frac{\varepsilon}{2}\right\}&\leq\sum_{i=1}^{4}P\left\{\left|\frac{1}{4N}\normSpectral{B_i x}^2-\frac{1}{4}\right|>\frac{\varepsilon}{8}\right\}\nonumber\\
			&\leq 8\exp\left[-\frac{c_1}{K^4}(C^2n+t^2)\right].
		\end{align*}
		
		Taking the union bound over all vectors $x$ in the net $\mathcal{N}$ of cardinality $\left|\mathcal{N}\right|\leq 9^n$, we obtain:
		\begin{align*}
			\mathbb{P}\left\{\max_{x\in\mathcal{N}}\left|\frac{1}{4N}\|\Upsilon_{\qmat{A}}x\|_2^2-1\right|\geq\frac{\varepsilon}{2}\right\}\leq 9^n\cdot8\exp\left[-\frac{c_1}{K^4}(C^2n+t^2)\right]\leq\exp\left(-\frac{c_1t^2}{K^4}\right),
		\end{align*}
		where the last inequality holds when $C\geq K^2\sqrt{\frac{1}{c_1}\left(\ln 9+\frac{\ln 8}{n}\right)}$.
		
		Using \eqref{eq:AA*-I}, we have:
		\begin{align*}
			\mathbb{P}\left\{\|\frac{1}{4N}\Upsilon_{\qmat{A}}^*\Upsilon_{\qmat{A}}-I\|_2\geq\varepsilon \right\}\leq \exp\left(-\frac{c_1t^2}{K^4}\right)
		\end{align*}
		It means that at least with probability $1-\exp\left(-\frac{c_1t^2}{K^4}\right)$, we have:
		
		\begin{align*}
			1-\delta\leq \sigma_{\min}\left(\frac{1}{2\sqrt{N}}\Upsilon_{\qmat{A}}\right)\leq \sigma_{\max}\left(\frac{1}{2\sqrt{N}}\Upsilon_{\qmat{A}}\right)\leq 1+\delta
		\end{align*}
		where $\delta=C\sqrt{\frac {n}{N}}+\frac {t}{\sqrt{N}}$ in \eqref{eq:AA*-I_tmp}. As noted at the beginning of the proof, this completes the proof of the theorem.
	\end{proof}
	\begin{remark}
		
		The right-hand side of inequality (\ref{eq:ConclusionOfSingularValueEstimation}) still holds when $N\leq 4n$. However, the left-hand side may yield a trivial result. 
		
		If $\qmat{A}$ has independent quaternion isotropic columns, the largest singular value can be estimated by using conjugate transposition $\qmat{A}^*$.   
	\end{remark}
	% \begin{remark}
	% 	Most commonly used sub-Gaussian distributions have relatively small sub-Gaussian norms. For example, the sub-Gaussian norm of Gaussian r.v.s is close to 1, while those of Bernoulli r.v. Rademacher r.v. and uniform distributed on $[-1,1]$. In practical, we can usually find proper constants due to special test matrices.
	% \end{remark}

	\subsection{Proof of Theorem \ref{thm:ErrorAnalysisOfSubGaussian}}

	The following lemma is   useful for proving Theorem \ref{thm:ErrorAnalysisOfSubGaussian}.
	\begin{lemma}\label{lem:psi_omega_has_isotropic_rows}
		Let $\bdPsi\in\mathbb{Q}^{l\times m}$ be a quaternion centered sub-Gaussian matrix (Def. \ref{def:sub_gaussian_q_mat}) with entries of $\bdPsi_r$ all having the same sub-Gaussian norm $K$;   $\qmat{Q}\in\mathbb{Q}^{m\times s}$ is   column-orthnormal. Then $\bdPhi=\bdPsi \qmat{Q}$ has independent rows, each of which is quaternion sub-Gaussian isotropic and has sub-Gaussian norm $CK$ for some constant $C>0$  (c.f. Def. \ref{def:q_subG_vector_isotropc}). 
	\end{lemma}
	
	\begin{proof}
		Let $\bdPsi_i,\bdPhi_i$ respectively be the $i$-th row of $\bdPsi$ and $\bdPhi$; then $\bdPhi_i = \bdPsi_i\qmat{Q},~i=1,\ldots,l$ are independent due to the independence of entries of  $\bdPsi$. 
		
		For notational convenience, we next use row vectors $\qmat{v}\in\bbQ^{1\times s}, \qmat{w}\in\bbQ^{1\times m}$ to respectively represent $\bdPhi_i$ and $\bdPsi_i$, i.e., $\qmat{v}=\qmat{w}\qmat{Q}$. We still denote $\qmat{v}_r = [v_w,v_x,v_y,v_z]\in\mathbb R^{1\times 4s}$; via real representaion, $\qmat{v}=\qmat{w}\qmat{Q}$ is equivalent to $\qmat{v}_r = \qmat{w}_r\Upsilon_{\qmat{Q}}$. We show that  $\qmat{v}$ is a quaternion sub-Gaussian and isotropic vector. %According to Def. \ref{def:sub_gaussian_q_mat},
		  For any real column vector $x\in\mathbb R^{4s}$, let $y = \Upsilon_{\qmat{Q}}x \in \mathbb R^{4m}$. Then
		\[
		\qmat{v}_r x = \qmat{w}_r\Upsilon_{\qmat{Q}}x = \qmat{w}_r y.
		\]
		Since $\qmat{w}_r$ is a   real row vector of independent  sub-Gaussian entries with zero mean, $\qmat{w}_ry$ is a linear combination of the entries of $\qmat{w}_r$, and so $\qmat{w}_r y$, namely, $\qmat{v}_r x$, is still a sub-Gaussian random variable. 
		To show that $\qmat{v}$ is  isotropc, again by Def. \ref{def:q_subG_vector_isotropc}, we compute  
		\begin{align}
			\mathbb{E}_{\qmat{v}}\left(\qmat{v}_r^T\qmat{v}_r\right)=\mathbb{E}_{\qmat{w}}\left(\Upsilon_{\qmat{Q}}^T\qmat{w}_r^T\qmat{w}_r\Upsilon_{\qmat{Q}}\right)=\Upsilon_{\qmat{Q}}^T\mathbb{E}_{\qmat{w}}\left(\qmat{w}_r^T\qmat{w}_r\right)\Upsilon_{\qmat{Q}}.
		\end{align}
		By assumption of $\bdPsi$, every entries of $\qmat{w}_r$ have zero mean and unit variance and are independent. Thus
		% Due to definition of sub-Gaussian matrix $\bdPsi$ with independent entries which has zero mean and unit variance, the diagonal entries of $\mathbb{E}_{\qmat{w}}\left(\qmat{w}_r^*\qmat{w}_r\right)$ are all equal to one because each cofficent of $\qmat{w}_r$ has unit variance and  the offdiagonal entries are all equal to $0$ because cofficient of $\qmat{w}_r$ are independent. Then we have 
		\begin{align}
			\mathbb{E}_{\qmat{w}}\left(\qmat{w}_r^T\qmat{w}_r\right)
			=I_{4m}\quad \text{and}\quad  \mathbb{E}_{\qmat{v}}\left(\qmat{v}_r^T\qmat{v}_r\right)=\Upsilon_{\qmat{\qmat{Q}}}^T\Upsilon_{\qmat{\qmat{Q}}}=I_{4s},
		\end{align}
	where the last equality is 	because $\Upsilon_{\qmat{Q}}^{4m\times 4s}$ is also column-orthonormal (by the assumption on $\qmat{Q}$). Thus $\qmat{v}$ is isotropc. 
	%	Thus, we can deduce that $v$ is quaternion isotropic. 

	Last, we estiamte the sub-Gaussian norm of the row $\qmat{v}$: $\|\qmat{v}\|_{\psi_2}$. We have
	\begin{align*}
		\|\qmat{v}\|_{\psi_2} &= \|\qmat{v}_r\|_{\psi_2} = \sup_{x^Tx=1,x\in \mathbb R^{4s}}\|\qmat{v}_r x\|_{\psi_2} \\
		& = \sup_{x^Tx=1,x\in \mathbb R^{4s}}\|\qmat{w}_r\Upsilon_{\qmat{Q}}x\|_{\psi_2}\\
		& \leq  \sup_{y^Ty=1,y\in \mathbb R^{4m}}\|\qmat{w}_ry\|_{\psi_2} = \|\qmat{w}_r\|_{\psi_2}, %= \|\qmat{w}\|_{\psi_2}
	\end{align*}
	where the inequality comes from that $y^Ty = x^T\Upsilon_{\qmat{Q}}^T \Upsilon_{\qmat{Q}} x = x^Tx  =1$, and that $y$ is in an $s$-dimensional subspace of $\mathbb R^{4m}$. Finally, since every entries of $\qmat{w}_r$ has the same sub-Gaussian norm  $K$, by Lemma \ref{lem:prod_subG}, $\|\qmat{w}_r\|_{\psi_2} \leq C\max_{j\leq 4m}\| (\qmat{w}_r)_j\|_{\psi_2} = CK$ for some absolulte constant $C$. Thus $\|\qmat{v}\|_{\psi_2} \leq CK$. This completes the proof. \end{proof}

	Now we prove   Theorem \ref{thm:ErrorAnalysisOfSubGaussian} based on Theorems \ref{thm:errorAnalysisOfTwoSketch} and   \ref{thm:ConclusionOfSingularValueEstimation}.
	\begin{proof}[Proof of Theorem \ref{thm:ErrorAnalysisOfSubGaussian}]
		Recall that \eqref{eq:thm_QB_error_for_Guassian}  of Theorem \ref{thm:errorAnalysisOfTwoSketch} implies 
		\begin{align}\label{eq:proof:prob_subG:1}
			&\|\qmat{H}\qmat{X}-\qmat{A}\|_F \leq \|\qmat{A}-\qmat{Q}\qmat{Q}^*\qmat{A}\|_F + \|\bdPsi_2^\dagger \bdPsi_1\left(\qmat{Q}_{\bot}^*\qmat{A}\right)\|_F,
		\end{align}
		where $\bdPsi_2=\bdPsi\qmat{Q}\in\mathbb{Q}^{l\times s}$, $\bdPsi_1=\bdPsi\qmat{Q}_\bot\in\mathbb{Q}^{l\times \left(m-s\right)}$.		On the other hand, \eqref{eq:lem:A-QQ^*A:2} of	Lemma  \ref{lem:A-QQ^*A} implies
		\begin{align}\label{eq:proof:prob_subG:2}
			 \normF{\qmat{A}-\qmat{Q}\qmat{Q}^*\qmat{A}} \leq\normF{\Sigma_2}+\normF{\Sigma_2\bdOmega_2 \bdOmega_1^\dagger},
		\end{align}
		where $\bdOmega_1=\qmat{V}^*_1 \bdOmega\in\mathbb{Q}^{r\times s}$, $\bdOmega_2=\qmat{V}^*_2 \bdOmega\in\mathbb{Q}^{(n-r)\times s}$. Recall that   $\bdPsi$ and $\bdOmega$ are quaternion matrices;   each entry of $\bdPsi_r$ and $\bdOmega_r$    is  independent centered   sub-Gaussian random variable with unit variance, all having  the same sub-Gaussian norm $K$. Some deviation bounds will be given first.

		By the property of the spectral norm, we have following estimations:
		\begin{align}
			\|\bdPsi_2^\dagger \bdPsi_1\left(\qmat{Q}_{\bot}^*\qmat{A}\right)\|_F&\leq \|\bdPsi_2^\dagger\|_2 \|\bdPsi_1\left(\qmat{Q}_{\bot}^*\qmat{A}\right)\|_F=\|\bdPsi_2^\dagger\|_2 \|\bdPsi\qmat{Q}_{\bot}\left(\qmat{Q}_{\bot}^*\qmat{A}\right)\|_F \label{eq:proof:prob_subG:3}\\
			\normF{\Sigma_2\bdOmega_2 \bdOmega_1^\dagger}&\leq\normSpectral{\bdOmega_1^\dagger}\normF{\Sigma_2\bdOmega_2}=\normSpectral{\bdOmega_1^\dagger}\normF{\Sigma_2\qmat{V}^*_2 \bdOmega} \label{eq:proof:prob_subG:3-2}
		\end{align}

	We first bound $\normSpectral{\bdPsi_2^\dagger}$ and $\normSpectral{\bdOmega_1^\dagger}$ using Theorem \ref{thm:ConclusionOfSingularValueEstimation}. 	By definition, $\normSpectral{\bdPsi_2^\dagger}=1/ \sigma_{\min}\left(\bdPsi_2\right)$ and $\normSpectral{\bdOmega_1^\dagger}=1/\sigma_{\min}\left(\bdOmega_1\right)=1/\sigma_{\min}\left(\bdOmega_1^*\right)$. 
	By Lemma \ref{lem:psi_omega_has_isotropic_rows}, 
	$\bdPsi_2 \in \bbQ^{l\times s}$ ($s<l$) and
	$\bdOmega_1^*\in\bbQ^{s\times r}$ ($r<s$) 
	all has independent rows, which are all isotropic having sub-Gaussian norm $CK$ for some constant $C>0$. 
		Thus, by   Theorem \ref{thm:ConclusionOfSingularValueEstimation}, we have:
		\begin{align}
			\mathbb{P}\left\{\normSpectral{\bdPsi_2^\dagger}>\frac{1}{2\sqrt{l}-2C_K\sqrt{s}-2t}\right\}&\leq\exp\left(-\frac{c_2t^2}{K^4}\right);\label{eq:proof:prob_subG:4}\\
			\mathbb{P}\left\{\normSpectral{\bdOmega_1^\dagger}>\frac{1}{2\sqrt{s}-2C_K\sqrt{r}-2t}\right\}&\leq\exp\left(-\frac{c_2t^2}{K^4}\right) \label{eq:proof:prob_subG:4-2}
		\end{align}
		for some absolulte constant $c_2>0$ and $C_K$ only dependes on the sub-Gaussian norm $CK$. 

We next bound $\|\bdPsi\qmat{Q}_{\bot}\left(\qmat{Q}_{\bot}^*\qmat{A}\right)\|_F$ and 		$\normF{\Sigma_2\qmat{V}^*_2 \bdOmega}$ by Proposition \ref{prop:AVOmega}. 
		First
		conjugate transpose such that $\|\bdPsi\qmat{Q}_{\bot}\left(\qmat{Q}_{\bot}^*\qmat{A}\right)\|_F =   \normF{ \bigxiaokuohao{\qmat{Q}_{\bot}^*\qmat{A}}^* \qmat{Q}_{\bot}^* \bdPsi^* }$; now $\bigxiaokuohao{\qmat{Q}_{\bot}^*\qmat{A}}^*\in \bbQ^{n\times (m-s)}$, $\qmat{Q}_{\bot}^* \in \bbQ^{(m-s)\times m}$, and $\bdPsi^*\in\bbQ^{m\times l}$. Since  $\qmat{Q}_{\bot}^* \in \bbQ^{(m-s)\times m}$ is row-orthonormal and $\bdPsi^*$ is a quaternion sub-Gaussian matrix with every entries of $\bdPsi^*_r$ having the same sub-Gaussian norm $K$, applying Proposition \ref{prop:AVOmega} with $$t_1 := K\sqrt{\frac{2\log\bigxiaokuohao{ \min\{ m-s,n \} }+ \log(4l) +1  }{c_1}   },$$
		where $c_1>0$ is the absolulte constant in Proposition \ref{prop:AVOmega}, 
		we have
		\begin{align*}%\label{eq:proof:prob_subG:5}
			&\mathbb{P}\left\{\normF{ \bigxiaokuohao{\qmat{Q}_{\bot}^*\qmat{A}}^* \qmat{Q}_{\bot}^* \bdPsi^* } > 2 \sqrt{l} t_1 \normF{\qmat{Q}_{\bot}^*\qmat{A}}\right\}	\\
					 \leq &  4e\cdot \exp\left(\log (l\min\{m-s,n\}  )-c_1 t_1^2/K^2\right) \\
			 = &4e\cdot \exp\bigxiaokuohao{ \log (l\min\{m-s,n\}  )-  2\log(\min\{m-s,n  \}) -\log(4l) - 1   } \\
			 =& \frac{1}{ \min\{m-s,n  \}  }.
		\end{align*}
	Note that	  %\label{eq:P*A=QQ*A-A}
	\begin{align} \label{eq:proof:prob_subG:6}
			\|\qmat{Q}_\bot^*\qmat{A}\|_F=\|\qmat{Q}_\bot \qmat{Q}_\bot^*\qmat{A}\|_F=\|\qmat{A}-\qmat{Q}\qmat{Q}^*\qmat{A}\|_F
	;\end{align}
	thus
	\begin{align}\label{eq:proof:prob_subG:7}
		\mathbb{P}\left\{\bdPsi\qmat{Q}_{\bot}\left(\qmat{Q}_{\bot}^*\qmat{A}\right)\|_F > 2 \sqrt{l} t_1 \|\qmat{A}-\qmat{Q}\qmat{Q}^*\qmat{A}\|_F\right\} \leq  \frac{1}{ \min\{m-s,n  \}  }.
	\end{align}	
		Similarly, for $\normF{\Sigma_2\qmat{V}^*_2 \bdOmega}$,  $\Sigma_2\in\bbQ^{(\min\{m,n\}-r)\times (\min\{m,n\}-r)  }$, $\qmat{V}^*_2 \in\bbQ^{ (\min\{m,n\}-r) \times n  }$, $\bdOmega \in\bbQ^{n\times s}$. 	Applying Proposition \ref{prop:AVOmega} with $$t_2 := K\sqrt{ \frac{ 2\log(\min\{m,n  \}-r ) + \log(4s) + 1   }{c_1}   },$$  we have
		\begin{align}\label{eq:proof:prob_subG:8}
			&\mathbb{P}\left\{\normF{\Sigma_2\qmat{V}^*_2 \bdOmega}>2\sqrt{s} t_2\normF{\Sigma_2}\right\}\nonumber\\
			 \leq & 4e\cdot \exp\bigxiaokuohao{ \log\bigxiaokuohao{s (\min\{ m,n \}-r  )   }  -c_1t_2^2/K^2  }\nonumber\\
			 = &4e\cdot \exp\bigxiaokuohao{ \log\bigxiaokuohao{s (\min\{ m,n \}-r  )   } -   2\log(\min\{m,n  \}-r ) - \log(4s) - 1   }\nonumber\\
			 = &\frac{1}{ \min\{m,n\}-r }.
		\end{align}

Finally, using \eqref{eq:proof:prob_subG:1}, \eqref{eq:proof:prob_subG:3},  \eqref{eq:proof:prob_subG:4}, \eqref{eq:proof:prob_subG:7},  and Lemma \ref{lem:rvsMultiplneq}, we have
		\begin{align*}
		&\mathbb{P}\bigdakuohao{  \normF{\qmat{H}\qmat{X} - \qmat{A} } \leq \bigxiaokuohao{ 1+ \frac{2\sqrt{l}t_1}{ 2\sqrt{l}-2C_K\sqrt{s}-t}} \normF{ \qmat{A}-\qmat{Q}\qmat{Q}^*\qmat{A}  }  } \\
		\geq & 1-   \exp\left(-\frac{c_2t^2}{K^4}\right) - \frac{1}{ \min\{m-s,n  \}  }.
		\end{align*} 
\eqref{eq:proof:prob_subG:2}, \eqref{eq:proof:prob_subG:3-2}, \eqref{eq:proof:prob_subG:4-2}, \eqref{eq:proof:prob_subG:8} together with Lemma \ref{lem:rvsMultiplneq} gives
\begin{align*}
	\mathbb P\bigdakuohao{ \normF{ \qmat{A}-\qmat{Q}\qmat{Q}^*\qmat{A}  } \leq   \bigxiaokuohao{ 1+ \frac{ 2\sqrt{s}t_2 }{ 2\sqrt{s}-2C_K\sqrt{r}-t }   } \normFSquare{\Sigma_2 }  } \geq 1-   \exp\left(-\frac{c_2t^2}{K^4}\right) -   \frac{1}{ \min\{m,n\}-r }. 
\end{align*}		
Combining the above two inequalities and using again Lemma \ref{lem:rvsMultiplneq} finally yields that
\begin{align*}
	& \normF{\qmat{H}\qmat{X}-\qmat{A}} \\
	\leq& \bigxiaokuohao{ 1+ \frac{K\sqrt{l} \sqrt{ {2\log\bigxiaokuohao{ \min\{ m-s,n \} }+ \log(4l) +1  }    }   }{ (\sqrt{l}-C_K\sqrt{s}-t)\sqrt{c_1}  }  } \bigxiaokuohao{ 1+ \frac{ K\sqrt{s} \sqrt{  { 2\log(\min\{m,n  \}-r ) + \log(4s) + 1   }   } }{ (\sqrt{s}-C_K\sqrt{r}-t)\sqrt{c_1} }   } \normF{\Sigma_2 }    
\end{align*}
with probability at least $ 
	 1-   \exp\left(-\frac{c_2t^2}{K^4}\right) - \frac{1}{ \min\{m-s,n  \}  } - \exp\left(-\frac{c_2t^2}{K^4}\right) -   \frac{1}{ \min\{m,n\}-r }$.
		% with
		% $\normF{\qmat{Q}_{\bot}^*\qmat{A}}=\normF{\qmat{A}-\qmat{Q}\qmat{Q}^*\qmat{A}}$ in theroem \ref{thm:errorAnalysisOfTwoSketch}, the following inequalities hold
		% which implies
		% \begin{align*}
		% 	&\mathbb{P}\left\{\normF{\qmat{HX}-\qmat{A}}>\left(1+g(l,m,s,t)\right)\left(1+g(s,n,r,t)\right) \normF{\Sigma_2}\right\}\\
		% 	\leq& 2\exp\left(-\frac{2c_Kt^2}{K^4}\right)+\frac{1}{l(m-s)}+\frac{1}{s(n-r)}
		% \end{align*}
		% by lemma \ref{lem:rvsMultiplneq}, where $g(l,m,s,t)=\frac{K\sqrt{2(log(l)+log(m-s))/c_K}}{2\sqrt{l}-2C_K\sqrt{s}-t}$.
		% %where $f(p,q,r,t)=\frac{2\sqrt{p}+2C_K\sqrt{q-r}+t}{2\sqrt{p}-2C_K\sqrt{r}-t}$
	\end{proof}

	\color{black}

\end{document}